\documentclass[oneside,english]{article}
\usepackage[T1]{fontenc}
\usepackage[latin9]{inputenc}
\usepackage{float}
\usepackage{epic,eepic}
\usepackage{comment}
\usepackage{amstext,amsthm,amssymb,amsmath}
\usepackage{epsfig,subfigure,epstopdf}
\usepackage{graphicx}
\usepackage{subfigure}
\usepackage{wrapfig}
\usepackage{fullpage}
\usepackage{color}
\usepackage{multirow}
\usepackage{tabulary}
\usepackage{booktabs}
\usepackage{enumerate}
\usepackage{color}

\theoremstyle{definition}
\newtheorem{theorem}{Theorem}[section]
\newtheorem{lemma}{Lemma}[section]

\makeatletter
\numberwithin{equation}{section}
\numberwithin{figure}{section}

\makeatother

\title{Contrast-independent partially explicit time discretizations for multiscale wave problems}
\author{Eric T. Chung \footnote{Department of Mathematics, The Chinese University of Hong Kong (CUHK), Hong Kong SAR}, ~Yalchin Efendiev\footnote{Department of Mathematics, Texas A\&M University, College Station, TX 77843, USA \& North-Eastern Federal University, Yakutsk, Russia}, ~ Wing Tat Leung\footnote{Department of Mathematics, University of California, Irvine, USA}, ~ Petr N. Vabishchevich\footnote{Nuclear Safety Institute, Russian Academy of Sciences, Moscow, Russia \& North-Eastern Federal University, Yakutsk, Russia}}

\usepackage{babel}
\begin{document}
\maketitle

\section*{Abstract}

In this work, we design and investigate contrast-independent partially explicit time discretizations for wave equations in heterogeneous high-contrast media. We consider multiscale problems, where the spatial heterogeneities are at subgrid level and are not resolved. In our previous work 
[Chung, Efendiev, Leung, and Vabishchevich,
Contrast-independent partially explicit time discretizations for multiscale flow problems, arXiv:2101.04863],
we have introduced contrast-independent partially explicit time discretizations and applied to parabolic equations. The main idea of contrast-independent partially explicit time discretization is to split the spatial space into two components: contrast dependent (fast) and contrast independent (slow) spaces defined via multiscale space decomposition. Using this decomposition, our goal  is further appropriately to introduce time splitting such that the resulting scheme is stable and can guarantee contrast-independent discretization under some suitable (reasonable) conditions. In this paper, we propose contrast-independent partially explicitly scheme for wave equations. The splitting requires a careful design. We prove that the proposed splitting is unconditionally stable under some suitable conditions formulated for the second space (slow). This condition requires some type of non-contrast dependent space and is easier to satisfy in the ``slow'' space. We present numerical results and show that the proposed methods provide results similar to implicit methods with the time step that is independent of the contrast.

\section{Introduction}

Multiscale problems for wave equations have been of interest for many applications. These include wave equations in materials, underground, and so on. In many applications, material properties have multiscale nature and high contrast, expressed as the ratio between media properties. These typically include some features with very distinct properties in small portions of the domain, or small narrow strips (known as channels). These heterogeneities bring challenges to numerical simulations as one needs to resolve the scales and the contrast. Recently, many spatial multiscale methods have been introduced to resolve spatial heterogeneities. However, because of high contrast, one needs to take small time step in wave equations when using explicit methods. To avoid this difficulty, we introduce partially explicit methods.

Explicit methods are commonly used for wave equations due to the finite speed of propagation. For example, in seismology applications, staggered explicit methods (\cite{virieux1984sh,chung2015staggered}) are used, which conserve the energy. Explicit methods have advantages as they provide a fast marching in time and can easily preserve some physical quantities. On the other hand, explicit methods require a small time step which is due to the mesh size and the contrast. 
Implicit methods typically give unconditionally stable schemes, but with a higher computational cost.
It is therefore desirable to develop a practical compromise that takes advantage of both explicit and implicit methods.
In our applications, the problems are solved on a coarse grid, where the mesh size is chosen to be much larger compared to heterogeneities. The contrast can be much larger compared to the inverse of the mesh size. It is important to remove the contrast dependency in the time stepping. For this, we propose a novel splitting algorithm and analyze its stability. Because the problems are solved on a coarse grid that is much larger compared to heterogeneities, we use multiscale methods.

Multiscale methods are often used in applications to reduce the computational cost and solve the problem on a coarser grid. These approaches provide models on a coarse grid. Many multiscale methods have been developed and analyzed. Some of them formulate coarse-grid problems using effective media properties and based on homogenization, \cite{dur91,weh02,ab05,eh09}. Other approaches are based on constructing multiscale basis functions and formulating coarse-grid equations.
These include
multiscale
finite element methods \cite{eh09,hw97,jennylt03}, 
generalized multiscale finite element methods (GMsFEM) \cite{chung2016adaptiveJCP,MixedGMsFEM,WaveGMsFEM,chung2018fast,GMsFEM13}, 
constraint energy minimizing GMsFEM (CEM-GMsFEM) 
\cite{chung2018constraint, chung2018constraintmixed}, nonlocal
multi-continua (NLMC) approaches \cite{NLMC},
metric-based upscaling \cite{oz06_1}, heterogeneous multiscale method 
\cite{ee03}, localized orthogonal decomposition (LOD) 
\cite{henning2012localized}, equation-free approaches \cite{rk07,skr06}, 
and
multiscale stochastic approaches \cite{hou2017exploring, hou2019model, hou2018adaptive}.
For high-contrast problems, GMsFEM and NLMC are proposed to extract
macroscopic quantities associated with the degrees of the freedom that
the operator ``cant see''.
For this reason, for GMsFEM and related approaches
\cite{chung2018constraint}, multiple basis
functions or continua are constructed to capture the multiscale features
due to high contrast
\cite{chung2018constraintmixed, NLMC}. 
These approaches require a careful design of multiscale
dominant modes that represent high-contrast related features. 
The contrast, in addition,  introduces a stiffness
in the forward problems. When treating explicitly, one needs to take
very small time steps when the contrast is high. In this paper, we will propose an approach
that allows taking the time step to be independent of the contrast by
handling some degrees of freedom implicitly and some explicitly.

Our approaches are based on some fundamental splitting algorithms
\cite{marchuk1990splitting,VabishchevichAdditive},
which are initially designed to split various physics.
These algorithms are originally designed for multi-physics problems
to separate various physics and reduce the computational cost. 
In recent works, we have proposed approaches for temporal splitting
that uses multiscale spaces  \cite{efendiev2020splitting, efendiev_split2020}.
In  \cite{efendiev2020splitting}, a general framework is proposed where
the transition to simpler problems is carried out 
based on spatial decomposition of the solution.
In the scheme proposed in  \cite{efendiev2020splitting}, we split
both mass and stiffness matrix for parabolic equations.
In these approaches, we divide the spatial space into various components
and use these subspaces in the temporal splitting. As a result, smaller
systems are inverted in each time step, which reduces the computational
cost. 
These algorithms are implicit.
In this paper, we propose explicit-implicit algorithms using 
careful solution space decomposition.
Though our proposed approaches share some common concepts with
implicit-explicit approaches (e.g., \cite{ascher1997implicit}, \cite{li2008effectiveness,abdulle2012explicit,engquist2005heterogeneous,ariel2009multiscale}, 
they  differ from these approaches as
our goal is to use
splitting concepts and treat implicitly and explicitly certain (contrast-dependent and
contrast-independent)
parts of the solution in order to make the time step contrast 
independent.

Our approach extends our previous work on partial explicit methods for parabolic equations \cite{chung_partial_expliict21} to wave equations. This is a significant extension as we will discuss next. First, the use of explicit methods for wave equations is more popular as discussed earlier and the proposed methods can be used to remove the restrictions on the contrast for the time stepping. Secondly, as our analysis shows that one needs a careful splitting construction in order to guarantee the stability. 
As in our previous CEM-GMsFEM approaches, we select 
dominant basis functions which capture important degrees of freedom
and it is known to give contrast-independent convergence that scales with
the mesh size. We design and
introduce an additional space in the complement space
and these degrees of freedom are treated explicitly.
In typical situations, one has very few degrees of freedom
in dominant basis functions that are treated implicitly.
We propose two approaches for temporal splitting. Both approaches require a careful introduction of energy functionals that can guarantee the stability. We note that a special decomposition
is needed to remove the contrast in the time stepping, which is
shown in this paper. Proposed approaches are still coupled via mass matrix; however, we remove the coupling via stiffness matrix which include high contrast. The coupling via mass matrix can be resolved via iterative methods much easier as it does not contain the contrast.

We remark several observations. 
\begin{itemize}

\item Additional degrees of freedom (basis functions
beyond CEM-GMsFEM basis functions) is needed for dynamic problems,
in general, to handle
missing information.

\item  Our approaches
share some similarities with online methods (e.g., \cite{chung2015residual}), 
where additional
basis functions are added and iterations are performed. However, proposed approaches correct the solution in dynamical problems without iterations.

\item We note that restrictive time step scales as the coarse mesh
size (e.g., $dt=H$)  and thus much coarser.

\end{itemize}

We present several  numerical results. 
We consider different heterogeneous media and different source frequencies. 
Examples are selected where additional basis functions provide an improvement
by choosing ``singular'' source terms, which are common for wave equations.
We compare
various methods and show that the proposed methods provide an approximation that is
independent of the contrast.

The paper is organized as follows. In the next section, we present
Preliminaries. In Section 3, we present a general construction
of partially explicit methods. Section 4 is devoted to the construction
of multiscale spaces. We present numerical results in Section 5.
The conclusions are presented in Section 6.

\section{Preliminaries}

We will consider the second order wave equation in heterogeneous
domain. The problem consists of finding 
$u$ such that 
\begin{equation}
\label{eq:main1}
{\partial^2\over\partial t^2}u=\nabla\cdot(\kappa\nabla u)\;\text{in }\Omega,
\end{equation}
where $\kappa\in L^{\infty}(\Omega)$ is a high contrast heterogeneous field.
The equation (\ref{eq:main1}) is equipped with initial and boundary conditions,
$u(0,\cdot)=u_0(x)$, $u_t(0,\cdot)=u_{00}(x)$, and $u(t,x)=g(x)$ on 
$\partial \Omega$.

We can write the problem in the weak formulation. 
Find $u(t,\cdot)\in V:=H^{1}(\Omega)$
such that 
\begin{equation}
({\partial^2\over\partial t^2}u,v)=-a(u,v)\;\forall v\in V,
\label{eq:problem_weak}
\end{equation}
where 
\[
a(u,v)=\int_{\Omega}\kappa\nabla u\cdot\nabla v.
\]
We take homogeneous boundary conditions.

For semidiscretization in space, we seek an approximation
 $u_H(t) \in V_H$, where $V_H$ ($ V_H \subset V$) is a
finite dimensional space ($H$ is a spatial mesh size),
\begin{equation}\label{2.3}
 \frac{d^2}{d t^2} (u_H(t), v) + a(u_H,v) = 0 \quad \forall v \in V_H, 
 \quad 0 < t \leq T , 
\end{equation} 
\begin{equation}\label{2.4}
 u(0) = u_H^0 ,
 \quad \frac{d u}{d t} (0) = {u}_H^{00} ,  
\end{equation}
where $(u_H^0,v) = (u_0, v), \ (u_H^{00},v) = (u_{00}, v), \ \forall \in V_H$.  
We set $v = d u_H /dt$ in (\ref{2.3}) and obtain the equalities for conservation
and stability for  (\ref{2.3}), (\ref{2.4}) with respect to initial conditions:
\begin{equation}\label{2.5}
 E(t) = E(0) ,
 \quad 0 < t \leq T ,
\end{equation} 
where
\[
 E(t) = \left \| \frac{d u}{d t} (t) \right \|^2 + \|u(t)\|^2_a , 
\] 
where $ \|u\|^2_a=a(u,u)$ and  $ \|u\|^2=(u,u)$.

We take a uniform mesh with the size $\tau$ and let
$t^n = n \tau, \ n =  0, \ldots, N, \ N \tau = T$, $u_H^n = u(t^n)$.
Stability conditions for three-layer schemes with second-order accuracy with respect to
$\tau$
is well known (see, for example, \cite{SamarskiiTheory, SamarskiiMatusVabischevich2002}).

When using implicit discretization, $u_H^{n} \approx u_H(t^n)$ is sought as
\begin{equation}
\label{2.6}
\begin{split}
 \left ( \frac{u_H^{n+1}  - 2u_H^{n} + u_H^{n-1} }{\tau^2}, v \right ) + & \frac{1}{4} a(u_H^{n+1} + 2u_H^{n} + u_H^{n-1} ,v) = 0 
 \quad \forall v \in V_H, \\
 & \quad n = 1, \ldots, N-1 , 
\end{split}
\end{equation} 
with appropriate initial conditions.
We introduce new variables
\[
 s^n = \frac{u_H^{n} + u_H^{n-1}}{2} ,
 \quad  r^n = \frac{u_H^{n} - u_H^{n-1}}{\tau } ,
\] 
and from (\ref{2.6}), we get
\[
 \left ( \frac{r^{n+1} - r^n}{\tau}, v \right ) + \frac{1}{2} a(s^{n+1} + s^n, v) = 0 . 
\] 
We take
\[
 v = 2 (s^{n+1} - s^n) = \tau (r^{n+1} + r^n) ,
\] 
which gives
\[
 \|r^{n+1}\|^2 - \|r^{n}\|^2 + \|s^{n+1}\|_a^2 -  \|s^{n}\|_a^2 = 0 .
\]

Thus, we have
\begin{equation}\label{2.8}
 E^{n+1/2} = E^{n-1/2} ,
\end{equation} 
where
\[
 E^{n+1/2} =  \left \|\frac{u_H^{n+1} - u_H^{n}}{\tau } \right \|^2 + \left \|\frac{u_H^{n+1} + u_H^{n}}{2} \right \|_a^2 ,
 \quad n = 1, \ldots, N-1 .
\] 
This equality is a discrete version of  (\ref{2.5}) and guarantees 
unconditional stability. 

Stability condition for explicit method
\begin{equation}\label{2.9}
 \left ( \frac{u_H^{n+1}  - 2u_H^{n} + u_H^{n-1} }{\tau^2}, v \right ) + a(u_H^{n},v) = 0 
 \quad \forall v \in V_H, 
 \quad n = 1, \ldots, N-1 , 
\end{equation} 
is done in a similar way.
Taking into account
\[
 u_H^{n} = \frac{u_H^{n+1} + 2 u_H^n + u_H^{n-1}}{4} 
 - \frac{\tau^2}{4} \frac{u_H^{n+1} - 2 u_H^n + u_H^{n-1}}{\tau^2} ,
\] 
in new variables defined,  (\ref{2.9}) can be written as
\[
 \|r^{n+1}\|^2 - \frac{\tau^2}{4} \|r^{n+1}\|_a^2 - \|r^{n}\|^2  + \frac{\tau^2}{4}\|r^{n}\|_a^2  + \|s^{n+1}\|_a^2 -  \|s^{n}\|_a^2 = 0 .
\] 
Thus, we have the equality of the energies, where the energy is defined as
\[
 E^{n+1/2} =  \left \|\frac{u_H^{n+1} - u_H^{n}}{\tau } \right \|^2 -  \frac{\tau^2}{4} \left \|\frac{u_H^{n+1} - u_H^{n}}{\tau } \right \|_a^2 + \left \|\frac{u_H^{n+1} + u_H^{n}}{2} \right \|_a^2 ,
 \quad n = 1, \ldots, N-1 .
\] 
The quantity $E^{n+1/2}$ defines a norm if 
\begin{equation}\label{2.10}
 \|v\|^2 \geq \frac{\tau^2}{4} \|v\|_a^2 \quad \forall v \in V_H .
\end{equation} 
Because of  (\ref{2.10}), the stability of (\ref{2.9}) takes place if $\tau$ is 
sufficiently small
$\tau \leq \tau_0$.

The schemes 
(\ref{2.6}) and (\ref{2.9}) 
are special cases of three-layer scheme with weights:
\begin{equation}\label{2.11}
\begin{split}
 \left ( \frac{u_H^{n+1}  - 2u_H^{n} + u_H^{n-1} }{\tau^2}, v \right ) + & a(\sigma u_H^{n+1} + (1-2\sigma)u_H^{n} + \sigma u_H^{n-1} ,v) = 0 
 \quad \forall v \in V_H, \\
 & \quad n = 1, \ldots, N-1 .
\end{split}
\end{equation}
When $\sigma = 0.25$ we have (\ref{2.6}), and if $\sigma = 0$, we have the scheme
(\ref{2.9}).
Unconditionally stable is the scheme
(\ref{2.11}), if $\sigma \geq 0.25$,
and conditionally stable if 
$0 \leq  \sigma < 0.25$. 
We have considered limit cases $\sigma = 0.25$ and $\sigma = 0$.

\section{Partially Explicit Temporal Splitting Scheme}

\subsection{The methodology}

In this section, we will discuss a temporal splitting scheme 
We consider
$V_{H}$ can be decomposed into two subspaces $V_{H,1}$, $V_{H,2}$
namely, 
\[
V_{H}=V_{H,1} +  V_{H,2}.
\]


We seek the solution $u_H=u_{H,1}+ u_{H,2}$ and 
  $\{u_{H,1}^{n}\}_{n=1}^{N}\in V_{1,H},\;\{u_{H,2}^{n}\}_{n=1}^{N}\in V_{H,2}$,
$u_H=u_{H,1}+u_{H,2}$, such that
\begin{equation}
\begin{split}
\label{eq:split12}
(u_{H}^{n+1}-2u_{H}^{n}+u_{H}^{n-1},w) + 
\cfrac{\tau^{2}}{2}a(u_{H,1}^{n+1}+u_{H,1}^{n-1}+2u_{H,2}^{n},w)  =0\;\forall w\in V_{1,H}\\
(u_{H}^{n+1}-2u_{H}^{n}+u_{H}^{n-1},w)+ 
\tau^{2}a(\omega u_{H,1}^n +(1-\omega)(u_{H,1}^{n+1}+u_{H,1}^{n-1})/2+u_{H,2}^{n},w)  =0\;\forall w\in V_{2,H}.
\end{split}
\end{equation}
We will consider the case $\omega=1$ mostly as the case $\omega=0$ performs 
similarly and is more difficult to show stability (see Appendix \ref{sec:appendix}). The case $\omega=1$ has the following form
\begin{align}
(u_{H}^{n+1}-2u_{H}^{n}+u_{H}^{n-1},w)+\cfrac{\tau^{2}}{2}a(u_{H,1}^{n+1}+u_{H,1}^{n-1}+2u_{H,2}^{n},w) & =0\;\forall w\in V_{1,H},\label{eq:simplied_eq1}\\
(u_{H}^{n+1}-2u_{H}^{n}+u_{H}^{n-1},w)+\tau^{2}a(u_{H,1}^{n}+u_{H,2}^{n},w) & =0\;\forall w\in V_{2,H}.\label{eq:simplied_eq2}
\end{align}
This is a special case of three-layer schemes, which we will investigate for stability 
and show that one can use contrast independent time step.
We denote the discrete energy $E^{n+\frac{1}{2}}$ of $u_{H}$ is
defined as 
\begin{equation}
\begin{split}
E^{n+\frac{1}{2}}=\|u_{H}^{n+1}-u_{H}^{n}\|^{2}+\cfrac{\tau^{2}}{2}\sum_{i=1,2}\Big(\|u_{H,i}^{n+1}\|_{a}^{2}+\|u_{H,i}^{n}\|_{a}^{2}\Big)+\\
\tau^{2}a(u_{H,2}^{n+1},u_{H,1}^{n})+\tau^{2}a(u_{H,1}^{n+1},u_{H,2}^{n})-\cfrac{\tau^{2}}{2}\|u_{H,2}^{n+1}-u_{H,2}^{n}\|_{a}^{2}
\end{split}
\end{equation}
or
\[
E^{n+\frac{1}{2}}=\|u_{H}^{n+1}-u_{H}^{n}\|^{2}+\cfrac{\tau^{2}}{2}\Big(\|u_{H,1}^{n+1}+u_{H,2}^{n}\|_{a}^{2}+\|u_{H,1}^{n}+u_{H,2}^{n+1}\|_{a}^{2}\Big)-\cfrac{\tau^{2}}{2}\|u_{H,2}^{n+1}-u_{H,2}^{n}\|_{a}^{2}
\]
\begin{lemma}
For $u_{H,1}$ and $u_{H,2}$ satisfying (\ref{eq:simplied_eq1})
and (\ref{eq:simplied_eq2}), we have 
\[
E^{n+\frac{1}{2}}=E^{n-\frac{1}{2}}.
\]
\end{lemma}

\begin{proof}
We have
\begin{align}
(u_{H}^{n+1}-2u_{H}^{n}+u_{H}^{n-1},w)+\cfrac{\tau^{2}}{2}a(u_{H,1}^{n+1}+u_{H,1}^{n-1}+2u_{H,2}^{n},w) & =0\;\forall w\in V_{1,H},\label{eq:simplied_eq1-1}\\
(u_{H}^{n+1}-2u_{H}^{n}+u_{H}^{n-1},w)+\tau^{2}a(u_{H,1}^{n}+u_{H,2}^{n},w) & =0\;\forall w\in V_{2,H}.\label{eq:simplied_eq2-1}
\end{align}
We consider $w=u_{H,1}^{n+1}-u_{H,1}^{n-1}$ in the first equation
and $w=u_{H,2}^{n+1}-u_{H,2}^{n-1}$ in the second equation
and obtain the following equations
\begin{equation}
\begin{split}
(u_{H}^{n+1}-2u_{H}^{n}+u_{H}^{n-1},u_{H,1}^{n+1}-u_{H,1}^{n-1})+\cfrac{\tau^{2}}{2}a(u_{H,1}^{n+1}+u_{H,1}^{n-1}+2u_{H,2}^{n},u_{H,1}^{n+1}-u_{H,1}^{n-1}) =0\\
(u_{H}^{n+1}-2u_{H}^{n}+u_{H}^{n-1},u_{H,2}^{n+1}-u_{H,2}^{n-1})+\tau^{2}a(u_{H,1}^{n}+u_{H,2}^{n},u_{H,2}^{n+1}-u_{H,2}^{n-1})  =0
\end{split}
\end{equation}
The sum of the left hand sides  can be estimated in the following
way 
\begin{align*}
\sum_{i=1,2}(u_{H}^{n+1}-2u_{H}^{n}+u_{H}^{n-1},u_{H,i}^{n+1}-u_{H,i}^{n-1}) & =(u_{H}^{n+1}-2u_{H}^{n}+u_{H}^{n-1},u_{H}^{n+1}-u_{H}^{n-1}) = \\
 & \|u_{H}^{n+1}-u_{H}^{n}\|_{}^{2}-\|u_{H}^{n}-u_{H}^{n-1}\|_{}^{2}.
\end{align*}
Next, we will estimate the right hand side of the equations. We have
\begin{equation}
\begin{split}
\cfrac{1}{2}a(u_{H,1}^{n+1}+u_{H,1}^{n-1}+2u_{H,2}^{n},u_{H,1}^{n+1}-u_{H,1}^{n-1})=\cfrac{1}{2}a(u_{H,1}^{n+1}+u_{H,1}^{n-1},u_{H,1}^{n+1}-u_{H,1}^{n-1})+\\
a(u_{H,2}^{n},u_{H,1}^{n+1}-u_{H,1}^{n-1})
\end{split}
\end{equation}
and
\[
a(u_{H,1}^{n}+u_{H,2}^{n},u_{H,2}^{n+1}-u_{H,2}^{n-1})=a(u_{H,1}^{n},u_{H,2}^{n+1}-u_{H,2}^{n-1})+a(u_{H,2}^{n},u_{H,2}^{n+1}-u_{H,2}^{n-1})
\]
Thus, we have 
\begin{align*}
 & \cfrac{\tau^{2}}{2}a(u_{H,1}^{n+1}+u_{H,1}^{n-1}+2u_{H,2}^{n},u_{H,1}^{n+1}-u_{H,1}^{n-1})+\tau^{2}a(u_{H,1}^{n}+u_{H,2}^{n},u_{H,2}^{n+1}-u_{H,2}^{n-1}) = \\
 & \cfrac{\tau^{2}}{2}B_{1}+\tau^{2}B_{2}+\tau^{2}B_{3}
\end{align*}
where 
\begin{align*}
B_{1} & =a(u_{H,1}^{n+1}+u_{H,1}^{n-1},u_{H,1}^{n+1}-u_{H,1}^{n-1}),\\
B_{2} & =a(u_{H,2}^{n},u_{H,1}^{n+1}-u_{H,1}^{n-1})+a(u_{H,1}^{n},u_{H,2}^{n+1}-u_{H,2}^{n-1}),\\
B_{3} & =a(u_{H,2}^{n},u_{H,2}^{n+1}-u_{H,2}^{n-1}).
\end{align*}
To estimate $B_{1}$, we have 
\begin{align*}
B_{1} & =a(u_{H,1}^{n+1}+u_{H,1}^{n-1},u_{H,1}^{n+1}-u_{H,1}^{n-1}) = \\
 & \|u_{H,1}^{n+1}\|_{a}^{2}-\|u_{H,1}^{n-1}\|_{a}^{2} = \\
 & (\|u_{H,1}^{n+1}\|_{a}^{2}+\|u_{H,1}^{n}\|_{a}^{2})-(\|u_{H,1}^{n}\|_{a}^{2}+\|u_{H,1}^{n-1}\|_{a}^{2}).
\end{align*}
We next estimate $B_{2}$ and have
\begin{align*}
B_{2} & =a(u_{H,2}^{n},u_{H,1}^{n+1}-u_{H,1}^{n-1})+a(u_{H,1}^{n},u_{H,2}^{n+1}-u_{H,2}^{n-1}) = \\
 & a(u_{H,2}^{n},u_{H,1}^{n+1})+a(u_{H,1}^{n},u_{H,2}^{n+1})-a(u_{H,2}^{n},u_{H,1}^{n-1})-a(u_{H,1}^{n},u_{H,2}^{n-1}).
\end{align*}
Since $a(\cdot,\cdot)$ is symmetric, we have 
\begin{align*}
B_{2} & =\Big(a(u_{H,1}^{n+1},u_{H,2}^{n})+a(u_{H,2}^{n+1},u_{H,1}^{n})\Big)-\Big(a(u_{H,1}^{n},u_{H,2}^{n-1})+a(u_{H,2}^{n},u_{H,1}^{n-1})\Big).
\end{align*}
To estimate $B_{3}$, we have
\begin{align*}
B_{3} & =a(u_{H,2}^{n},u_{H,2}^{n+1}-u_{H,2}^{n-1})= \\
 & a(u_{H,2}^{n},u_{H,2}^{n+1})-a(u_{H,2}^{n},u_{H,2}^{n-1}) = \\
 & a(u_{H,2}^{n+1},u_{H,2}^{n})-a(u_{H,2}^{n},u_{H,2}^{n-1}).
\end{align*}
We also have
\[
a(u_{H,2}^{n+1},u_{H,2}^{n})=\cfrac{1}{2}\Big(\|u_{H,2}^{n+1}\|_{a}^{2}+\|u_{H,2}^{n}\|_{a}^{2}-\|u_{H,2}^{n+1}-u_{H,2}^{n}\|_{a}^{2}\Big)
\]
and 
\[
a(u_{H,2}^{n},u_{H,2}^{n-1})=\cfrac{1}{2}\Big(\|u_{H,2}^{n}\|_{a}^{2}+\|u_{H,2}^{n-1}\|_{a}^{2}-\|u_{H,2}^{n}-u_{H,2}^{n-1}\|_{a}^{2}\Big).
\]
Thus, we have 
\begin{equation}
\begin{split}
  \cfrac{\tau^{2}}{2}a(u_{H,1}^{n+1}+u_{H,1}^{n-1}+2u_{H,2}^{n},u_{H,1}^{n+1}-u_{H,1}^{n-1})+\tau^{2}a(u_{H,1}^{n}+u_{H,2}^{n},u_{H,2}^{n+1}-u_{H,2}^{n-1}) = \\
  \cfrac{\tau^{2}}{2}\Big(\|u_{H,1}^{n+1}\|_{a}^{2}+\|u_{H,1}^{n}\|_{a}^{2}\Big)-\cfrac{\tau^{2}}{2}\Big(\|u_{H,1}^{n}\|_{a}^{2}+\|u_{H,1}^{n-1}\|_{a}^{2}\Big) + \\
  \tau^{2}\Big(a(u_{H,1}^{n+1},u_{H,2}^{n})+a(u_{H,2}^{n+1},u_{H,1}^{n})\Big)-
\tau^{2}\Big(a(u_{H,1}^{n},u_{H,2}^{n-1})+a(u_{H,2}^{n},u_{H,1}^{n-1})\Big)  + \\
  \cfrac{\tau^{2}}{2}\Big(\|u_{H,1}^{n+1}\|_{a}^{2}+\|u_{H,1}^{n}\|_{a}^{2}-\|u_{H,2}^{n+1}-u_{H,2}^{n}\|_{a}^{2}\Big)-\\
\cfrac{\tau^{2}}{2}\Big(\|u_{H,1}^{n}\|_{a}^{2}+\|u_{H,1}^{n-1}\|_{a}^{2}-\|u_{H,2}^{n}-u_{H,2}^{n-1}\|_{a}^{2}\Big)
\end{split}
\end{equation}
By definition of $E^{n+\frac{1}{2}}$, we have 
\begin{equation}
\begin{split}
  E^{n+\frac{1}{2}} = \\
 \|u_{H}^{n+1}-u_{H}^{n}\|_{}^{2}+\sum_{i=1,2}\Big(\cfrac{\tau^{2}}{2}(\|u_{H,i}^{n+1}\|_{a}^{2}+\|u_{H,i}^{n}\|_{a}^{2})\Big)+\\
\tau^{2}a(u_{H,2}^{n+1},u_{H,1}^{n})+a(u_{H,1}^{n+1},u_{H,2}^{n})-\cfrac{\tau^{2}}{2}\|u_{H,2}^{n+1}-u_{H,2}^{n}\|_{a}^{2} = \\
  \|u_{H}^{n}-u_{H}^{n-1}\|_{L^{2}}^{2}+\sum_{i=1,2}\Big(\cfrac{\tau^{2}}{2}(\|u_{H,i}^{n}\|_{a}^{2}+\|u_{H,i}^{n-1}\|_{a}^{2})\Big)+\\
\tau^{2}a(u_{H,2}^{n},u_{H,1}^{n-1})+a(u_{H,1}^{n},u_{H,2}^{n-1})-\cfrac{\tau^{2}}{2}\|u_{H,2}^{n}-u_{H,2}^{n-1}\|_{a}^{2} = \\
 E^{n-\frac{1}{2}}.
\end{split}
\end{equation}
\end{proof}

We will discuss two cases. In the first case, we assume $V_{H,1}$ and $V_{H,2}$ are
orthogonal and in the second case, we will consider the case when they are
not orthogonal.

\subsection{Case $V_{H,1}$ and $V_{H,2}$ are orthogonal}

\begin{theorem}\label{t-1}
The partially explicit scheme (\ref{eq:simplied_eq1}) and (\ref{eq:simplied_eq2})
is stable if 
\begin{equation}\label{3.4}
 \|v_2\|^2 \geq \frac{\tau^2}{2} \|v_2\|_a^2  \quad \forall v_2 \in V_{2,H} .
\end{equation}
\end{theorem}
\begin{proof}
In this case, we can show that
\begin{equation}
\begin{split}
  E^{n+\frac{1}{2}} = 
 \|u_{H}^{n+1}-u_{H}^{n}\|_{}^{2}+\sum_{i=1,2}\Big(\cfrac{\tau^{2}}{2}(\|u_{H,i}^{n+1}\|_{a}^{2}+\|u_{H,i}^{n}\|_{a}^{2})\Big)+\\
\tau^{2}a(u_{H,2}^{n+1},u_{H,1}^{n})+a(u_{H,1}^{n+1},u_{H,2}^{n})-\cfrac{\tau^{2}}{2}\|u_{H,2}^{n+1}-u_{H,2}^{n}\|_{a}^{2} \geq 
\|u_{H,1}^{n+1}-u_{H,1}^{n}\|_{L^{2}}^{2}+\\
\sum_{i=1,2}\Big(\cfrac{\tau^{2}}{2}(\|u_{H,i}^{n+1}\|_{a}^{2}+\|u_{H,i}^{n}\|_{a}^{2})\Big)+
\tau^{2}a(u_{H,2}^{n+1},u_{H,1}^{n})+a(u_{H,1}^{n+1},u_{H,2}^{n})
\end{split}
\end{equation}
this defines a norm. In this norm, we have the stability.
\end{proof}

We also refer to Appendix \ref{sec:appendix}, where we present the proof
for the case $\omega=0$ when the spaces $V_{H,1}$ and $V_{H,2}$ are orthogonal.

\subsection{Case $V_{H,1}$ and $V_{H,2}$ are non-orthogonal}  

We define $\gamma,\gamma_{a}<1$ and $\alpha\in\mathbb{R}$ as
\[
\gamma:=\sup_{v_{1}\in V_{H,1},v_{2}\in V_{H,2}}\cfrac{(v_{1},v_{2})}{\|v_{1}\|_{}\|v_{2}\|_{}},\;\gamma_{a}:=\sup_{v_{1}\in V_{H,1},v_{2}\in V_{H,2}}\cfrac{a(v_{1},v_{2})}{\|v_{1}\|_{a}\|v_{2}\|_{a}}
\]
and 
\[
\alpha=\sup_{v_{2}\in V_{H,2}}\cfrac{\|v_{2}\|_{a}}{\|v_{2}\|_{}}.
\]

\begin{lemma}
If 
\begin{equation}
\label{eq:stab_cond}
2(1-\gamma)\alpha^{-2}\geq\tau^{2},
\end{equation}
 we have 
\[
\left(1-\gamma^{2}-\cfrac{\alpha^{2}\tau^{2}}{2}\right)\|u_{H,2}^{n+1}-u_{H,2}^{n}\|_{}^{2}+\cfrac{\tau^{2}(1-\gamma_{a})}{2}\sum_{i=1,2}\Big(\|u_{H,i}^{n+1}\|_{a}^{2}+\|u_{H,i}^{n}\|_{a}^{2}\Big)\leq E^{\frac{1}{2}}.
\]
\end{lemma}

\begin{proof}
Since we have 
\[
E^{\frac{1}{2}}=E^{n+\frac{1}{2}}\;\text{for any }n\geq0,
\]
we have 
\begin{equation}
\begin{split}
\|u_{H}^{n+1}-u_{H}^{n}\|_{}^{2}+\sum_{i=1,2}\Big(\cfrac{\tau^{2}}{2}(\|u_{H,i}^{n+1}\|_{a}^{2}+\|u_{H,i}^{n}\|_{a}^{2})\Big)+\tau^{2}a(u_{H,2}^{n+1},u_{H,1}^{n})+\\
a(u_{H,1}^{n+1},u_{H,2}^{n})-\cfrac{\tau^{2}}{2}\|u_{H,2}^{n+1}-u_{H,2}^{n}\|_{a}^{2}=E^{\frac{1}{2}}.
\end{split}
\end{equation}
We have 
\begin{align*}
\|u_{H}^{n+1}-u_{H}^{n}\|_{}^{2} & =\sum_{i=1,2}\|u_{H,i}^{n+1}-u_{H,i}^{n}\|_{}^{2}+2(u_{H,1}^{n+1}-u_{H,1}^{n},u_{H,2}^{n+1}-u_{H,2}^{n}) \geq \\
 & \sum_{i=1,2}\|u_{H,i}^{n+1}-u_{H,i}^{n}\|_{}^{2}-2\gamma\|u_{H,1}^{n+1}-u_{H,1}^{n}\|_{}\|u_{H,2}^{n+1}-u_{H,2}^{n}\|_{} \geq\\
 & (1-\gamma^{2})\|u_{H,2}^{n+1}-u_{H,2}^{n}\|_{}^{2}.
\end{align*}
If $2(1-\gamma^{2})\alpha^{-2}\geq\tau^{2}$, we have 
\[
(1-\gamma^{2})\|u_{H,2}^{n+1}-u_{H,2}^{n}\|_{}^{2}-\cfrac{\tau^{2}}{2}\|u_{H,2}^{n+1}-u_{H,2}^{n}\|_{a}^{2}\geq(1-\gamma^{2}-\cfrac{\alpha^{2}\tau^{2}}{2})\|u_{H,2}^{n+1}-u_{H,2}^{n}\|_{}^{2}.
\]
We also obtain that
\begin{align*}
 & \cfrac{\tau^{2}}{2}\sum_{i=1,2}\Big(\|u_{H,i}^{n+1}\|_{a}^{2}+\|u_{H,i}^{n}\|_{a}^{2}\Big)+\tau^{2}a(u_{H,2}^{n+1},u_{H,1}^{n})+a(u_{H,1}^{n+1},u_{H,2}^{n}) \geq\\
 & \cfrac{\tau^{2}}{2}\sum_{i=1,2}\Big(\|u_{H,i}^{n+1}\|_{a}^{2}+\|u_{H,i}^{n}\|_{a}^{2}\Big)-\gamma_{a}\tau^{2}\Big(\|u_{H,2}^{n+1}\|_{a}\|u_{H,1}^{n}\|_{a}+\|u_{H,1}^{n+1}\|_{a}\|u_{H,2}^{n}\|_{a}\Big)
\end{align*}
and 
\[
\gamma_{a}\tau^{2}\Big(\|u_{H,2}^{n+1}\|_{a}\|u_{H,1}^{n}\|_{a}+\|u_{H,1}^{n+1}\|_{a}\|u_{H,2}^{n}\|_{a}\Big)\leq\cfrac{\gamma_{a}\tau^{2}}{2}\sum_{i=1,2}\Big(\|u_{H,i}^{n+1}\|_{a}^{2}+\|u_{H,i}^{n}\|_{a}^{2}\Big).
\]
Therefore, we have 
\begin{equation}
\begin{split}
\cfrac{\tau^{2}}{2}\sum_{i=1,2}\Big(\|u_{H,i}^{n+1}\|_{a}^{2}+\|u_{H,i}^{n}\|_{a}^{2}\Big)+\tau^{2}a(u_{H,2}^{n+1},u_{H,1}^{n})+a(u_{H,1}^{n+1},u_{H,2}^{n})\geq \\
\cfrac{\tau^{2}(1-\gamma_{a})}{2}\sum_{i=1,2}\Big(\|u_{H,i}^{n+1}\|_{a}^{2}+\|u_{H,i}^{n}\|_{a}^{2}\Big)
\end{split}
\end{equation}
and obtain 
\begin{equation}
\begin{split}
 E^{\frac{1}{2}}=  \sum_{i=1,2}\Big(\|u_{H,i}^{n+1}-u_{H,i}^{n}\|_{}^{2}+\cfrac{\tau^{2}}{2}(\|u_{H,i}^{n+1}\|_{a}^{2}+\|u_{H,i}^{n}\|_{a}^{2})\Big)+\tau^{2}a(u_{H,2}^{n+1},u_{H,1}^{n})+\\
a(u_{H,1}^{n+1},u_{H,2}^{n})-\cfrac{\tau^{2}}{2}\|u_{H,2}^{n+1}-u_{H,2}^{n}\|_{a}^{2}\geq \\
 (1-\gamma^{2}-\cfrac{\alpha^{2}\tau^{2}}{2})\|u_{H,2}^{n+1}-u_{H,2}^{n}\|_{}^{2}+\cfrac{\tau^{2}(1-\gamma_{a})}{2}\sum_{i=1,2}\Big(\|u_{H,i}^{n+1}\|_{a}^{2}+\|u_{H,i}^{n}\|_{a}^{2}\Big).
\end{split}
\end{equation}

\end{proof}

\section{$V_{H,1}$ and $V_{H,2}$ constructions}

In this section, we introduce a possible way to construct
the spaces satisfying (\ref{eq:stab_cond}). 
Here, we follow our previous work \cite{chung_partial_expliict21}.
We will show that the
constrained energy minimization finite element space is a good choice
of $V_{H,1}$ since the CEM basis functions are constructed such that
they are almost orthogonal to a space $\tilde{V}$ which can be easily
defined. To obtain a $V_{H,2}$ satisfying the condition (\ref{eq:stab_cond}),
one of the possible way is using an eigenvalue problem to construct
the local basis function. Before, discussing the construction of $V_{H,2}$,
we will first introduce the CEM finite element space.
In the following, we let $V(S) = H_0^1(S)$ for a proper subset $S\subset \Omega$.

\subsection{CEM method}
\label{sec:cem}

In this section, we will discuss the CEM method for solving the problem
(\ref{eq:problem_weak}). 
We will construct the finite element space by solving
a constrained energy minimization problem. We let $\mathcal{T}_{H}$
be a coarse grid partition of $\Omega$. For each element $K_{i}\in\mathcal{T}_{H}$,
we consider a set of auxiliary basis functions $\{\psi_{j}^{(i)}\}_{j=1}^{L_{i}}\in V(K_{j})$.
We then can define a projection operator $\Pi_{K_{i}}:L^{2}(K_{i})\mapsto V_{aux}^{(i)}\subset L^{2}(K_{i})$
such that 
\[
s_{i}(\Pi_{i}u,v)=s_{i}(u,v)\;\forall v\in V_{aux}^{(i)}:=\text{span}\{\psi_{j}^{(i)}:\;1\leq j\leq L_{i}\},
\]
where 
\begin{equation}
s_{i}(u,v)=\int_{K_{i}}\tilde{\kappa}uv.\label{eq:eigenvalueproblem1}
\end{equation}
and $\tilde{\kappa}=\kappa H^{-2}$ or $\tilde{\kappa}=\kappa\sum_{i}|\nabla\chi_{i}|^{2}$
with some partition of unity ${\chi_{i}}$.

We next define a  projection operator by $\Pi:L^{2}(\Omega)\mapsto V_{aux}\subset L^{2}(\Omega)$
\[
s(\Pi u,v)=s(u,v)\;\forall v\in V_{aux}:=\sum_{i=1}^{N_{e}}V^{(i)},
\]
where $s(u,v):=\sum_{i=1}^{N_{e}}s_{i}(u|_{K_{i}},v|_{K_{i}})$.
For each auxiliary basis functions $\psi_{j}^{(i)}$, we can define
a local basis function $\phi_{j}^{(i)}\in V(K_{i}^{+})$
such that 
\begin{align*}
a(\phi_{j}^{(i)},v)+s(\mu_{j}^{(i)},v) & =0\;\forall v\in V(K_{i}^{+})\\
s(\phi_{j}^{(i)},\nu) & =s(\psi_{j}^{(i)},\nu)\;\forall\nu\in V_{aux}(K_{i}^{+})
\end{align*}
where $K_{i}^{+}$ is an oversampling domain of $K_{i}$, which is a few coarse blocks larger than $K_i$ \cite{chung2018constraint}. 
We then define the space $V_{cem}$ as 
\begin{align*}
V_{cem} & :=\text{span}\{\phi_{j}^{(i)}:\;1\leq i\leq N_{e},1\leq j\leq L_{i}\},
\end{align*}
where $N_e$ is the number of coarse elements. 
The CEM solution $u_{cem}$ is given by
\begin{align*}
({\partial^2\over\partial t^2}(u_{cem}),v) & = -a(u_{cem},v)\;\forall v\in V_{cem}.
\end{align*}

We remark that the $V_{glo}$ is $a-$orthogonal to a space $\tilde{V}:=\{v\in V:\;\Pi(v)=0\}$.
We also know that $V_{cem}$ is closed to $V_{glo}$ and therefore
it is almost orthogonal to $\tilde{V}$. Thus, we can choice $V_{cem}$
to be $V_{H,1}$ and construct a space $V_{H,2}$ in $\tilde{V}$.

\subsection{Construction of $V_{H,2}$}

We discuss two choices for the space $V_{H,2} \subset \tilde{V}$.
We will present the stability properties numerically in Section~\ref{sec:num}. 

\subsubsection{First choice}

We will define basis functions for each coarse neighborhood $\omega_i$,
which is the union of all coarse elements having the $i$-th coarse grid node. 
For each coarse neighborhood $\omega_{i}$, we consider the following
eigenvalue problem: find $(\xi_{j}^{(i)},\gamma_{j}^{(i)})\in ( V_0(\omega_i) \cap \tilde{V}) \times\mathbb{R}$,
\begin{align}
\label{eq:eigenvalueproblem_case2}
\int_{\omega_{i}}\kappa\nabla\xi_{j}^{(i)}\cdot\nabla v & = \cfrac{\gamma_{j}^{(i)}}{H^2}\int_{\omega_{i}}\xi_{j}^{(i)}v, \;\forall v\in V_0(\omega_i) \cap \tilde{V}.
\end{align}
We arrange the eigenvalues by $\gamma_1^{(i)}  \leq \gamma_2^{(i)} \leq \cdots$.
In order to obtain a reduction in error, we will select the first few $J_i$ dominant eigenfunctions corresponding to smallest eigenvalues of (\ref{eq:eigenvalueproblem_case2}). 
We define
\begin{equation*}
V_{H,2} = \text{span} \{ \xi_j^{(i)} \; | \; \forall \omega_i, \forall 1\leq j \leq J_i\}.
\end{equation*}

\subsubsection{Second choice}

The second choice of $V_{H,2}$ is based on the CEM type finite
element space. For
each coarse element $K_{i}$, we will solve an eigenvalue problem to obtain the auxiliary
basis. More precisely, we find eigenpairs $(\xi_{j}^{(i)},\gamma_{j}^{(i)})\in(V(K_{i})\cap\tilde{V})\times\mathbb{R}$ by solving
\begin{align}
\label{eq:spectralCEM2}
\int_{K_{i}}\kappa\nabla\xi_{j}^{(i)}\cdot\nabla v & =\gamma_{j}^{(i)}\int_{K_{i}}\xi_{j}^{(i)}v, \;\ \forall v\in V(K_{i})\cap\tilde{V}.
\end{align}
For each $K_i$, we choose the first few $J_i$ eigenfunctions corresponding to the smallest $J_i$ eigenvalues. 
The span of these functions form a space which is called $V_{aux,2}$.
For each auxiliary basis function $\xi_j^{(i)} \in V_{aux,2}$, we define a basis function $\zeta_{j}^{(i)} \in V(K_i^+)$ such
that $\mu_{j}^{(i)} \in V_{aux,1}$, $ \mu_{j}^{(i),2} \in V_{aux,2}$ and 
\begin{align}
a(\zeta_{j}^{(i)},v)+s(\mu_{j}^{(i),1},v)+ ( \mu_{j}^{(i),2},v) & =0, \;\forall v\in V(K_i^+), \label{eq:v2a} \\
s(\zeta_{j}^{(i)},\nu) & =0, \;\forall\nu\in V_{aux,1}, \label{eq:v2b} \\
(\zeta_{j}^{(i)},\nu) & =( \xi_{j}^{(i)},\nu), \;\forall\nu\in V_{aux,2}. \label{eq:v2c}
\end{align}
where we use the notation $V_{aux,1}$ to denote the space $V_{aux}$ defined in Section \ref{sec:cem},
and $K_i^+$ is an oversampling domain a few coarse blocks larger than $K_i$ (see \cite{chung2018constraint}).
We define $$V_{H,2}=\text{span}\{\zeta_{j}^{(i)}| \; \forall K_i, \; \forall 1 \leq j\leq J_i\}.$$

\section{Numerical Result}
\label{sec:num}

In this section, we will present representative numerical results.
We consider the following mesh and time step parameters in all examples.
\[
H=1/10,\;h=1/100,\;dt=0.006,\;T=0.05.
\]
We consider two
 medium parameter $\kappa(x)$, where one is a simpler compared to the
other (see Figure \ref{fig:kappa}). 
Both medium parameters are high contrast and multiscale. 
We choose the source term as a source distrubuted in a small
region as shown in Figure \ref{fig:kappa}). 
It, $f$, is given by 
\begin{equation}
\label{eq:source}
f(t,x)=\cfrac{2-2/f_{0}}{4h^{2}}exp(-\pi^{2}f_{0}^{2}(t-2/f_{0})^{2})f_{x}(x),
\end{equation}
where $f_0$ is a frequency. We will consider two values for
$f_0$. In numerical results, we will refer to the case
(Case 1 refers to first medium and Case 2 refers to the second
medium) and the frequency $f_0$.

\begin{figure}[H]
\centering
\includegraphics[scale=0.35]{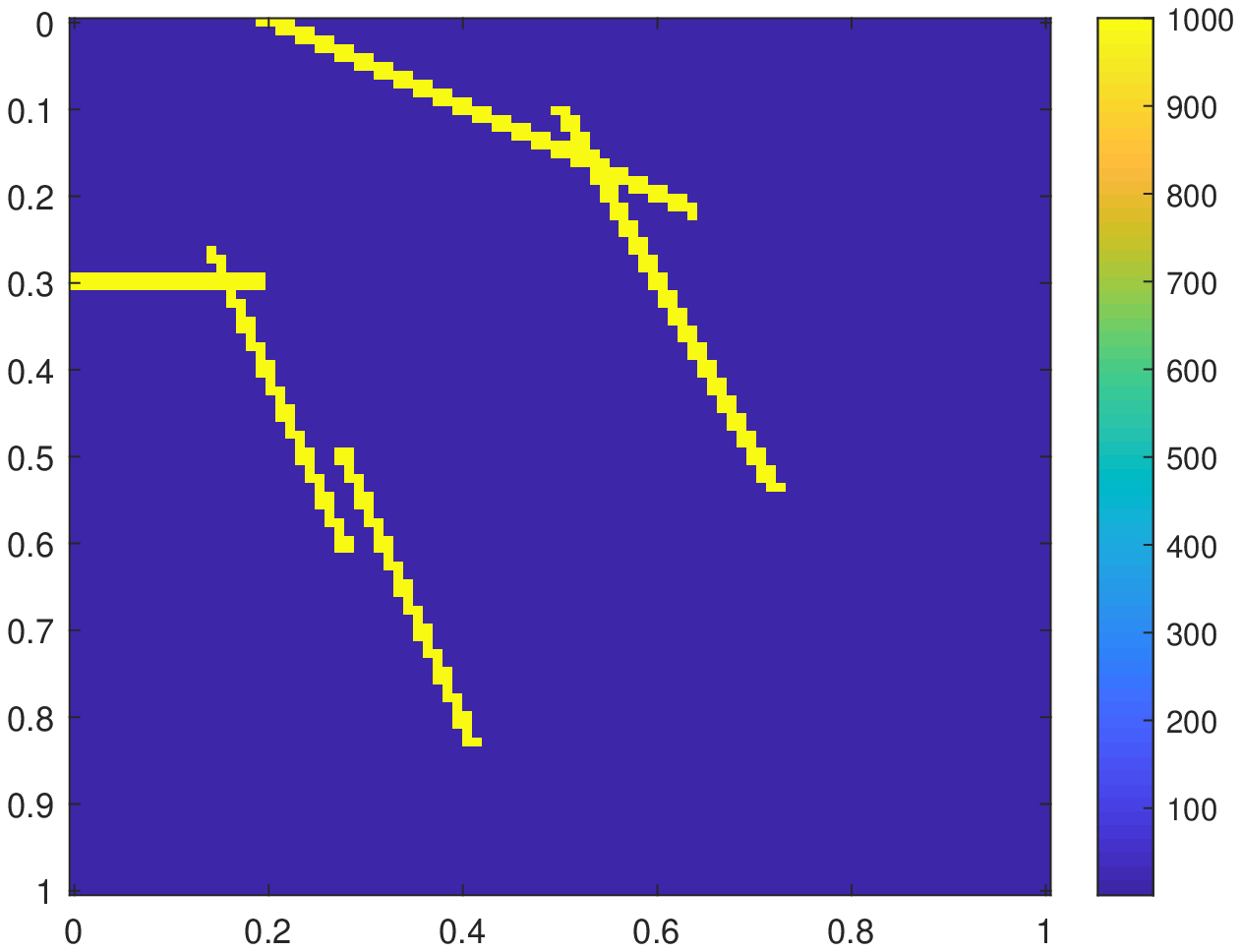} 
\includegraphics[scale=0.35]{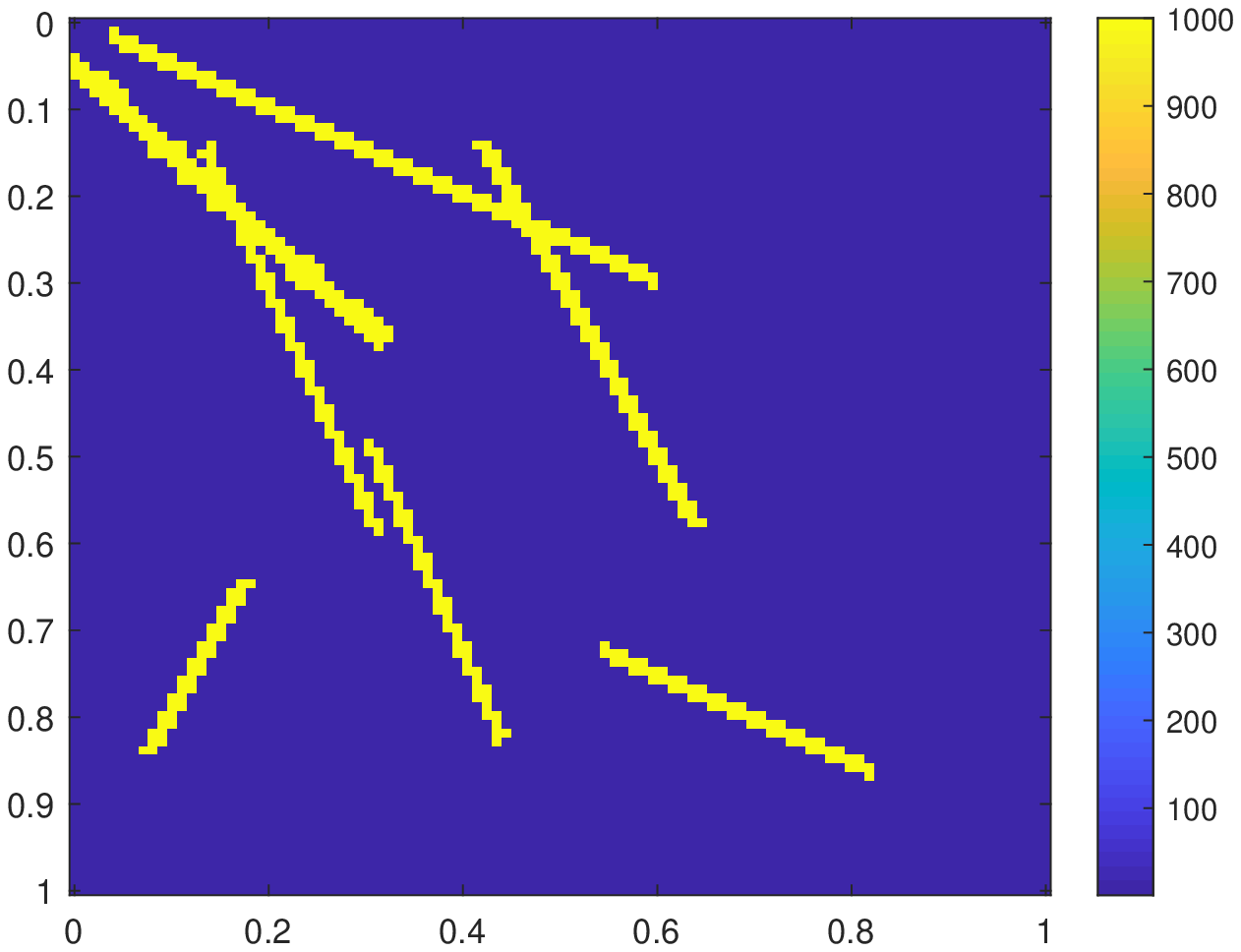} 
\includegraphics[scale=0.35]{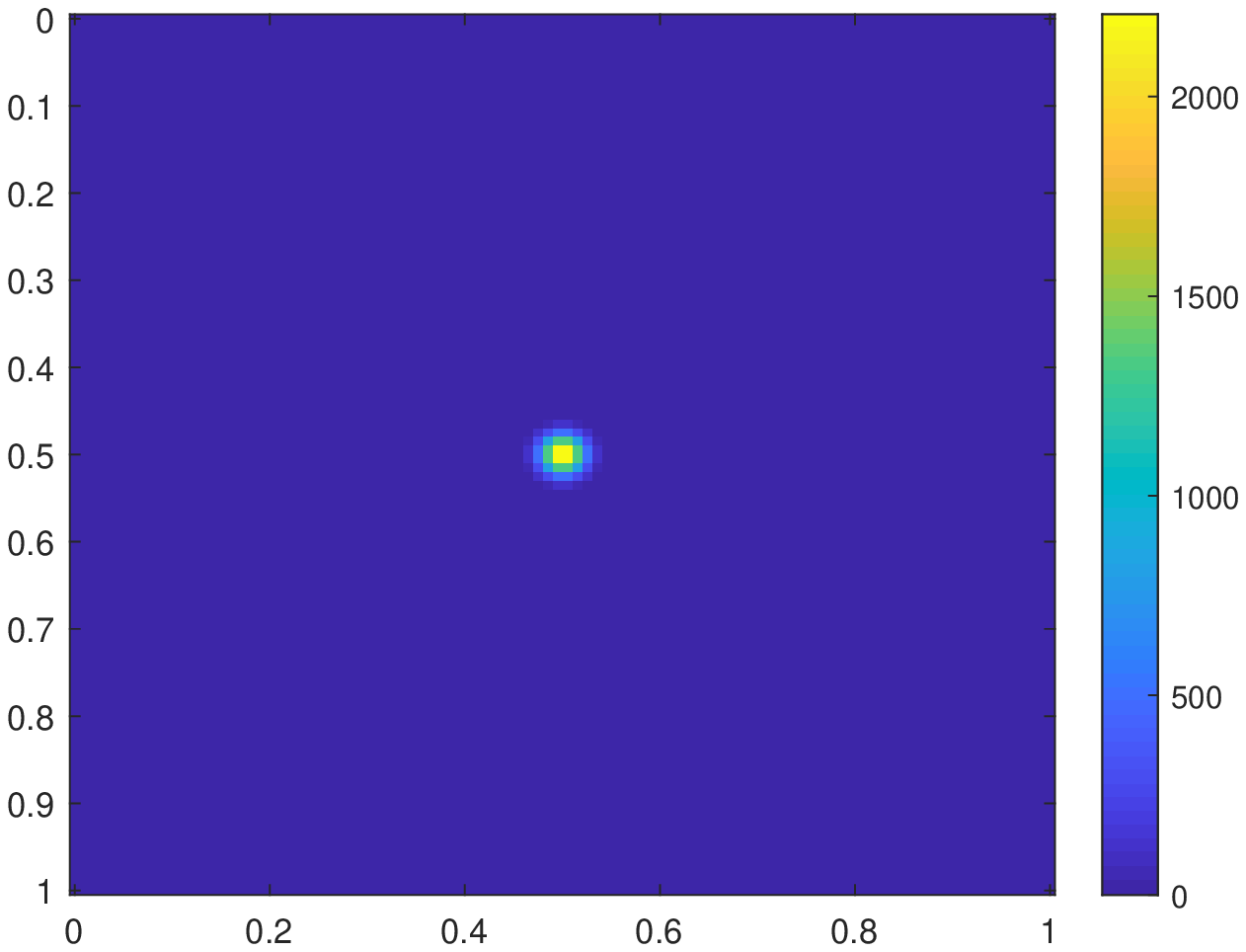}
\caption{Left: $\kappa$ for Case 1 (denote $\kappa_1(x)$); Middle: $\kappa$ for Case 2 (denote $\kappa_2(x)$); Right: $f_{x}(x)$ from (\ref{eq:source}).}
\label{fig:kappa}
\end{figure}

\subsection{Case 1 and $f_0=1/2$.}

In our first numerical example, we consider the case with the conductivity
$\kappa_1(x)$. In Figure \ref{fig:case1low}, we plot the reference 
solution computed
on the fine grid (top-left), the solution computed using CEM-GMsFEM (top-right)
(without additional basis functions), the solution computed using additional basis functions with implicit method (bottom-left), and the solution computed using additional basis functions using partially explicit method (bottom-right).
In all cases, the reference solution is computed using standard explicit
finite element
discretization with small time step (in our case, $\tau=1e-4$).
First, we note that it can be seen that additional basis functions provide 
improved results. This is because of the choice of source term since
CEM-GMsFEM solution requires more detailed information to improve the solution.
 In Figure \ref{fig:case1low}, we plot the solution at the time
$T=0.3$. In Figure \ref{fig:case1low2}, we plot the solution and its approximations (as in Figure \ref{fig:case1low}) for $T=0.6$. The errors are plotted 
in Figure \ref{fig:case1low_error1} and  Figure \ref{fig:case1low_error2}, where we plot both $L_2$ and energy errors. The errors are computed in a standard way using finite element discretization. In Figure \ref{fig:case1low_error2}, we zoom the error graph into small time interval since the error at initial time is larger.  Our main observation is that our proposed approach that treats additional degrees of freedom explicitly provides a similar result as the approach where all degrees of freedom are treated implicitly.

\begin{figure}[H]
\centering
\includegraphics[scale=0.35]{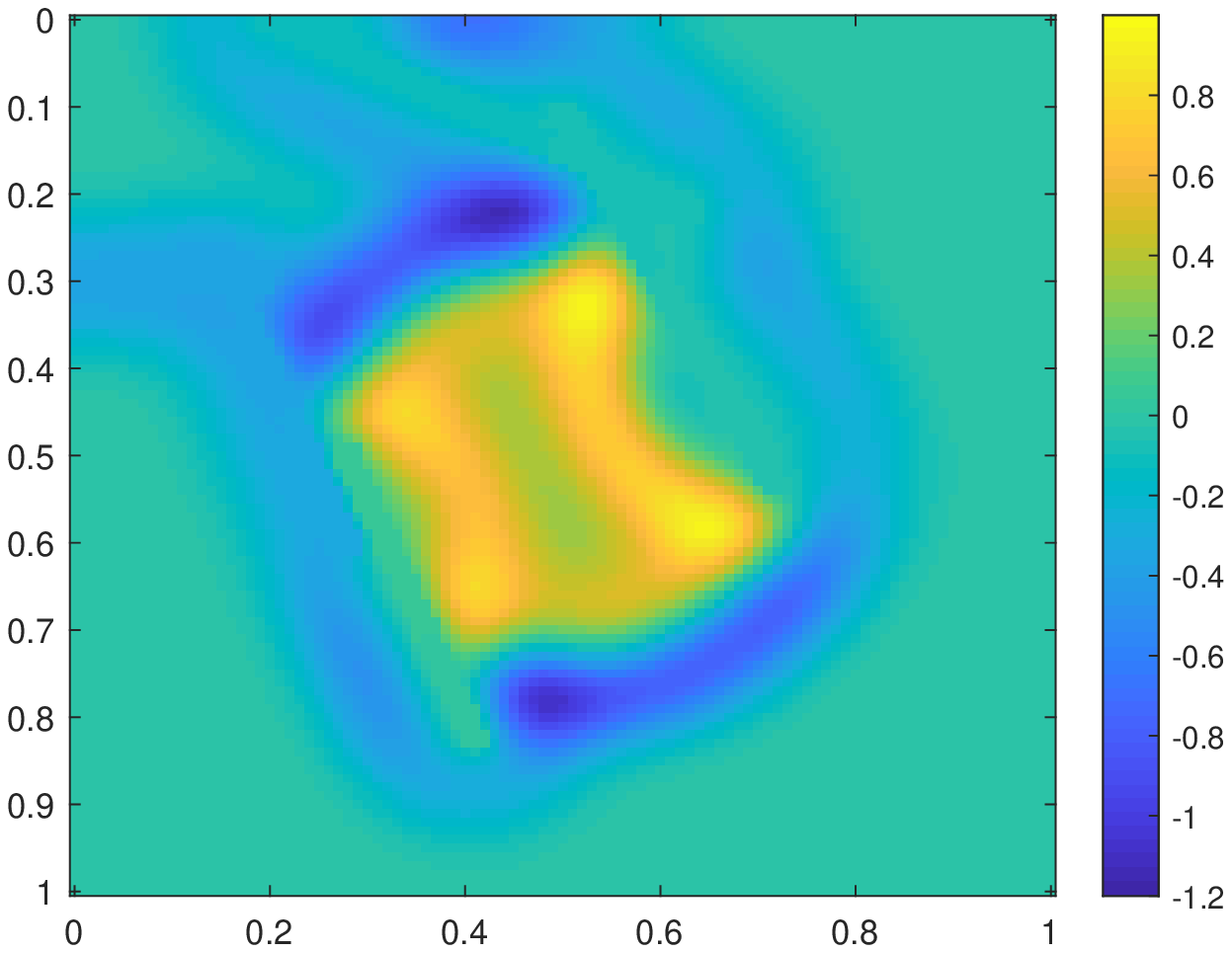} 
\includegraphics[scale=0.35]{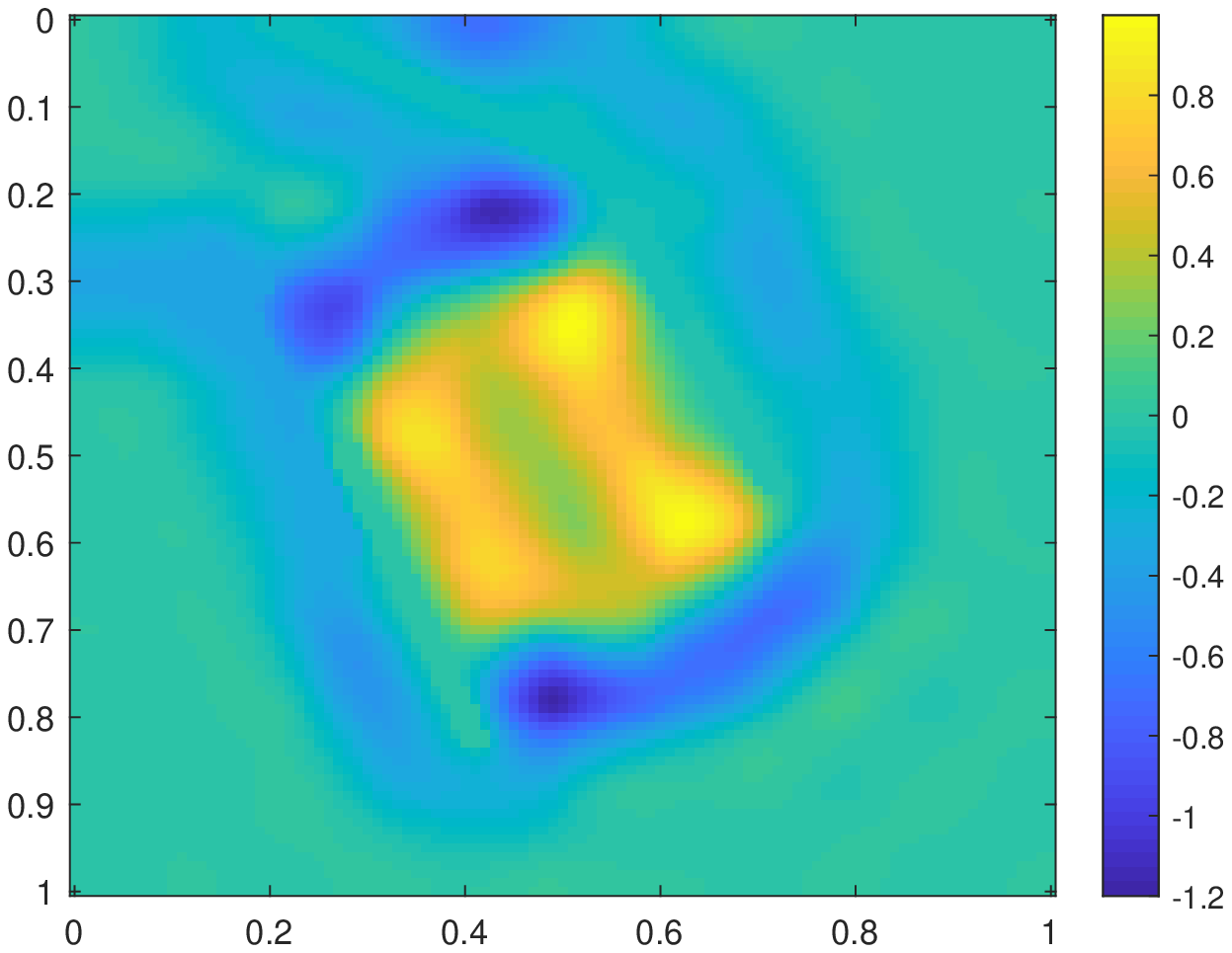}

\includegraphics[scale=0.35]{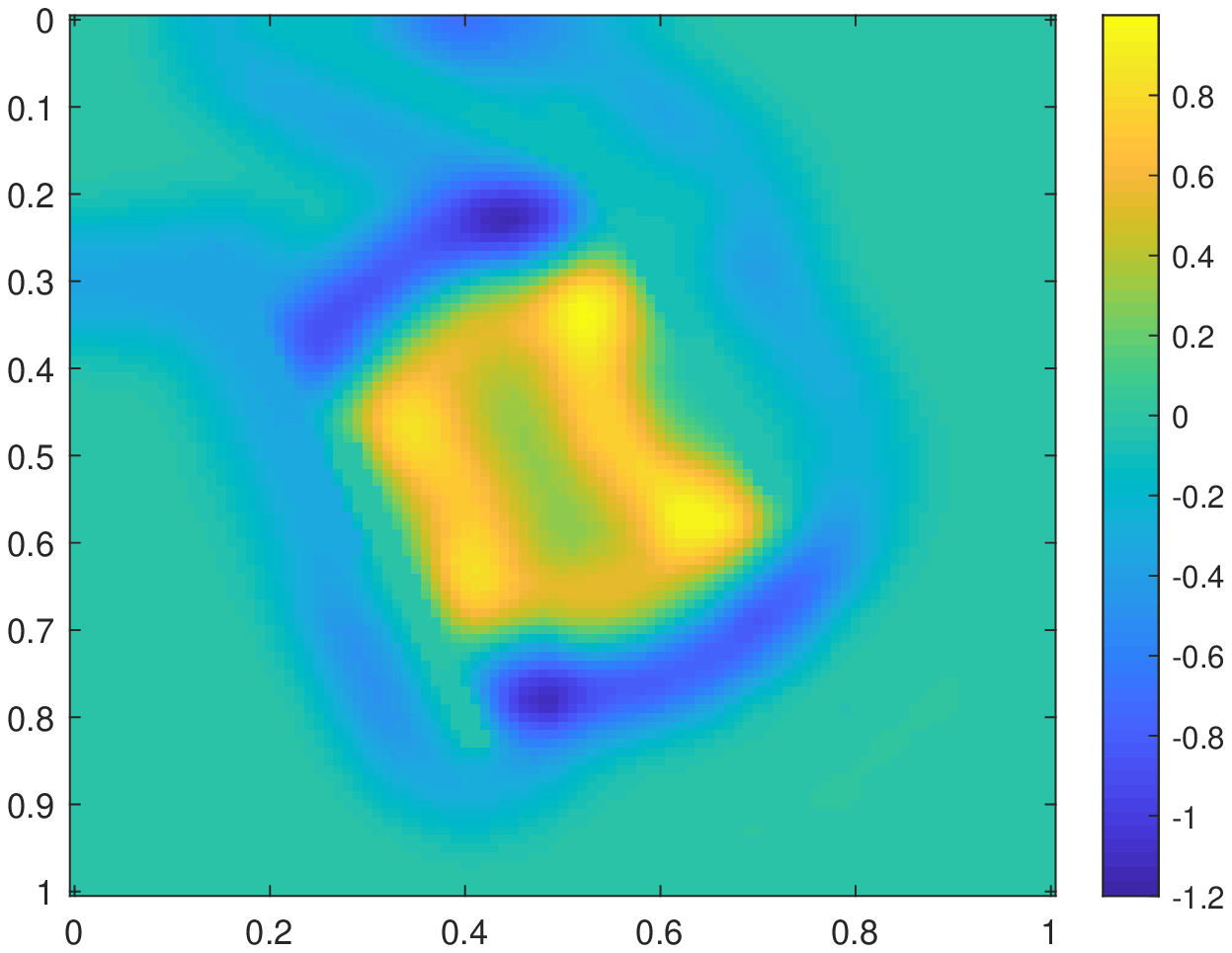} 
\includegraphics[scale=0.35]{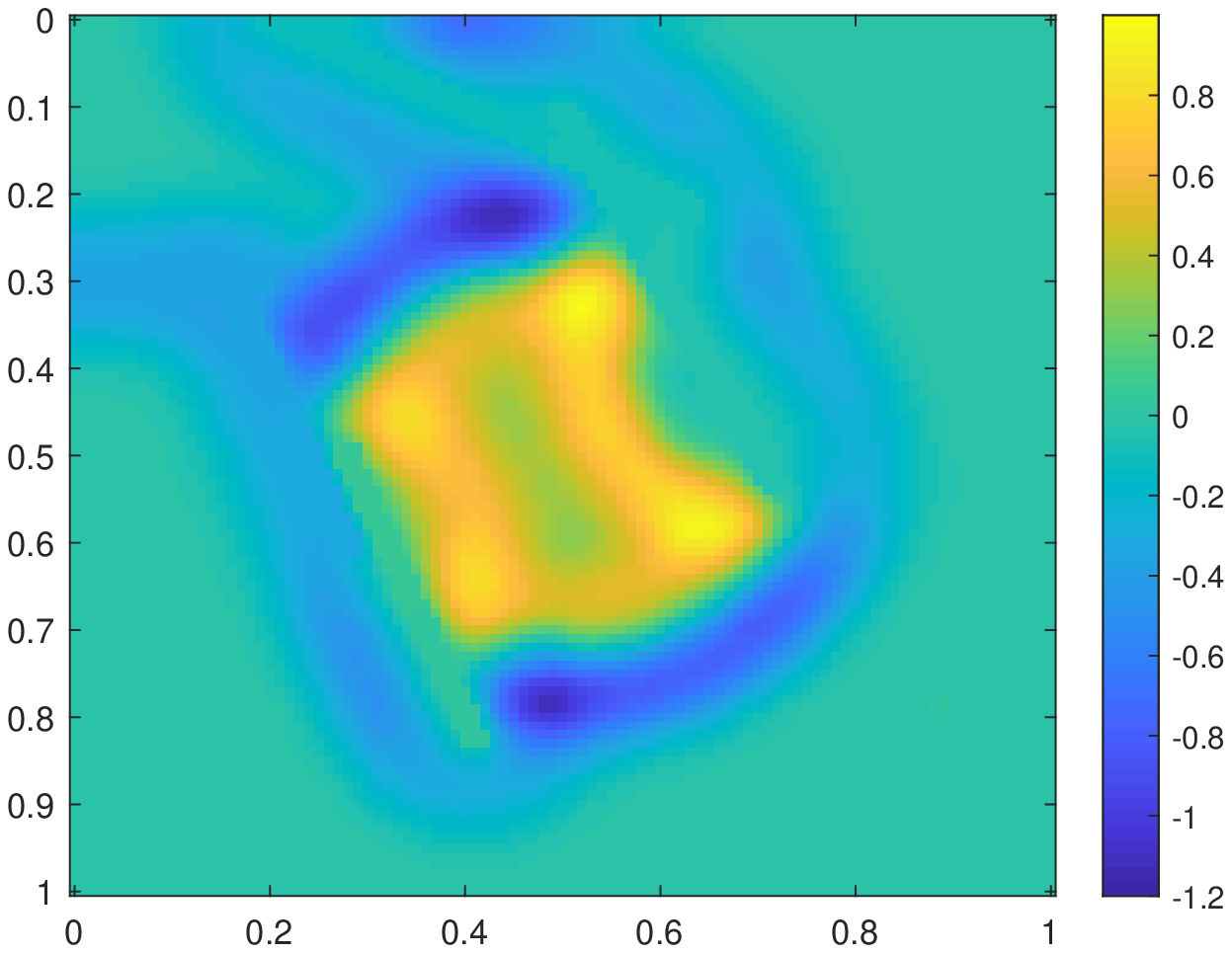} 
\caption{Snapshot at $T=0.3$. Top-left: reference solution. Top-Right: Implicit CEM-GMsFEM solution. Top-left: Proposed splitting method with additional basis functions. Top-right: Implicit CEM-GMsFEM with additional basis.}
\label{fig:case1low}
\end{figure}

\begin{figure}[H]
\centering
\includegraphics[scale=0.35]{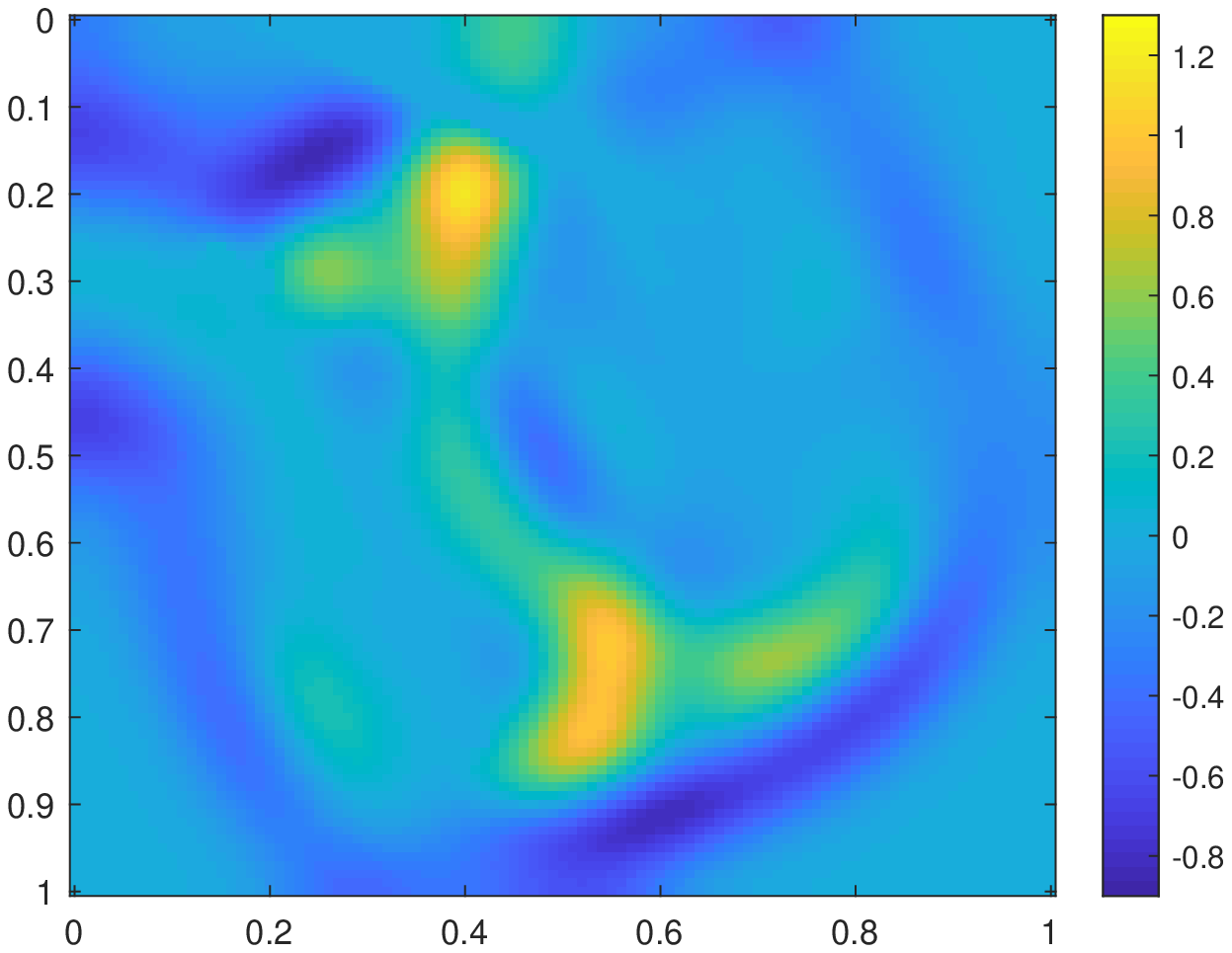} 
\includegraphics[scale=0.35]{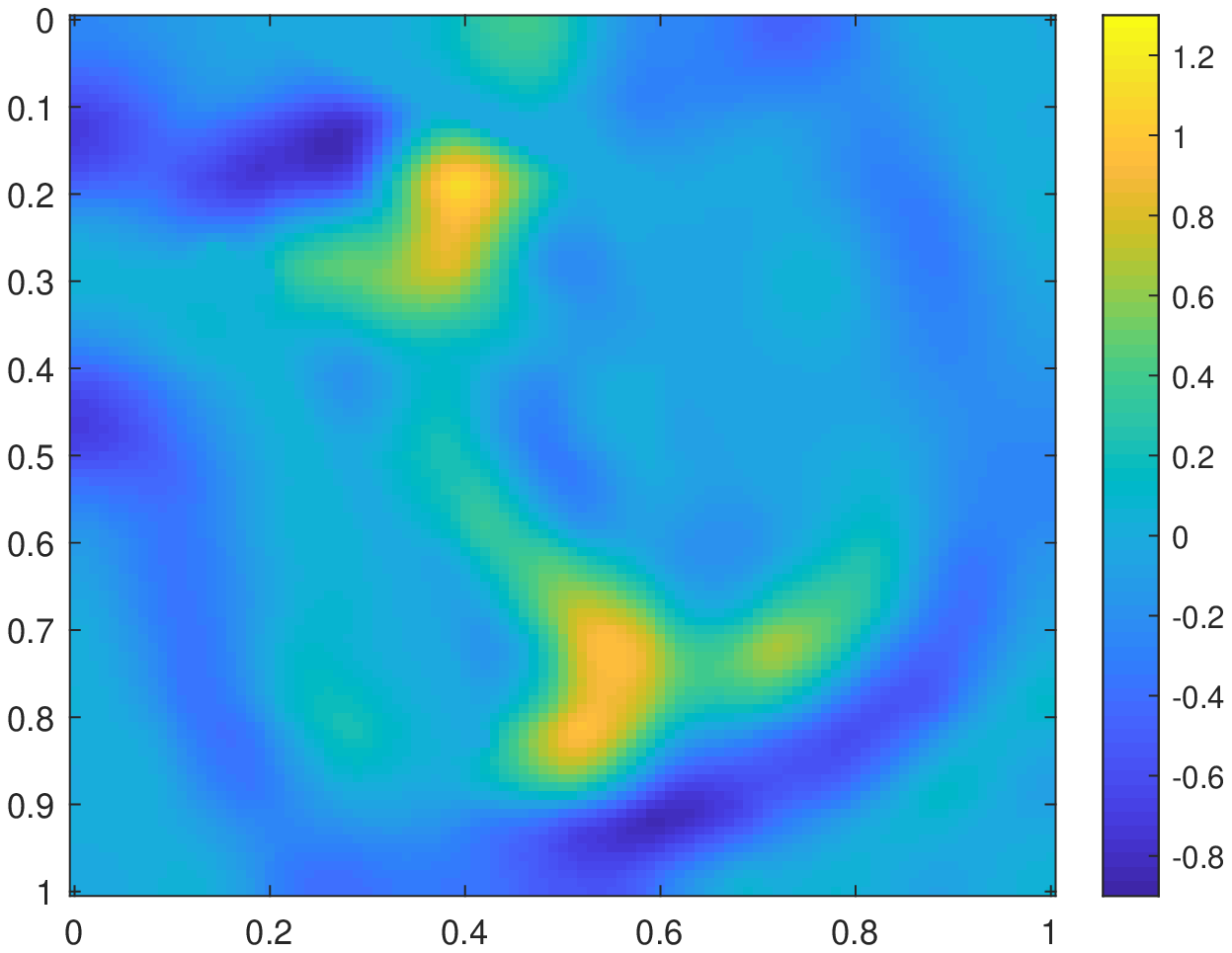}

\includegraphics[scale=0.35]{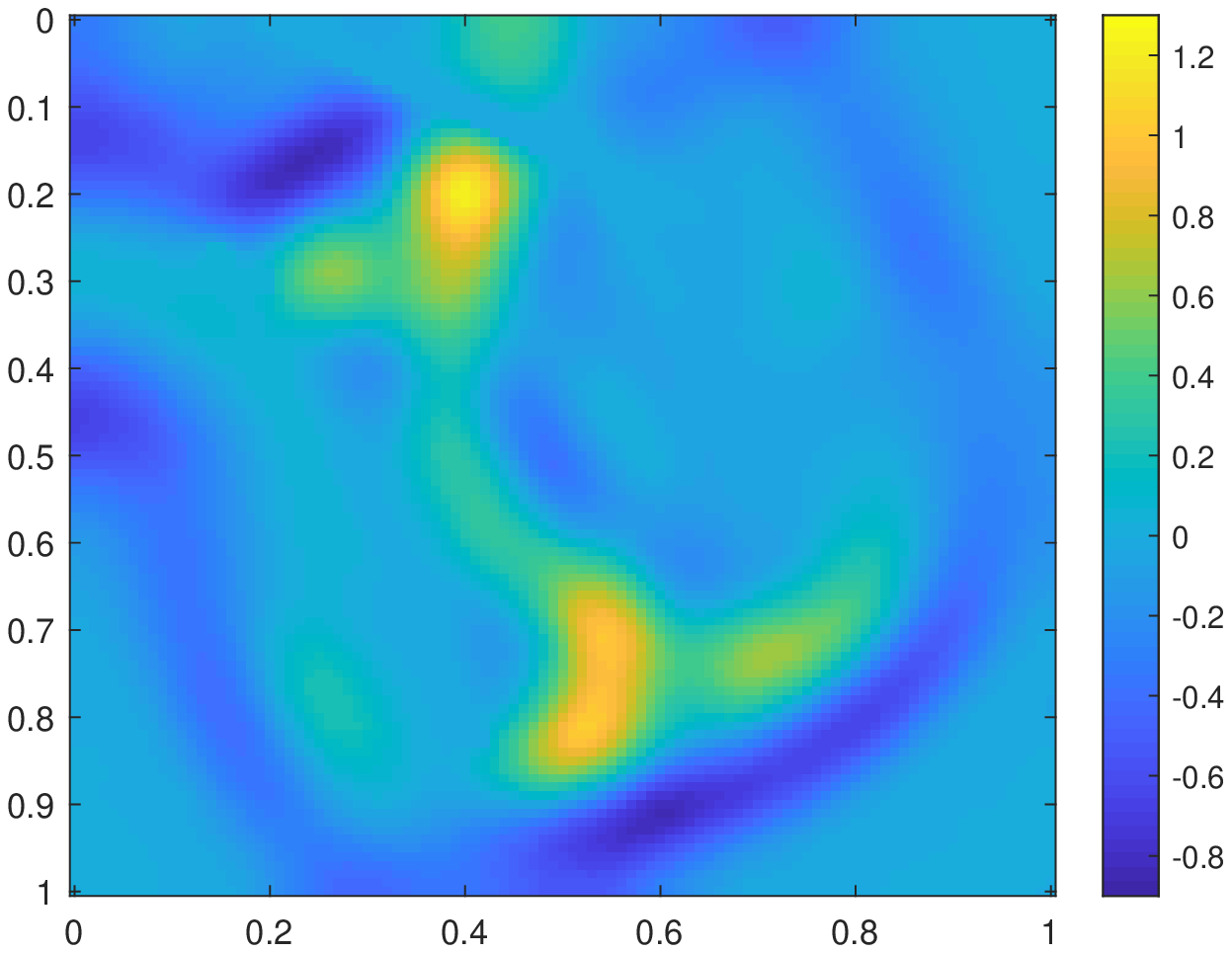} 
\includegraphics[scale=0.35]{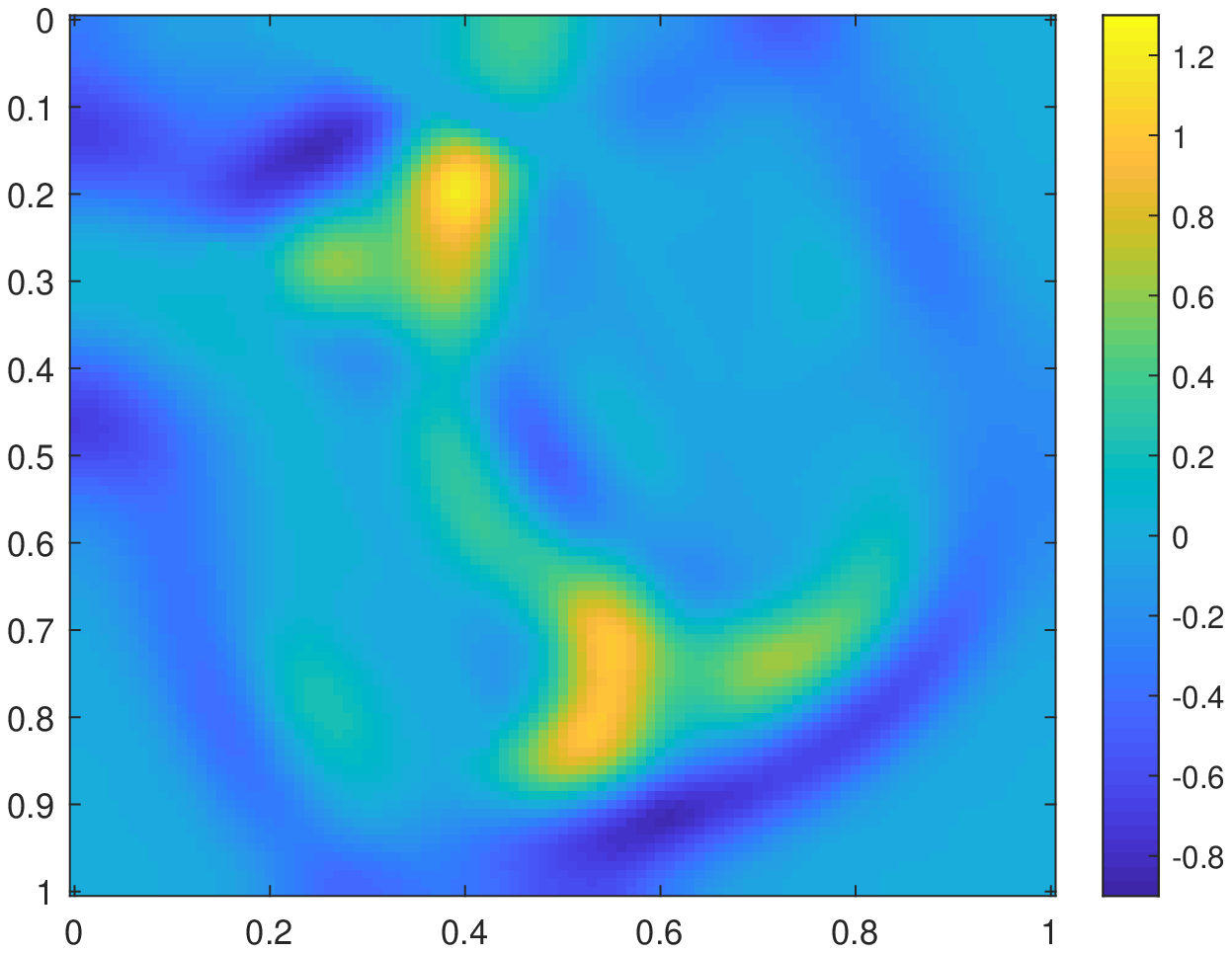}
\caption{Snapshot at $T=0.6$. op-left: reference solution. Top-Right: Implicit CEM-GMsFEM solution. Top-left: Proposed splitting method with additional basis functions. Top-right: Implicit CEM-GMsFEM with additional basis.}
\label{fig:case1low2}
\end{figure}

\begin{figure}[H]
\centering
\includegraphics[scale=0.4]{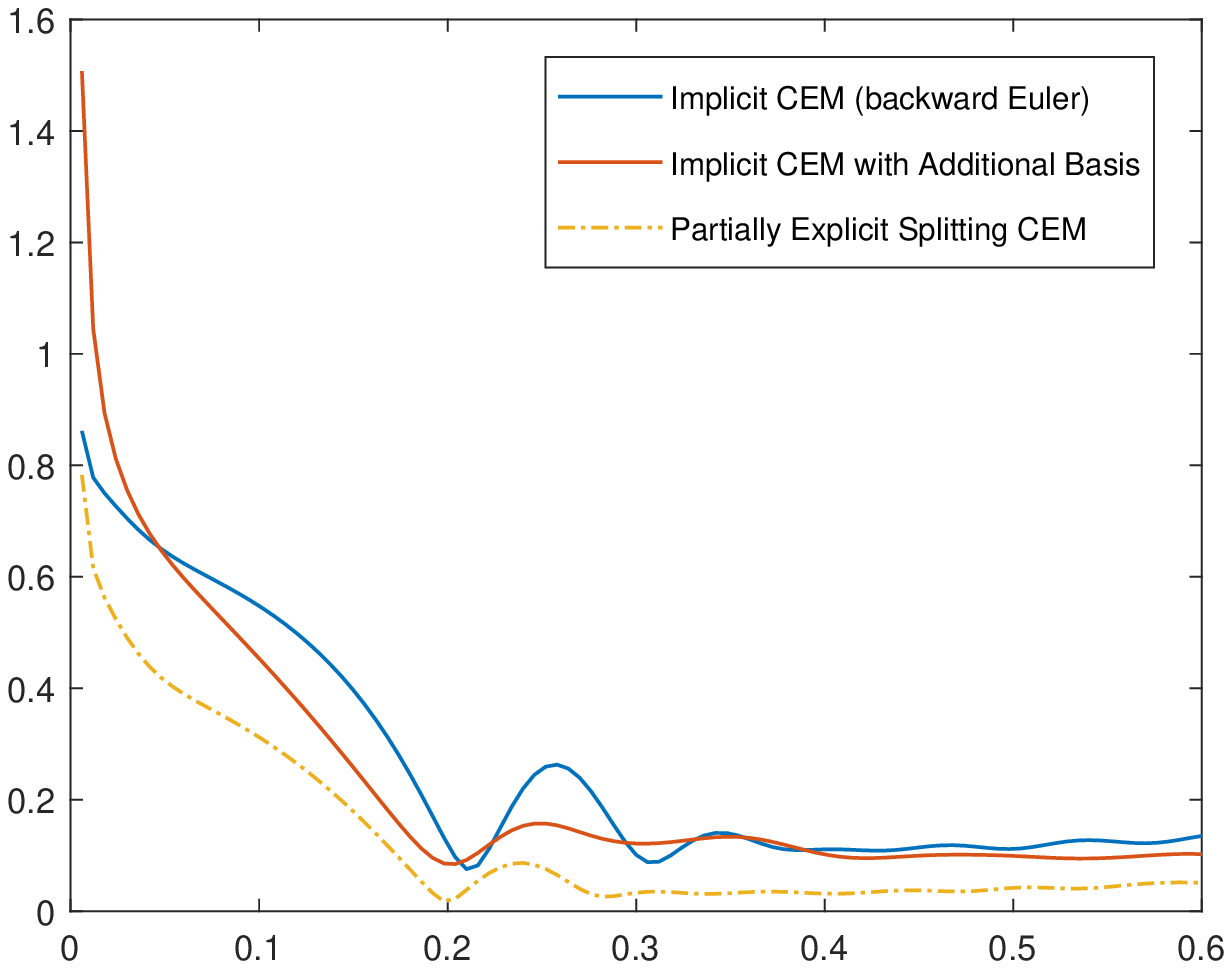} \includegraphics[scale=0.4]{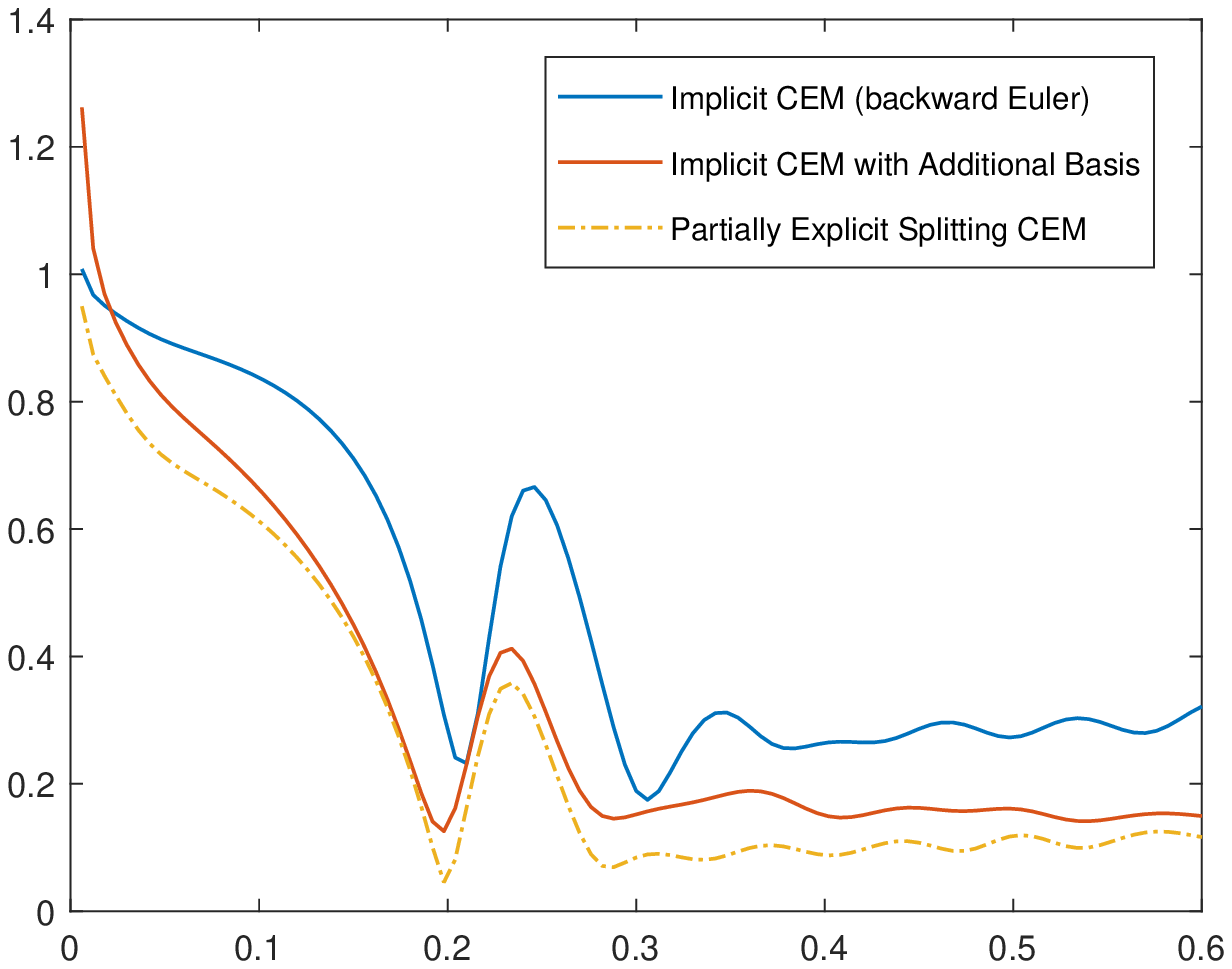}
\caption{Second type of $V_{2,H}$ (CEM Dof: $300$, $V_{2,H}$ Dof: $300$).
Left: $L_{2}$ error. Right: Energy error.}
\label{fig:case1low_error1}
\end{figure}

\begin{figure}[H]
\centering
\includegraphics[scale=0.4]{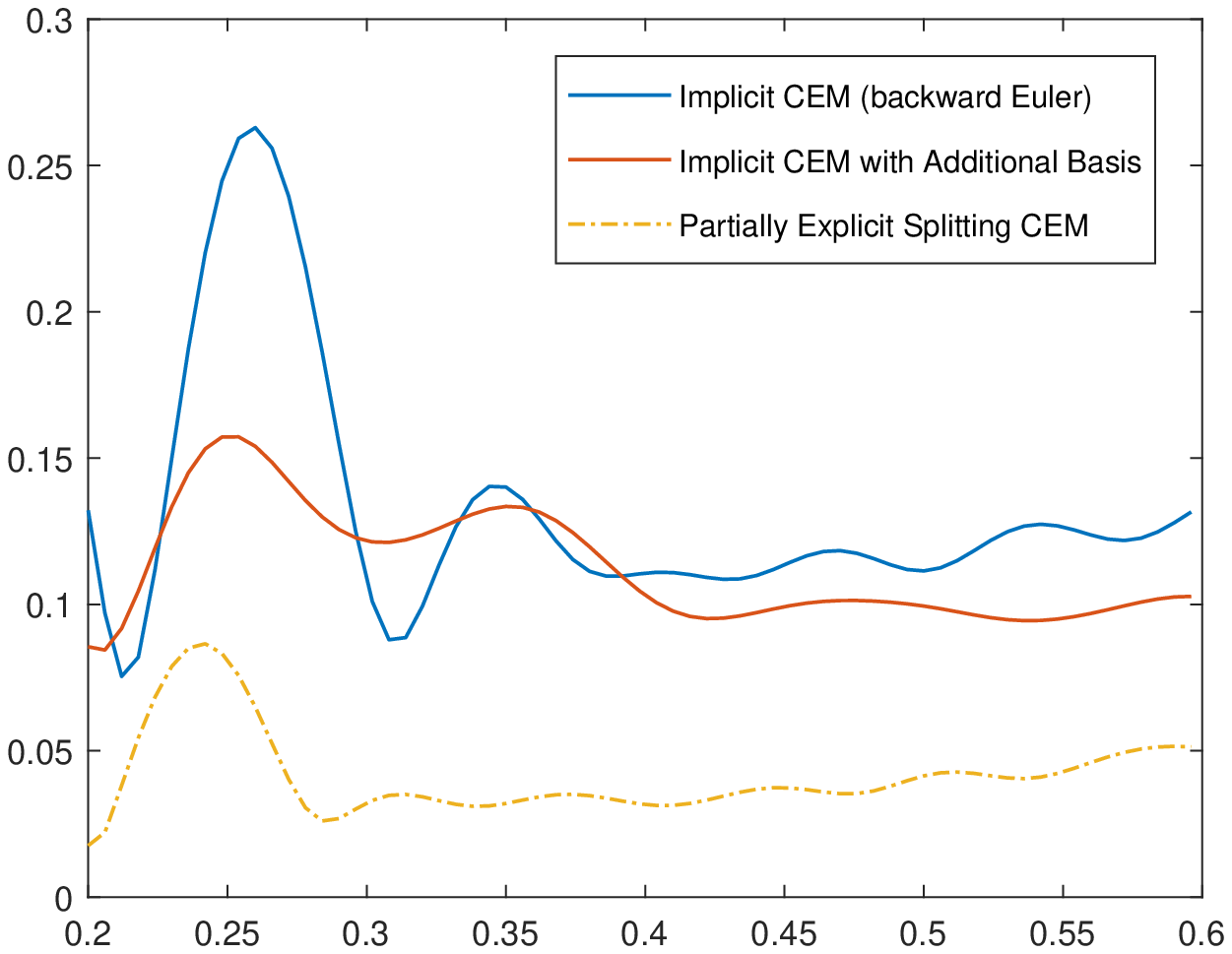} \includegraphics[scale=0.4]{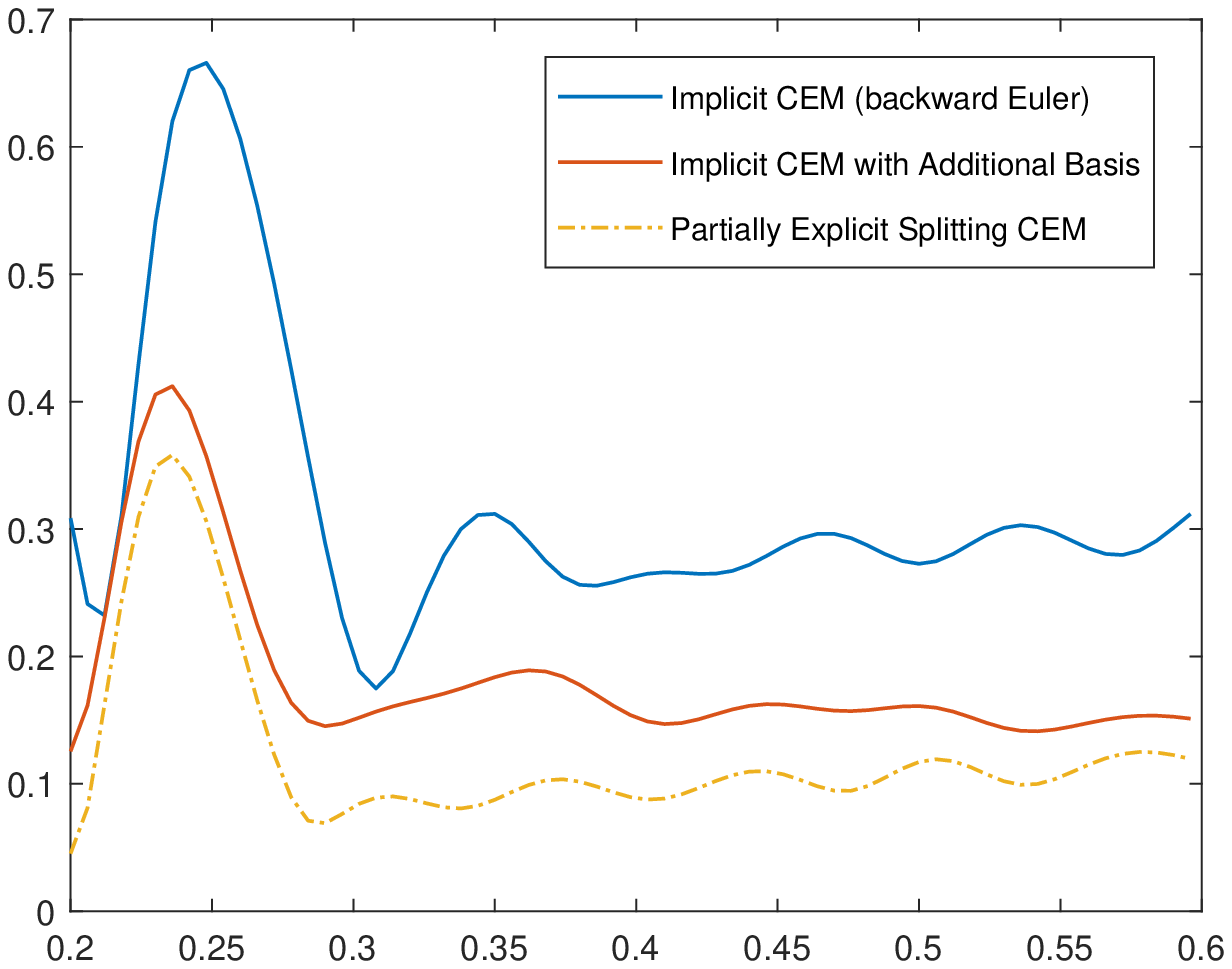}
\caption{The error in the time interval from $0.2$ to $0.6$. Second type of $V_{2,H}$ is used (CEM Dof: $300$, $V_{2,H}$ Dof: $300$).
Left: $L_{2}$ error. Right: Energy error.}
\label{fig:case1low_error2}
\end{figure}

\subsection{Case 2 and  $f_0=1/2$. }

In our second numerical example, we consider the case with the conductivity
$\kappa_2(x)$. As in the previous case, we first depict
the reference solution (top-left), the solution computed using CEM-GMsFEM 
(top-right)
(without additional basis functions), the solution 
computed using additional basis functions with implicit 
method (bottom-left), and 
the solution computed using additional basis functions 
using partially explicit method (bottom-right) in Figure \ref{fig:case2low}.
We see that additional basis functions provide an improved result.
Moreover, there is a little difference between 
the implicit solution that uses additional basis functions
and the partially explicit solution that treats additional degrees
of freedom explicitly. 
 In Figure \ref{fig:case2low}, we plot the solution at the time
$T=0.3$ and we can make similar observations.
 In Figure \ref{fig:case2low2}, we plot the solution and its approximations (as in Figure \ref{fig:case2low}) for $T=0.6$. The errors are plotted 
in Figure \ref{fig:case2low_error1} and  Figure \ref{fig:case2low_error2}, where we plot both $L_2$ and energy errors. In Figure \ref{fig:case2low_error2}, we zoom the error graph into smaller time interval.  Our main observation is that our proposed approach that treats additional degrees of freedom explicitly provides a similar result as the approach where all degrees of freedom are treated implicitly.

\begin{figure}[H]
\centering
\includegraphics[scale=0.35]{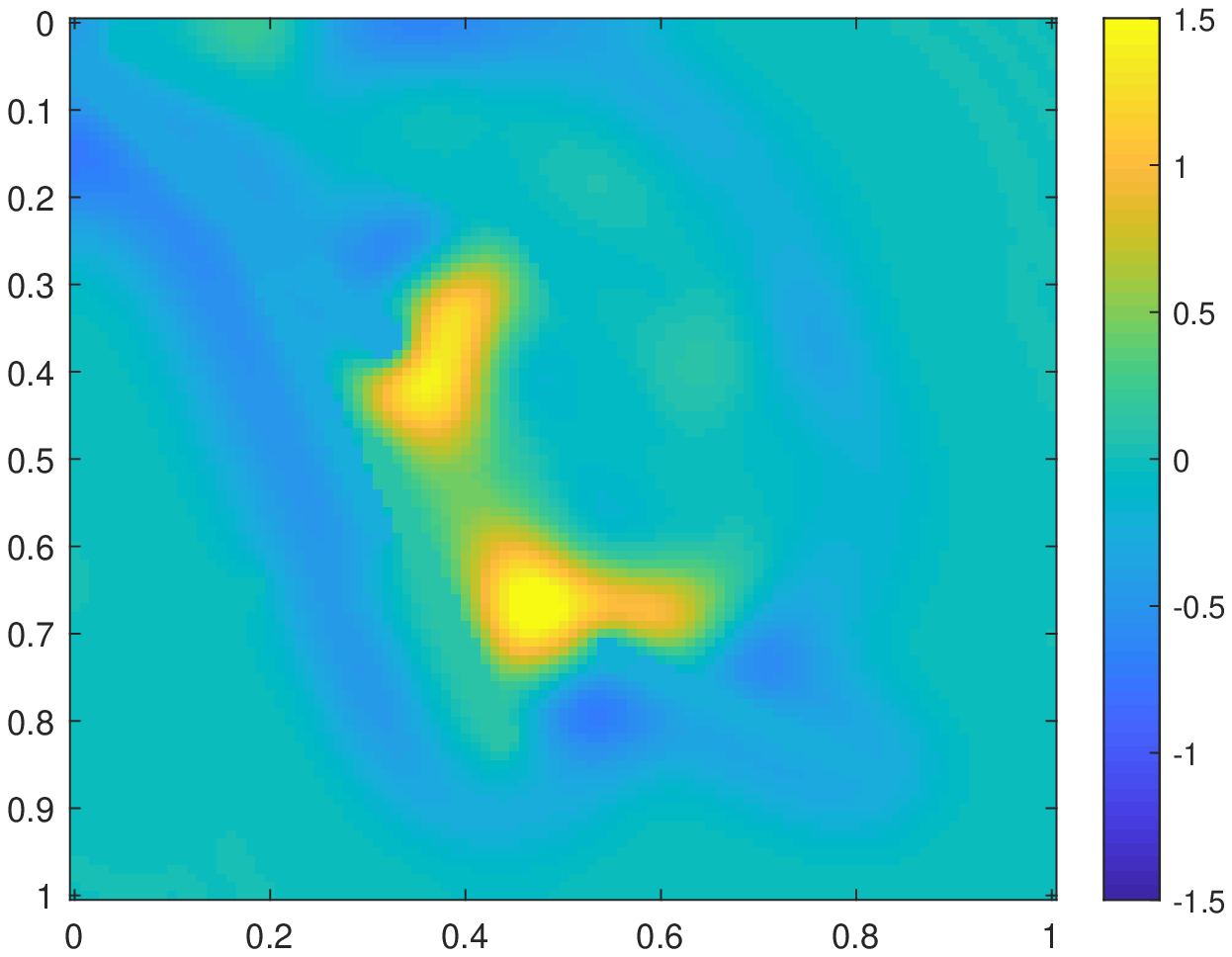} 
\includegraphics[scale=0.35]{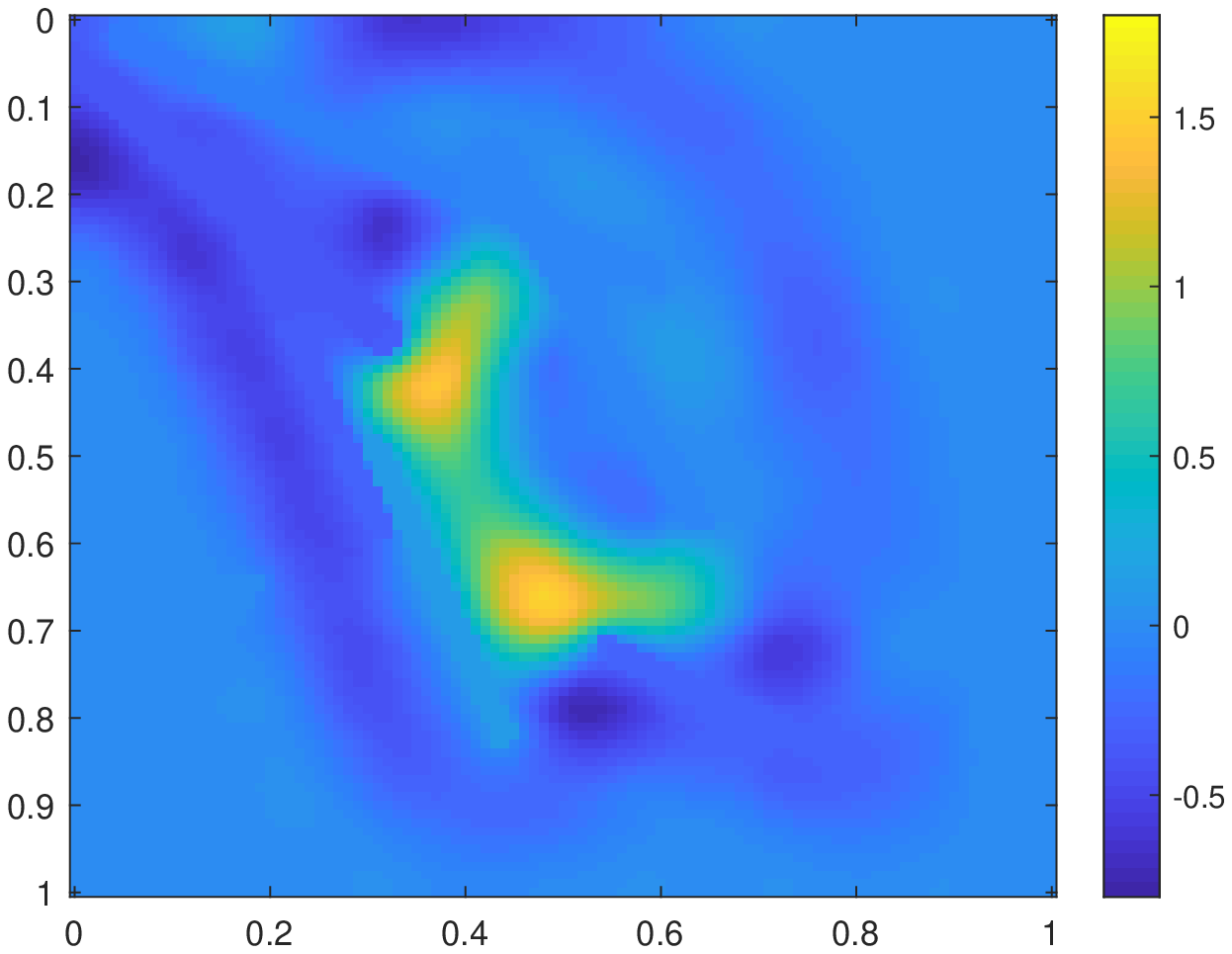}

\includegraphics[scale=0.35]{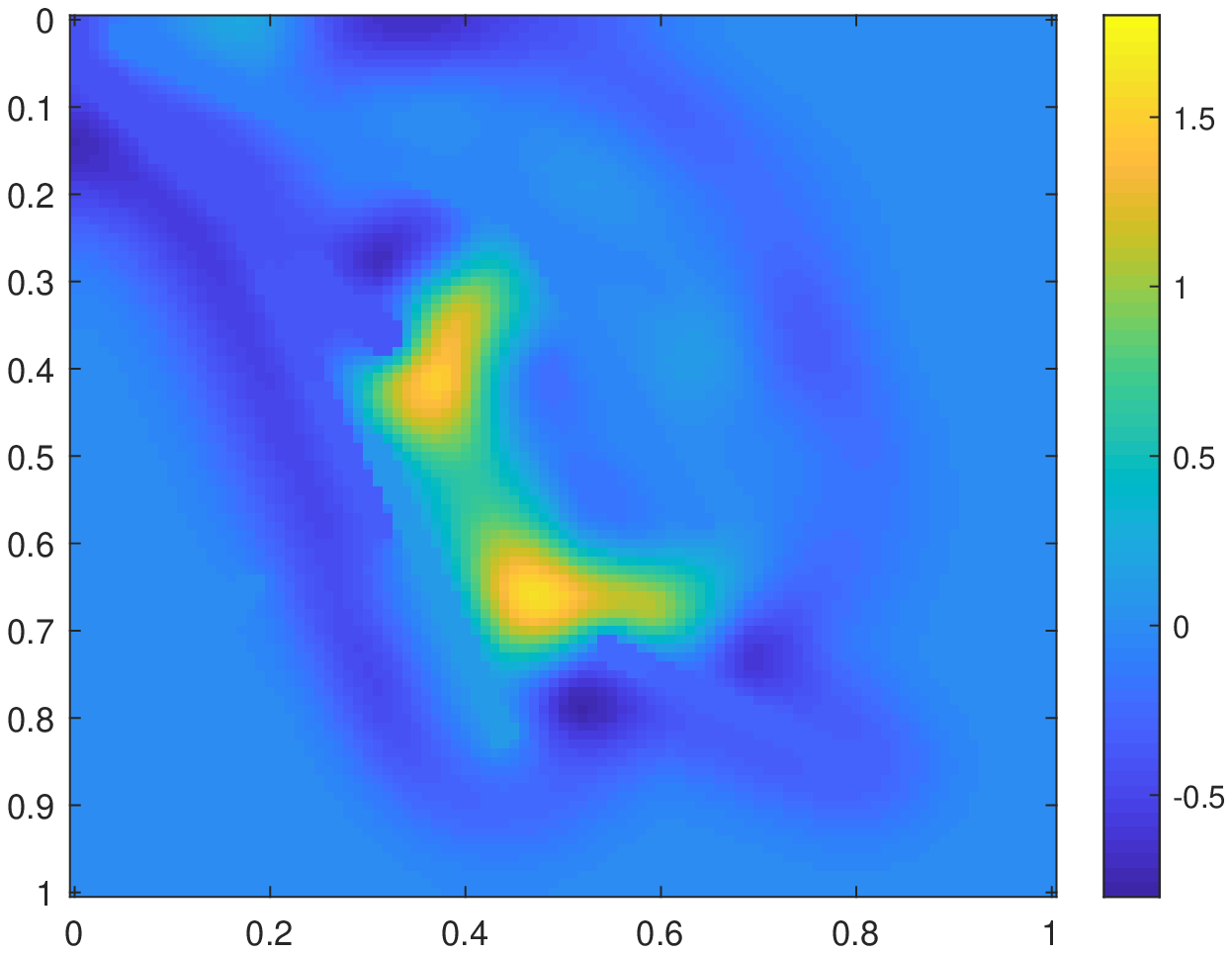} 
\includegraphics[scale=0.35]{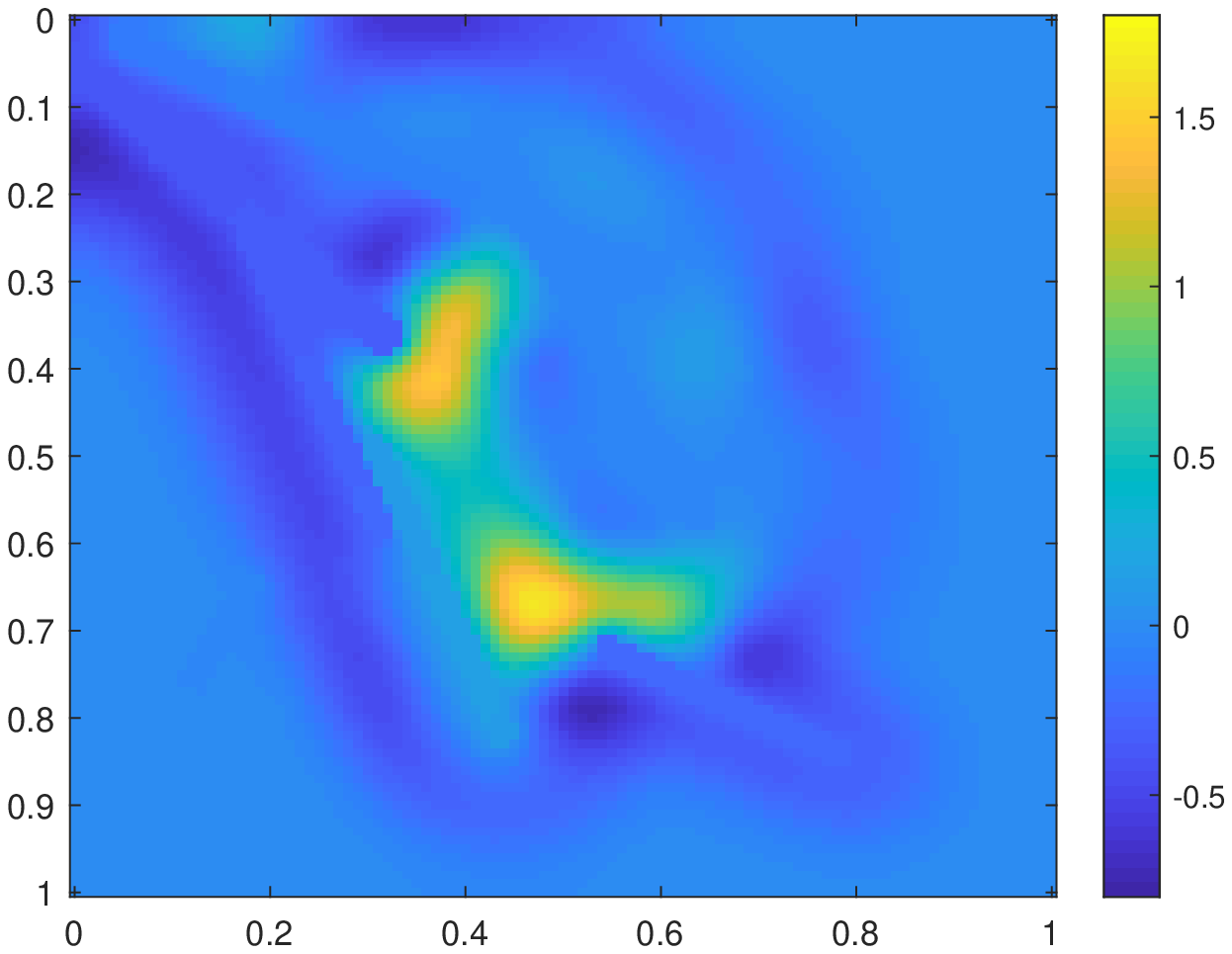} 
\caption{Snapshot at $T=0.3$. Top-left: reference solution. Top-Right: Implicit CEM-GMsFEM solution. Top-left: Proposed splitting method with additional basis functions. Top-right: Implicit CEM-GMsFEM with additional basis.}
\label{fig:case2low}
\end{figure}

\begin{figure}[H]
\centering
\includegraphics[scale=0.35]{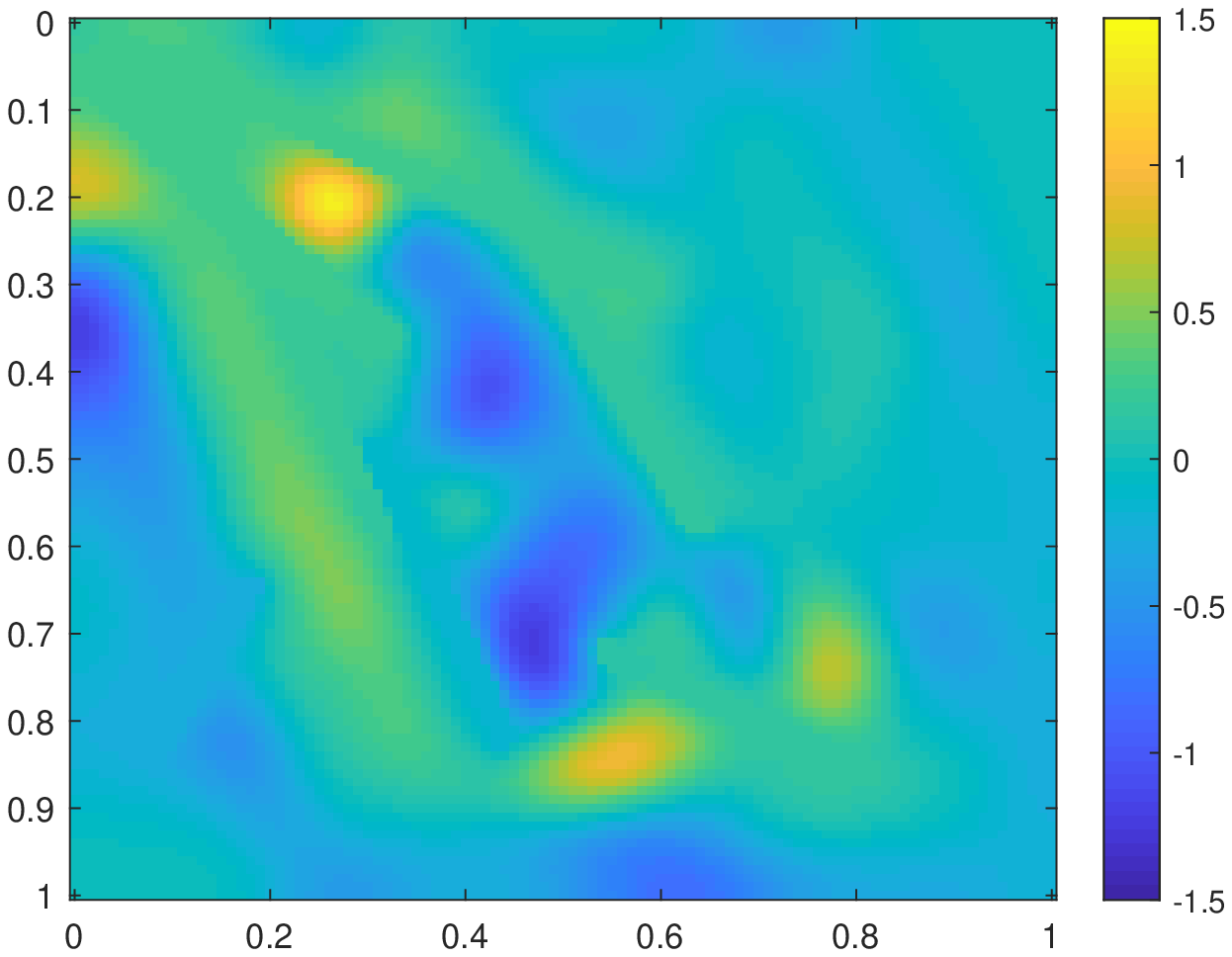} 
\includegraphics[scale=0.35]{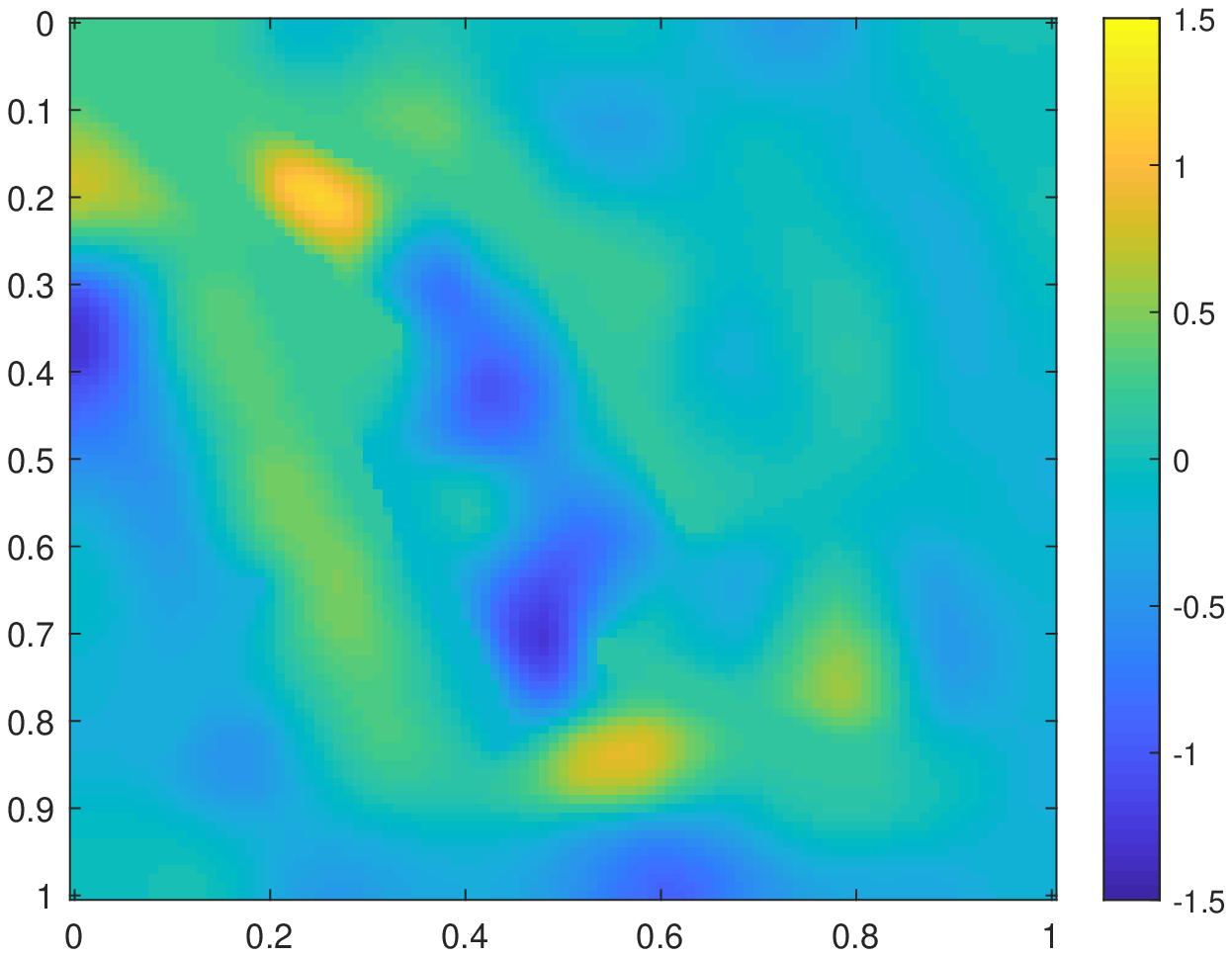}

\includegraphics[scale=0.35]{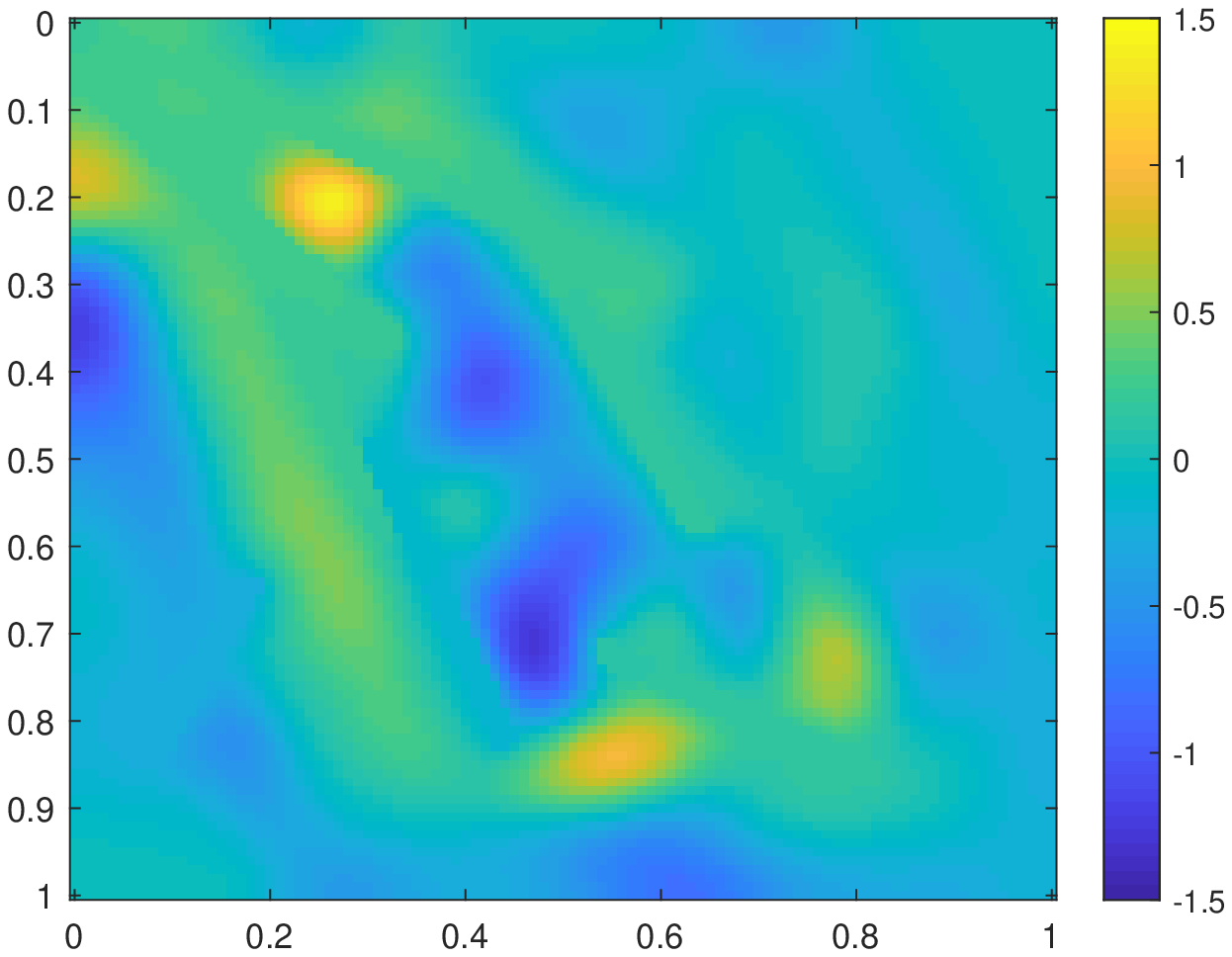} 
\includegraphics[scale=0.35]{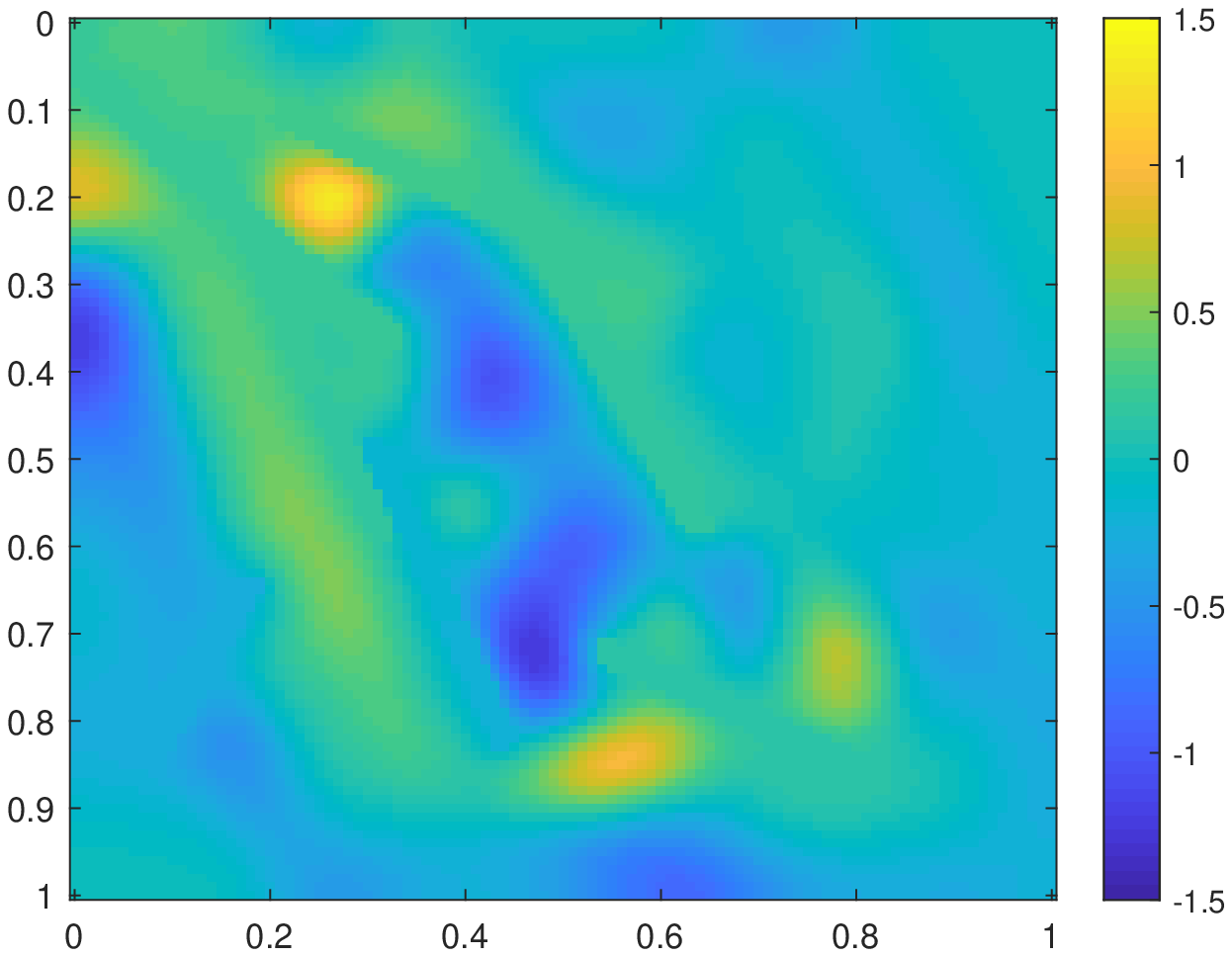}
\caption{Snapshot at $T=0.6$. op-left: reference solution. Top-Right: Implicit CEM-GMsFEM solution. Top-left: Proposed splitting method with additional basis functions. Top-right: Implicit CEM-GMsFEM with additional basis.}
\label{fig:case2low2}
\end{figure}

\begin{figure}[H]
\centering
\includegraphics[scale=0.4]{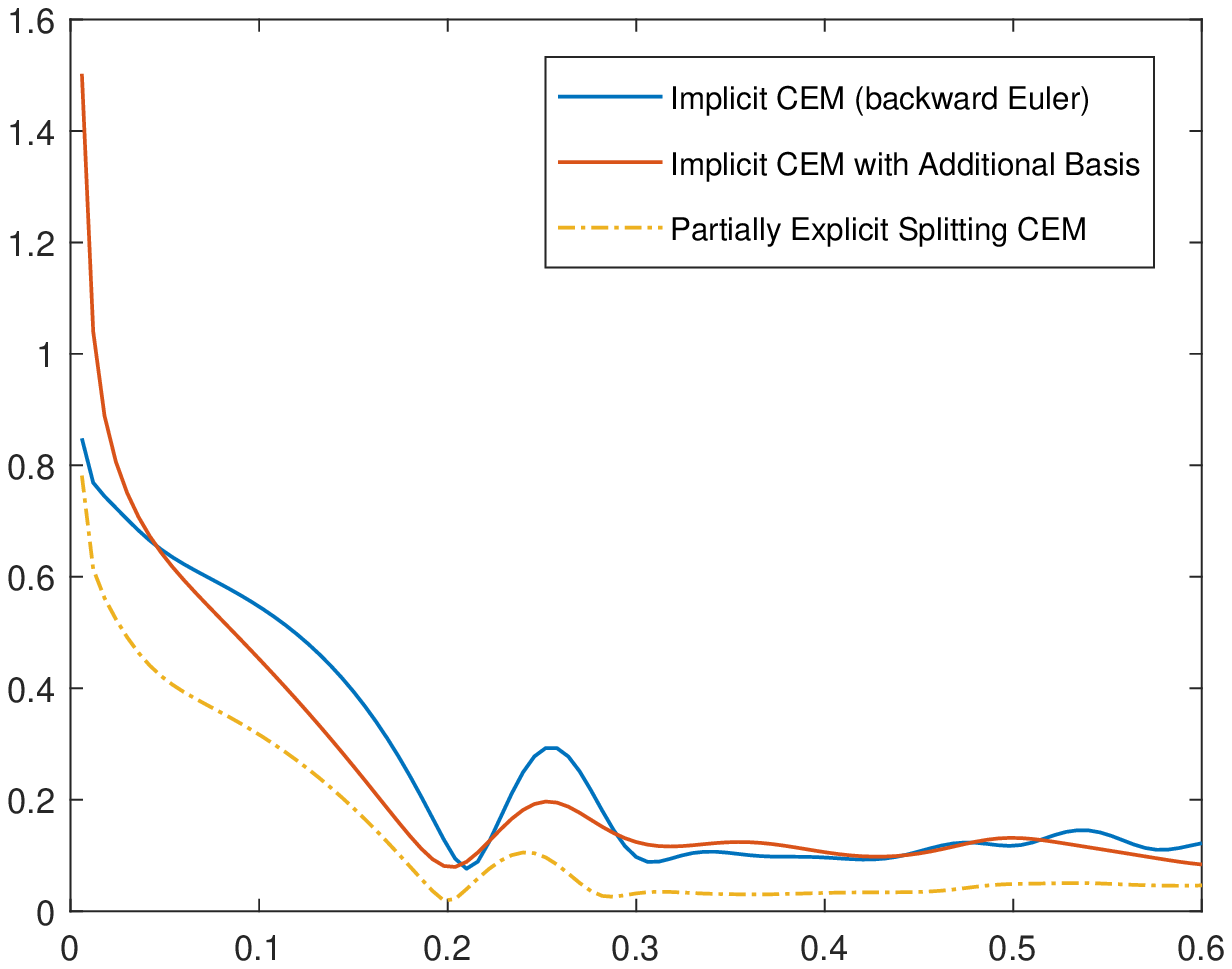} \includegraphics[scale=0.4]{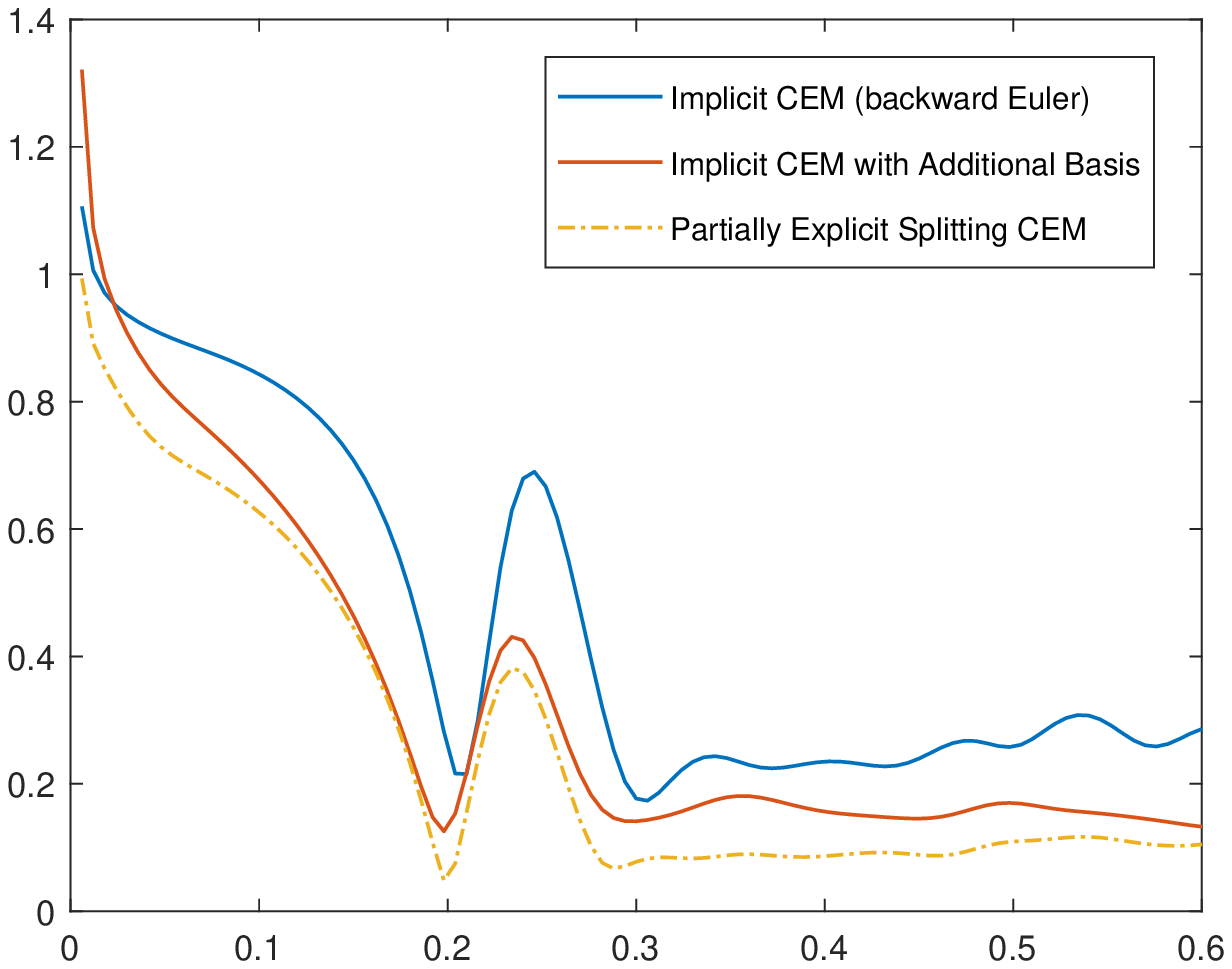}
\caption{Second type of $V_{2,H}$ (CEM Dof: $300$, $V_{2,H}$ Dof: $300$).
Left: $L_{2}$ error. Right: Energy error.}
\label{fig:case2low_error1}
\end{figure}

\begin{figure}[H]
\centering

\includegraphics[scale=0.4]{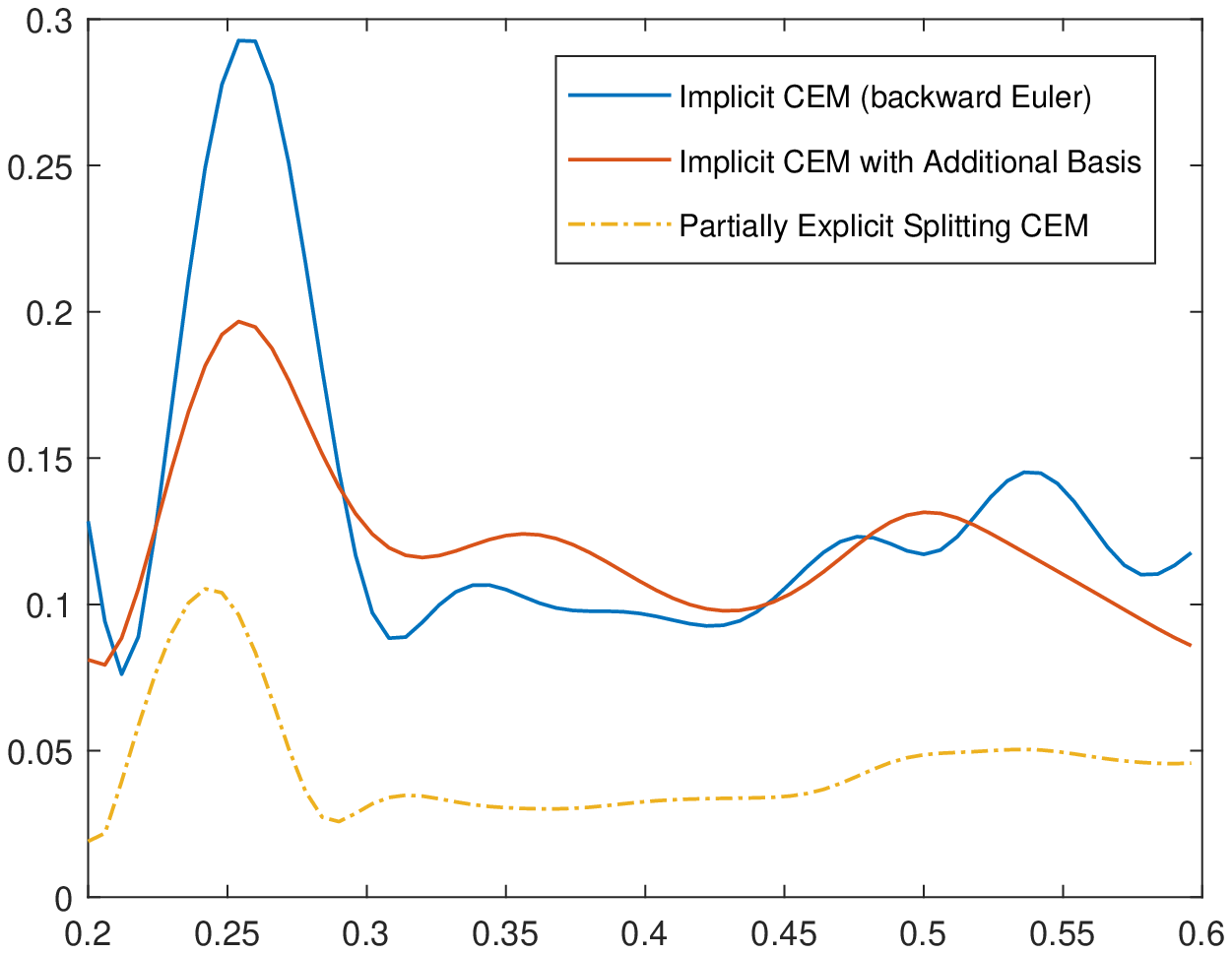} \includegraphics[scale=0.4]{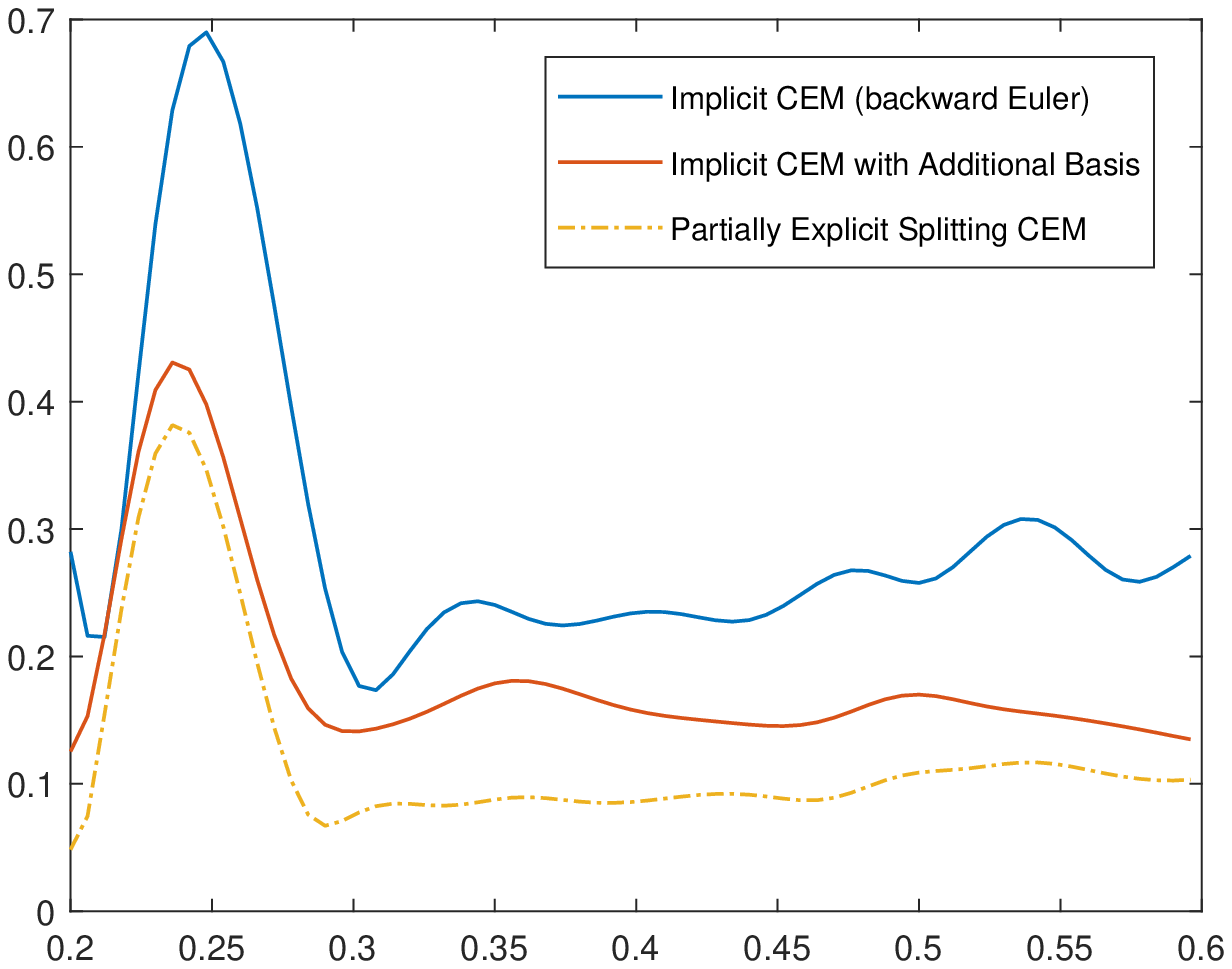}
\caption{The error in the time interval from $0.2$ to $0.6$. Second type of $V_{2,H}$ is used (CEM Dof: $300$, $V_{2,H}$ Dof: $300$).
Left: $L_{2}$ error. Right: Energy error.}
\label{fig:case2low_error2}
\end{figure}

\subsection{Case 2 and $f_0=1$}

In our final numerical example, we consider the case with the conductivity
$\kappa_2(x)$ and higher frequency. 
In this case, the error is larger as expected and additional basis functions
in CEM-GMsFEM provide improved solution.
As in the previous cases, we first depict
the reference solution (top-left), the solution computed using CEM-GMsFEM 
(top-right)
(without additional basis functions), the solution 
computed using additional basis functions with implicit 
method (bottom-left), and 
the solution computed using additional basis functions 
using partially explicit method (bottom-right) in Figure \ref{fig:case2high}.
We observe that 
there is a little difference between 
the implicit solution that uses additional basis functions
and the partially explicit solution that treats additional degrees
of freedom explicitly. 
 In Figure \ref{fig:case2high}, we plot the solution at the time
$T=0.3$ and we can make similar observations.
 In Figure \ref{fig:case2high2}, we plot the solution and its approximations (as in Figure \ref{fig:case2high}) for $T=0.6$. The errors are plotted 
in Figure \ref{fig:case2high_error1} and  Figure \ref{fig:case2high_error2}, where we plot both $L_2$ and energy errors. In Figure \ref{fig:case2high_error2}, we zoom the error graph into smaller time interval.  Our main observation is that our proposed approach that treats additional degrees of freedom explicitly provides a similar result as the approach where all degrees of freedom are treated implicitly.

\begin{figure}[H]
\centering
\includegraphics[scale=0.35]{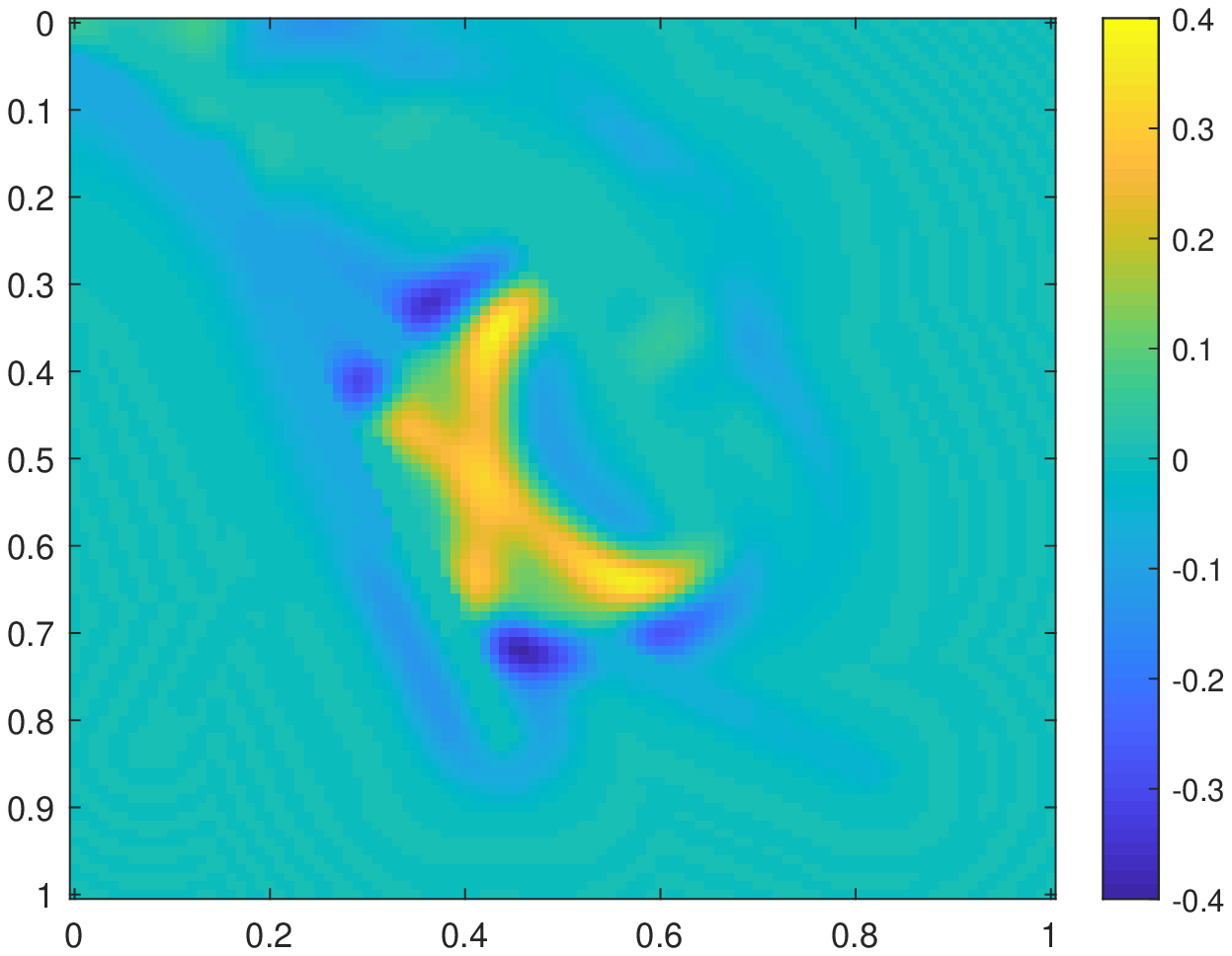} 
\includegraphics[scale=0.35]{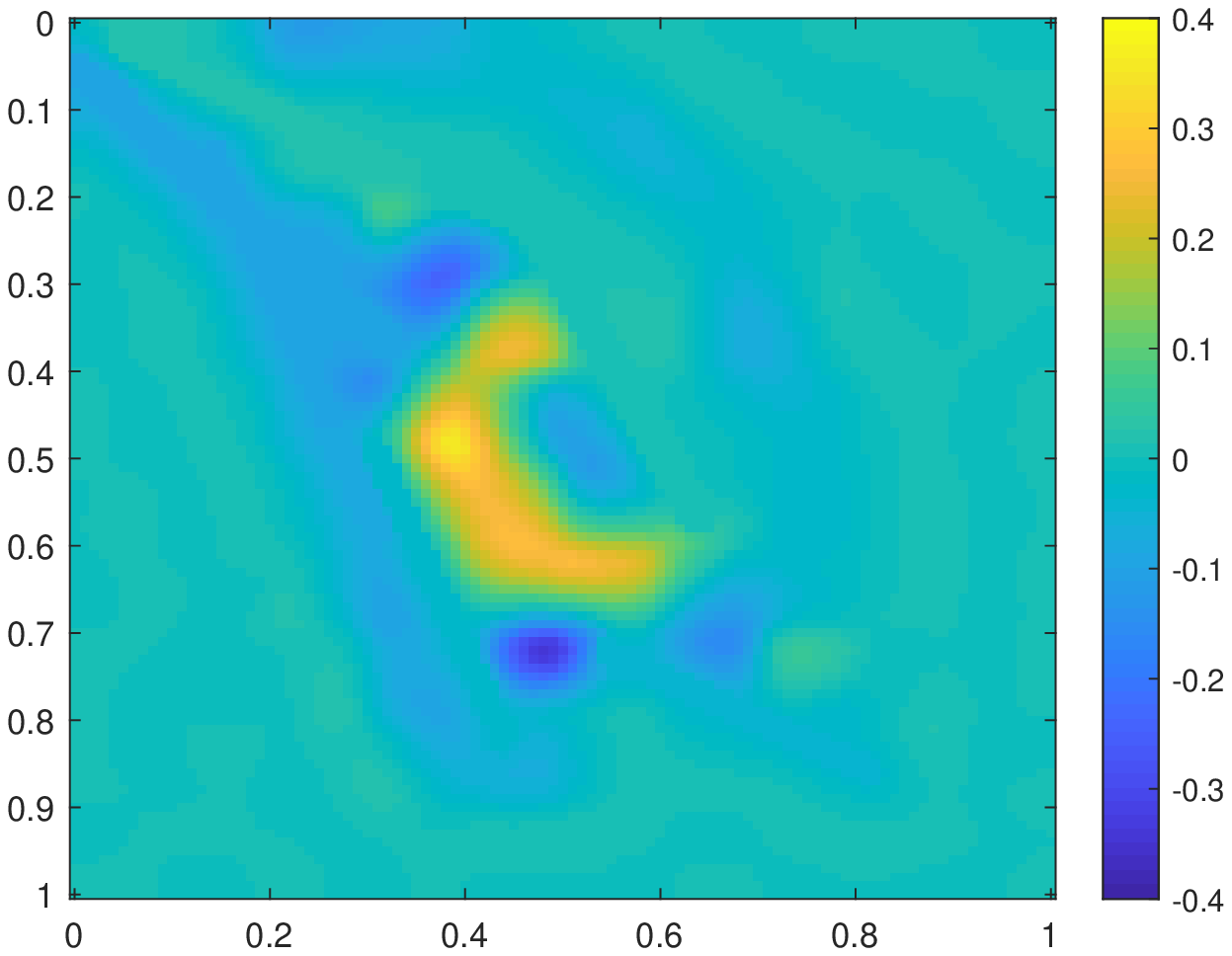}

\includegraphics[scale=0.35]{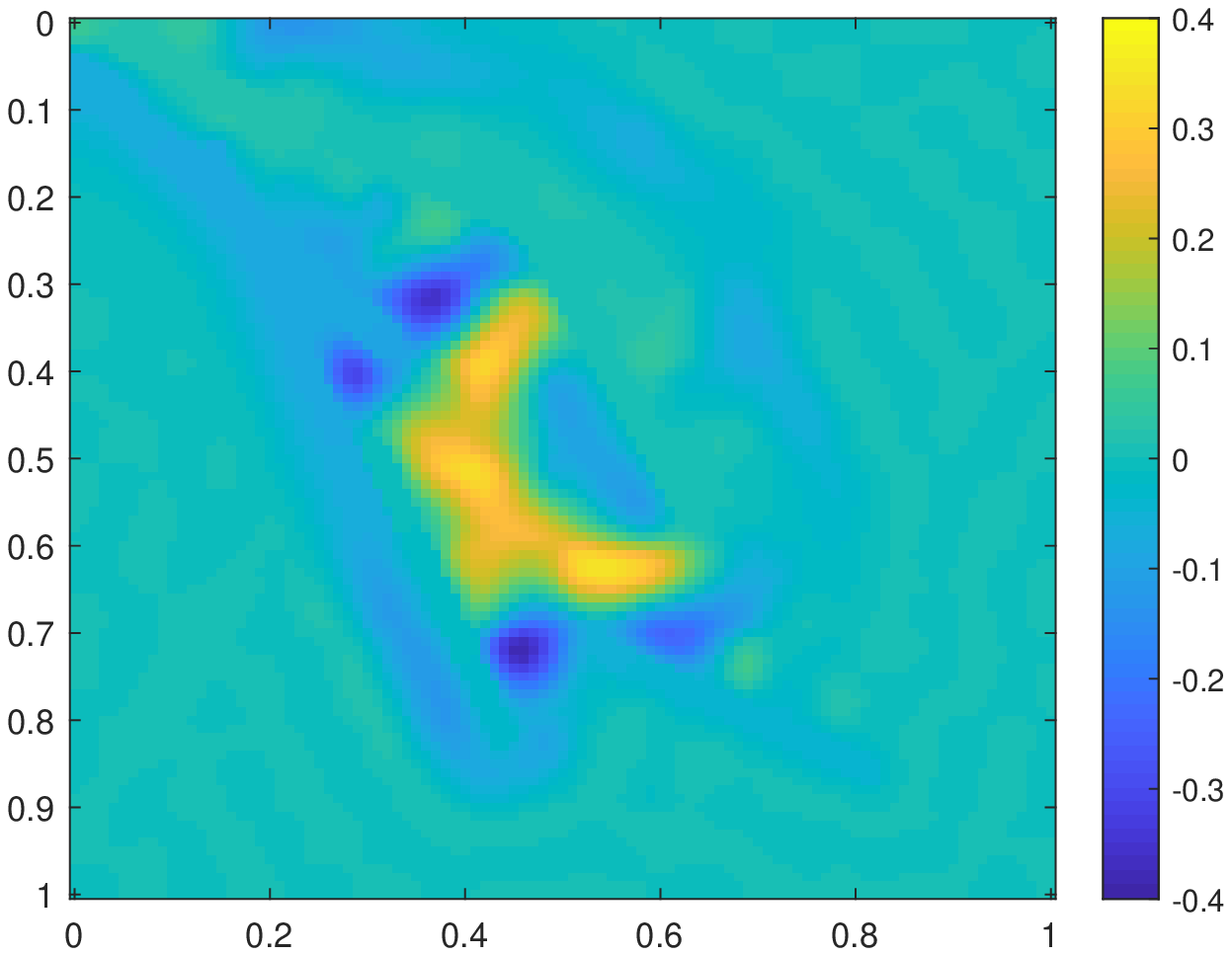} 
\includegraphics[scale=0.35]{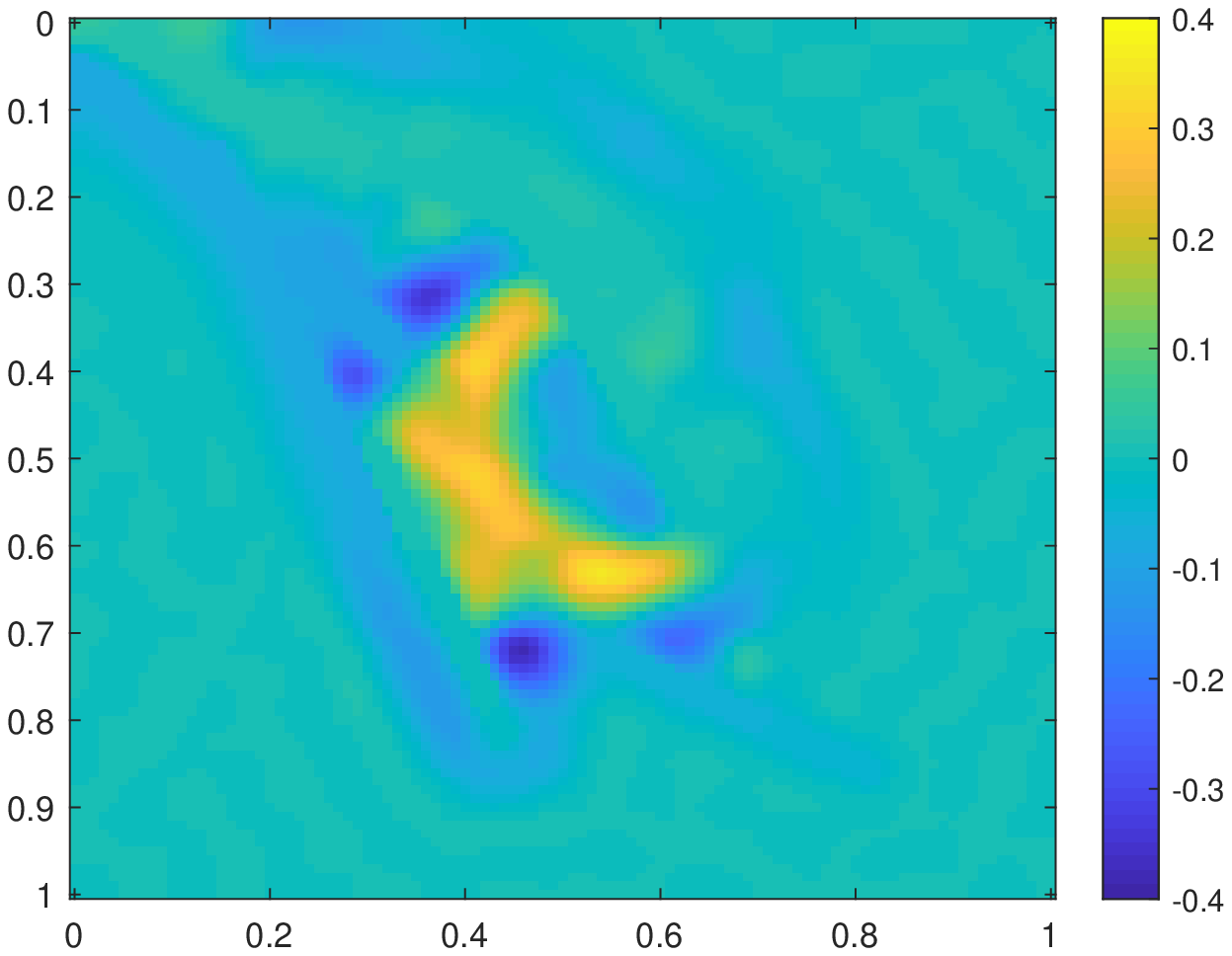} 
\caption{Snapshot at $T=0.3$. Top-left: reference solution. Top-Right: Implicit CEM-GMsFEM solution. Top-left: Proposed splitting method with additional basis functions. Top-right: Implicit CEM-GMsFEM with additional basis.}
\label{fig:case2high}
\end{figure}

\begin{figure}[H]
\centering
\includegraphics[scale=0.35]{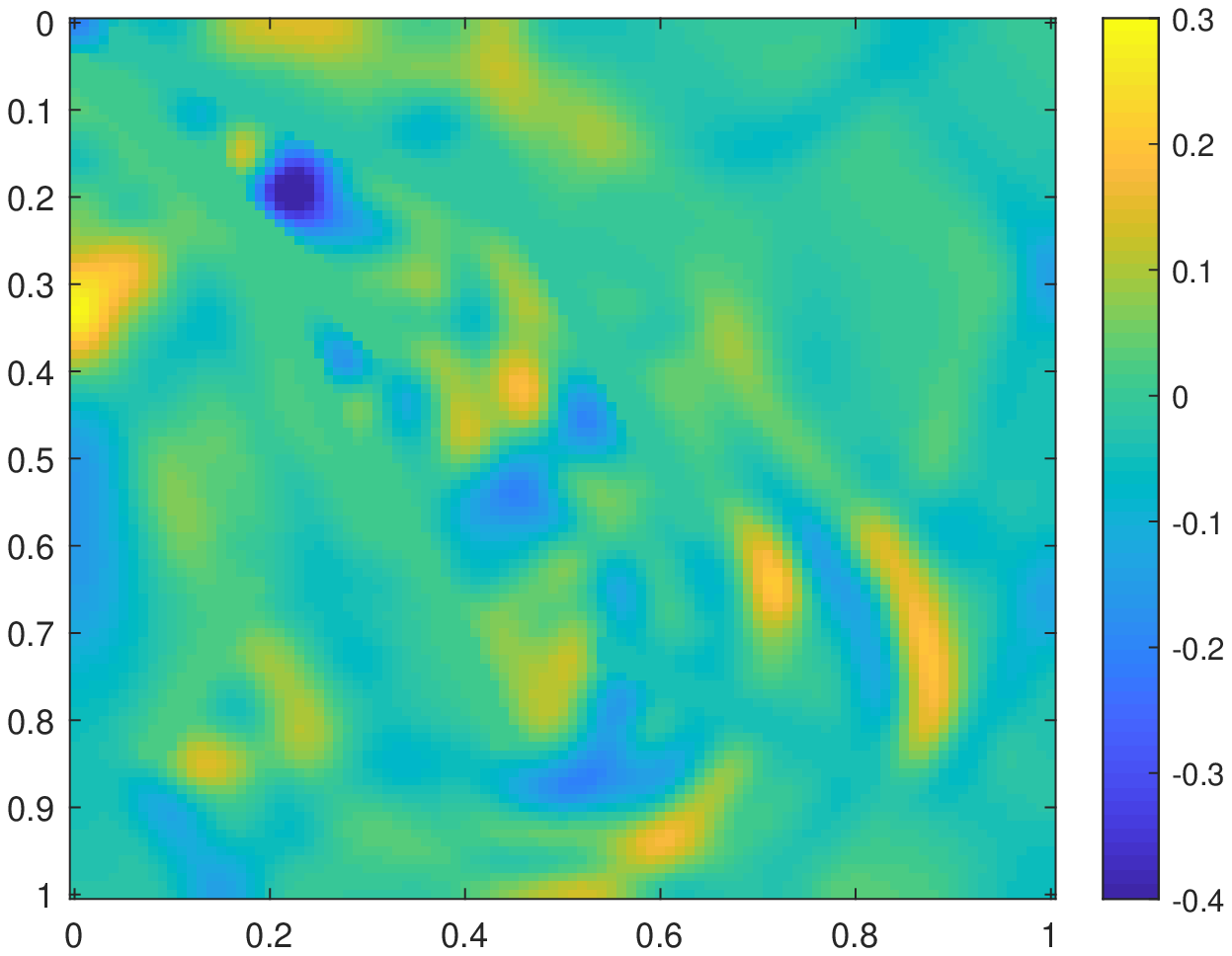} 
\includegraphics[scale=0.35]{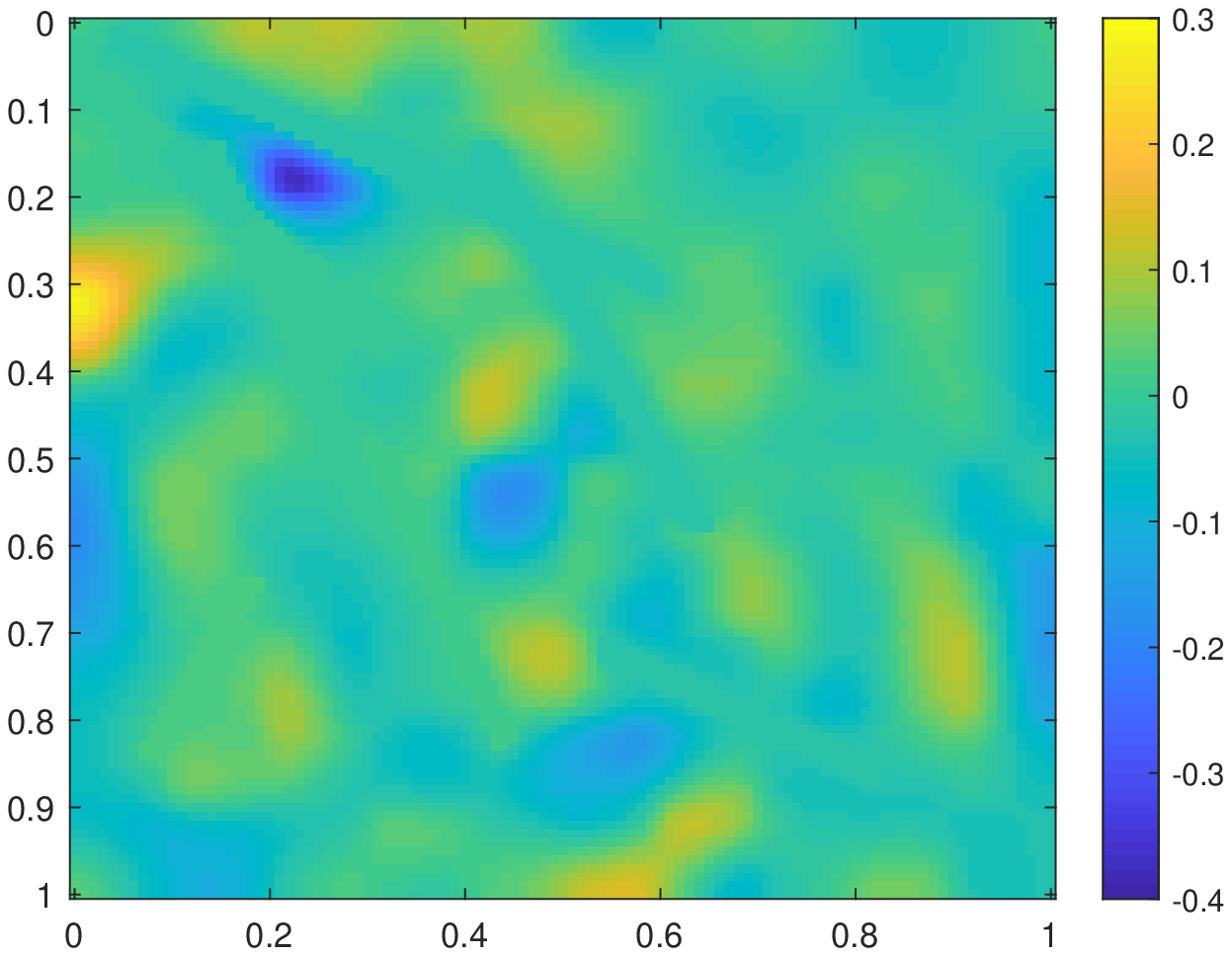}

\includegraphics[scale=0.35]{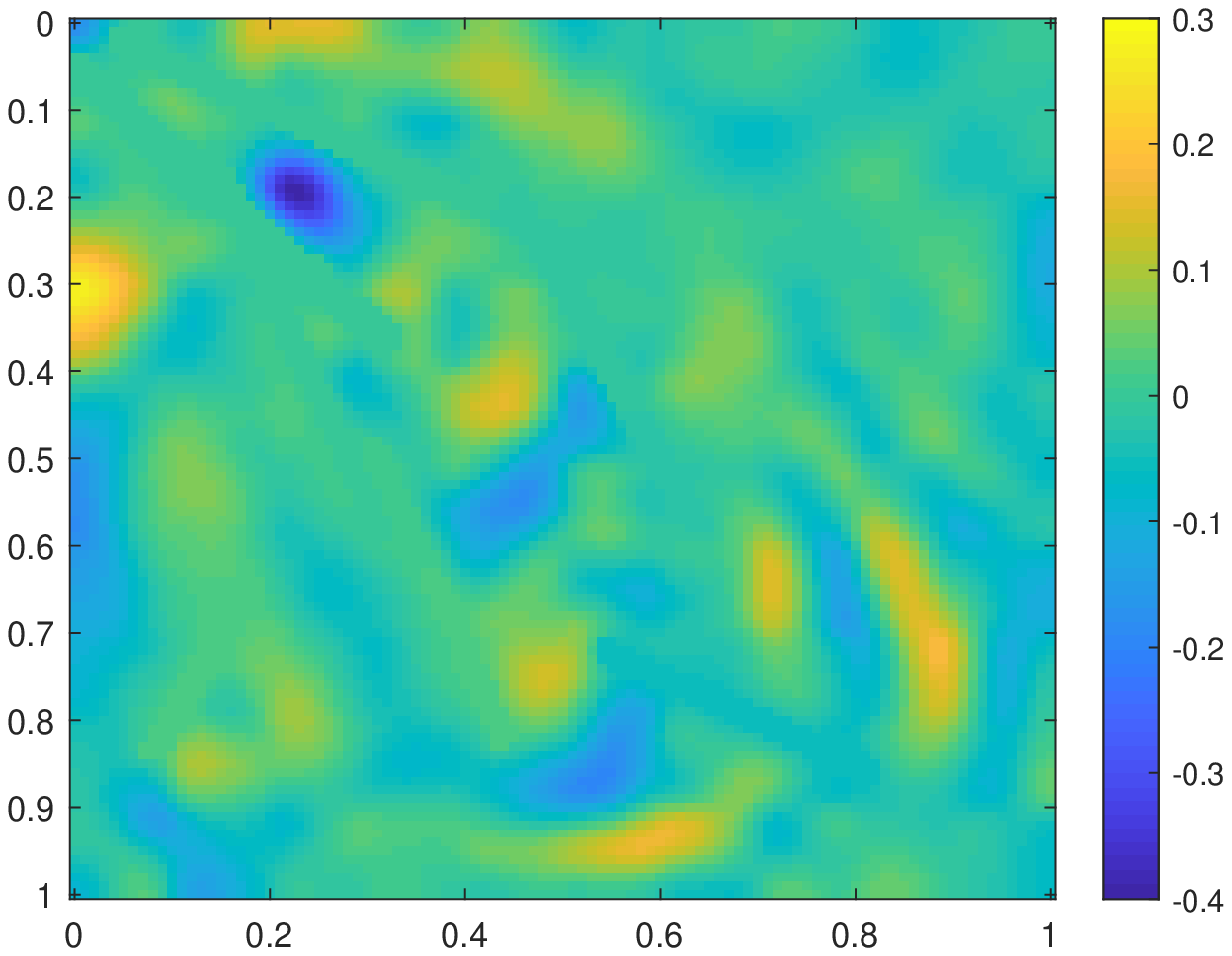} 
\includegraphics[scale=0.35]{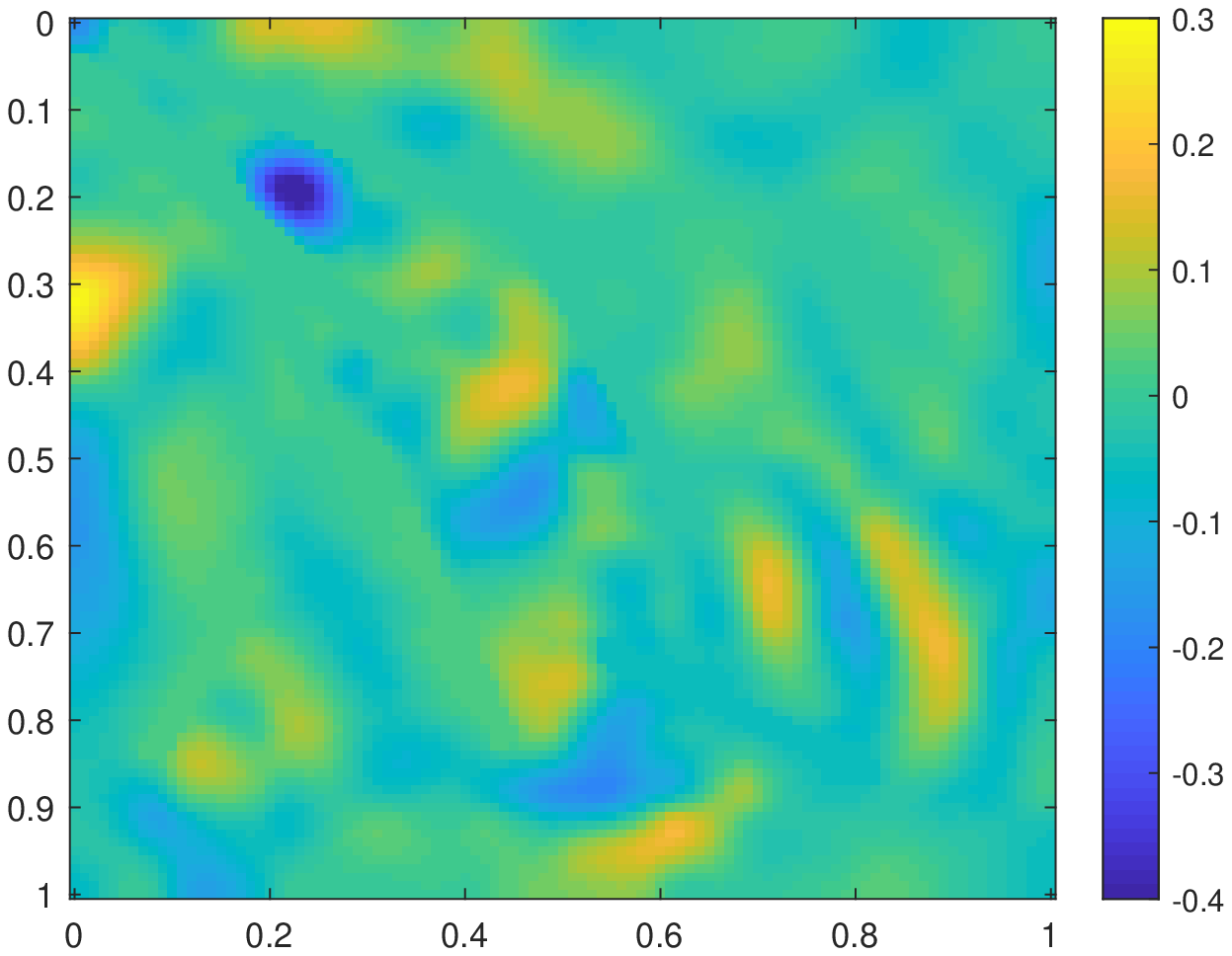}
\caption{Snapshot at $T=0.6$. op-left: reference solution. Top-Right: Implicit CEM-GMsFEM solution. Top-left: Proposed splitting method with additional basis functions. Top-right: Implicit CEM-GMsFEM with additional basis.}
\label{fig:case2high2}
\end{figure}

\begin{figure}[H]
\centering
\includegraphics[scale=0.4]{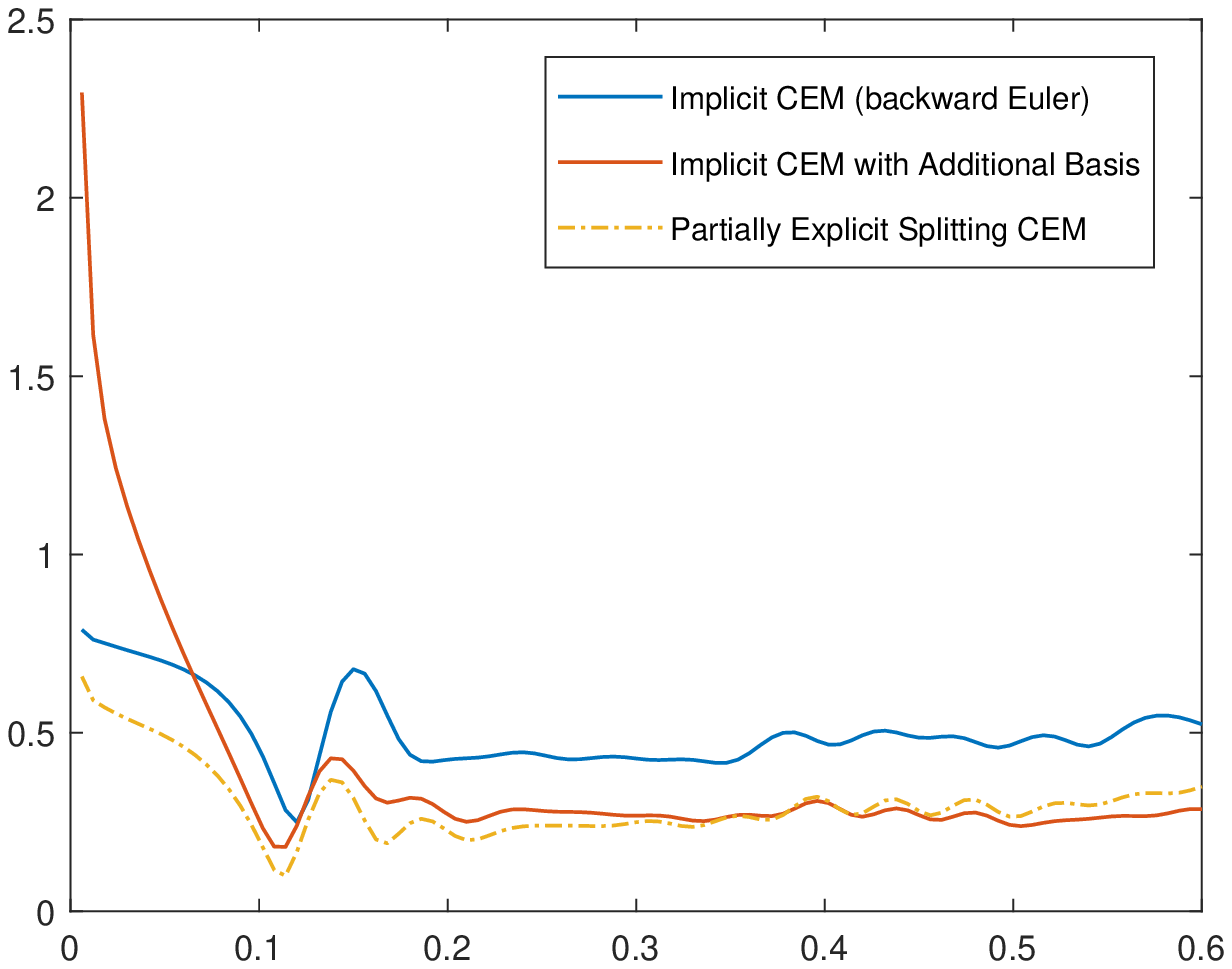} \includegraphics[scale=0.4]{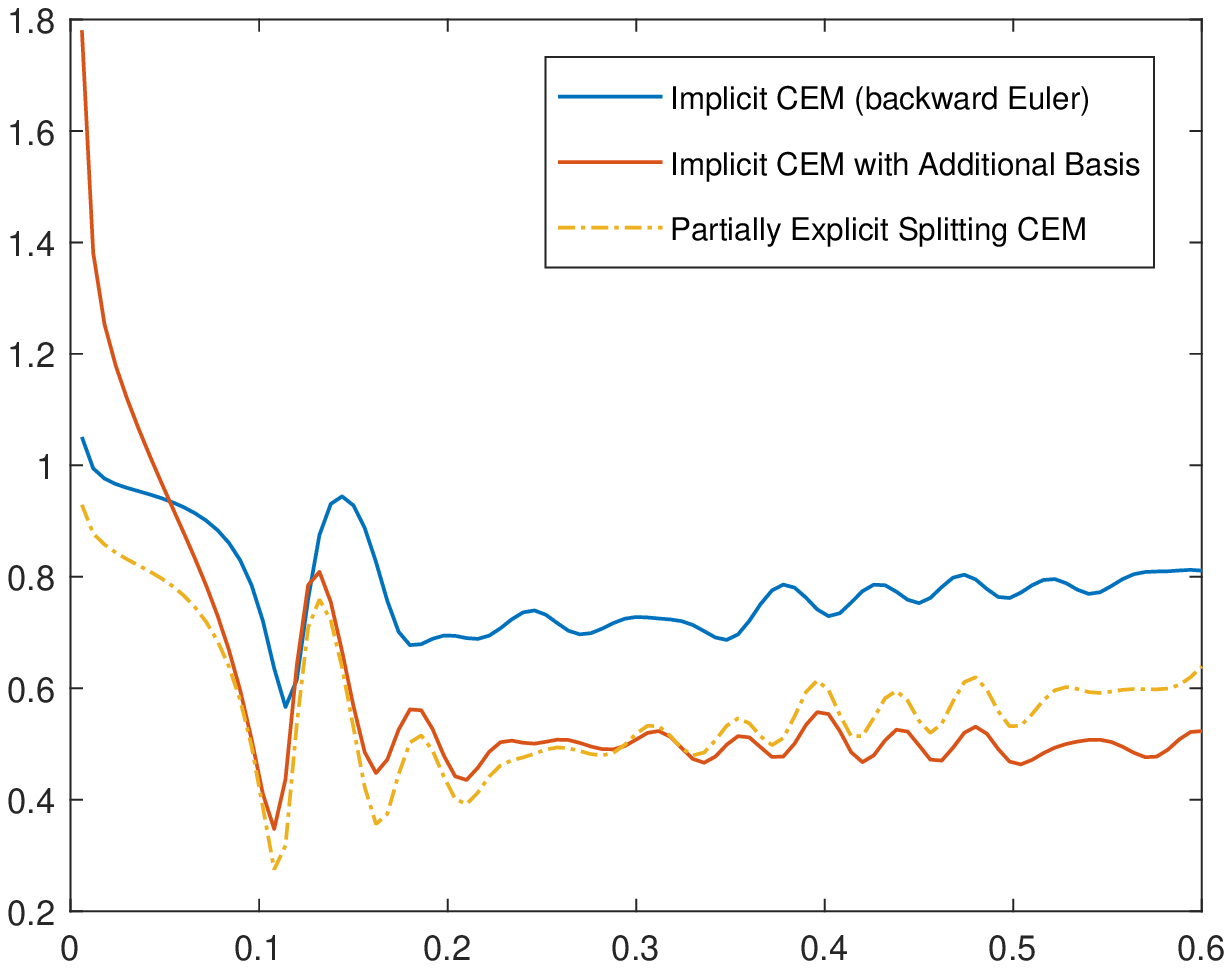}
\caption{Second type of $V_{2,H}$ (CEM Dof: $300$, $V_{2,H}$ Dof: $300$).
Left: $L_{2}$ error. Right: Energy error.}
\label{fig:case2high_error1}
\end{figure}

\begin{figure}[H]
\centering
\includegraphics[scale=0.4]{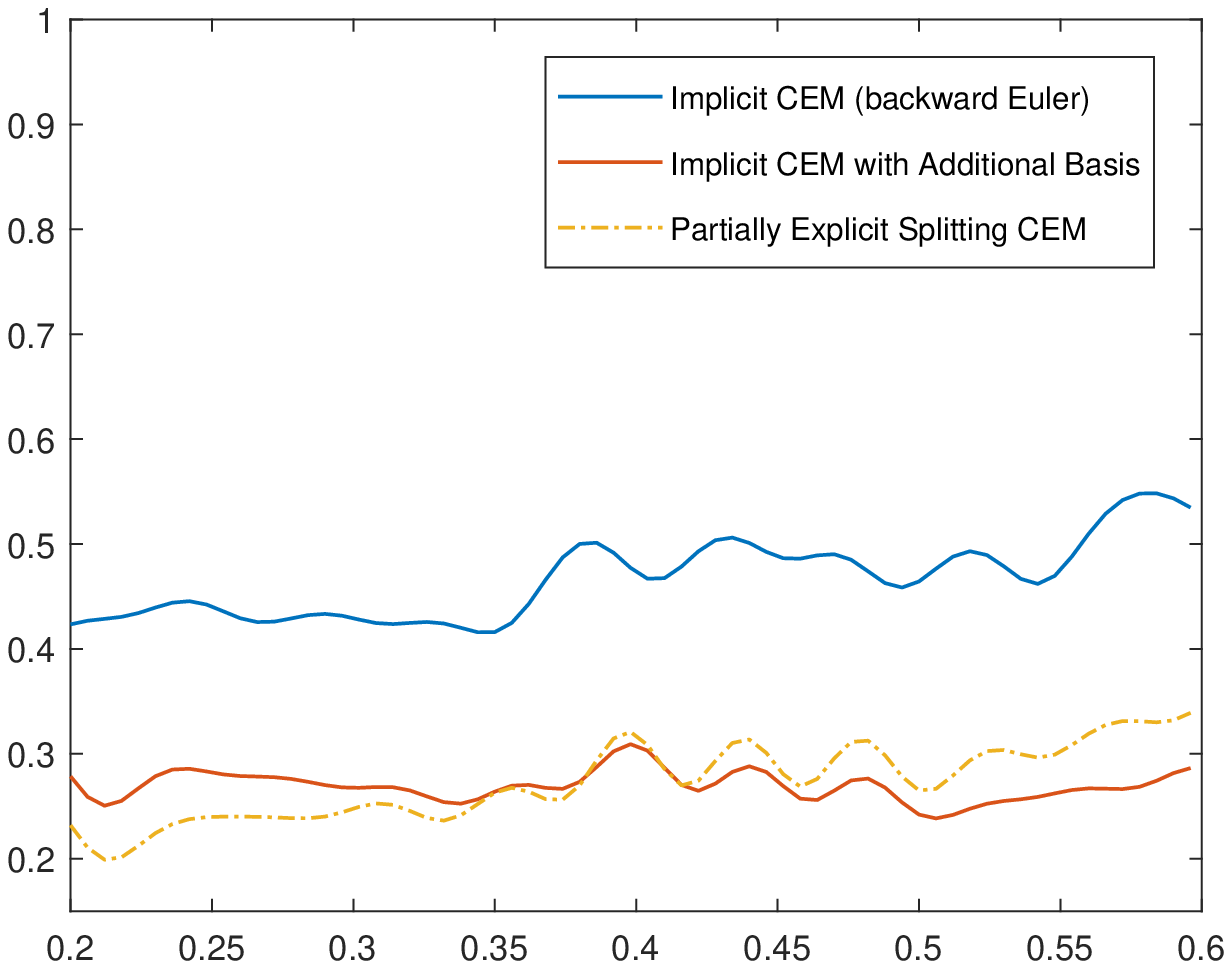} \includegraphics[scale=0.4]{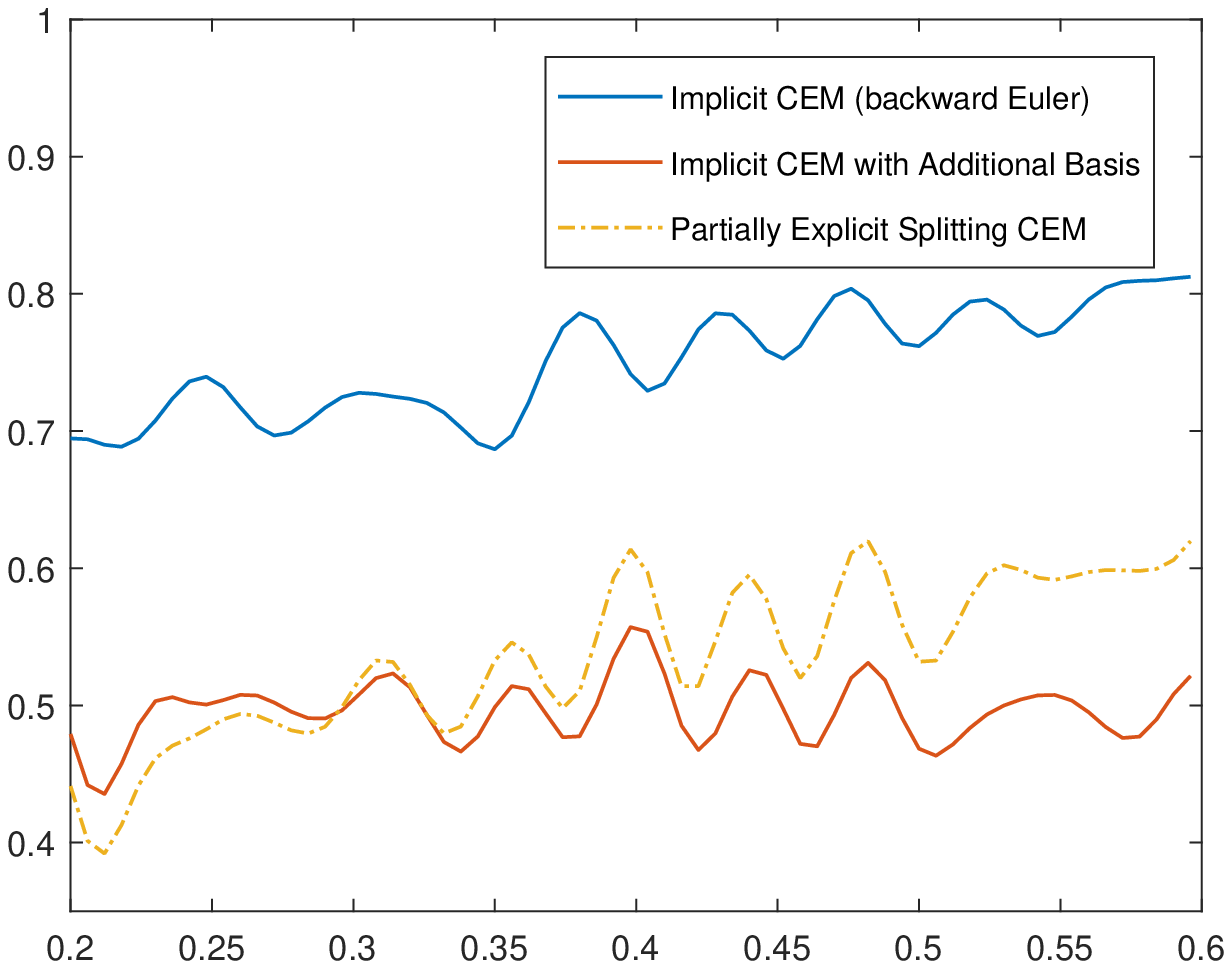}
\caption{The error in the time interval from $0.2$ to $0.6$. Second type of $V_{2,H}$ is used (CEM Dof: $300$, $V_{2,H}$ Dof: $300$).
Left: $L_{2}$ error. Right: Energy error.}
\label{fig:case2high_error2}
\end{figure}

\subsection{Example of a mass lumping}

The proposed methods are still coupled via mass matrix. 
One way to remove this coupling is via mass lumping methods
\cite{cohen2001higher}.
There are many ways to do mass lumping, which can be applied in this paper.
We have tried some mass lumping schemes; 
however, their performance were not optimal.
We believe for mass lumping other discretizations 
(mixed or Discontinuous Galerkin, e.g., \cite{cheung2020explicit}) 
may be more effective.
Here, we consider one example with mass lumping.

In this subsection, we will introduce a way for mass laamping. We
will consider the auxiliary basis functions of $V_{H,1}$ are defined
as following: for each coarse element $K_{i}\in\mathcal{T}_{H}$,
we consider
\[
\psi_{j}^{(i)}=I_{K_{j}^{(i)}}(x)
\]
 where $I_{K_{j}^{(i)}}$ is the characteristic function of $K_{j}^{(i)}\subset K_{i}.$
$K_{j}^{(i)}$ is defined as $K_{1}^{(i)}$ is the region with low
wave speed and $K_{2}^{(i)}$ is the region with high wave speed.
For example, we consider 
\begin{align*}
K_{1}^{(i)} & =\{x\in K_{i}|\;\kappa(x)\leq1\}\\
K_{2}^{(i)} & =\{x\in K_{i}|\;\kappa(x)>1\}.
\end{align*}
For each coarse element $K_{i}$, we can solve the eigenvalue problem
to obtain the auxiliary basis for $V_{2,H}$. We can find $(\xi_{j}^{(i)},\gamma_{j}^{(i)})\in(V(K_{i})\cap\tilde{V})\times\mathbb{R}$,
\begin{align*}
\int_{\omega_{i}}\kappa\nabla\xi_{j}^{(i)}\cdot\nabla v & =\gamma_{j}^{(i)}\int_{\omega_{i}}\xi_{j}^{(i)}v\;\forall v\in V(K_{i})\cap\tilde{V}
\end{align*}
where $\tilde{V}=\{v\in V|\;(v,\psi_{j}^{(i)})=0\;\forall i,j\}.$
The auxiliary space $V_{aux,1}$ and $V_{aux,2}$ are defined as 
\begin{align*}
V_{aux,1} & =\text{span}_{i,j}\{\psi_{j}^{(i)}\},\\
V_{aux,2} & =\text{span}_{i,j}\{\xi_{j}^{(i)}\}.
\end{align*}

The multiscale basis functions $\phi_{j,1}^{(i)}$ of $V_{H,1}$ are
obtained by finding $(\phi_{j,1}^{(i)},\mu_{j,1}^{(i)})\in V_0(K_i^+) \times(V_{aux,1}+V_{aux,2})$
\begin{align*}
a(\phi_{j,1}^{(i)},v)+(\mu_{j,1}^{(i)},v) & =0\;\forall v\in V_0(K_i^+)\\
(\phi_{j,1}^{(i)},\psi_{k}^{(l)}) & =\delta_{il}\delta_{jk}\\
(\phi_{j,1}^{(i)},\xi_{k}^{(l)}) & =0.
\end{align*}
Similarly, the multiscale basis functions $\phi_{j,2}^{(i)}$ of $V_{H,2}$
are obtained by finding $(\phi_{j,2}^{(i)},\mu_{j,2}^{(i)})\in V_0(K_i^+)\times(V_{aux,1}+V_{aux,2})$
\begin{align*}
a(\phi_{j,2}^{(i)},v)+(\mu_{j,2}^{(i)},v) & =0\;\forall v\in V_0(K_i^+)\\
(\phi_{j,2}^{(i)},\psi_{k}^{(l)}) & =0\\
(\phi_{j,2}^{(i)},\xi_{k}^{(l)}) & =\delta_{il}\delta_{jk}.
\end{align*}
The multiscale finite element spaces $V_{H,1}$ and $V_{H,2}$
are defined as 
\begin{align*}
V_{H,1} & =\text{span}_{i,j}\{\phi_{j,1}^{(i)}\},\\
V_{H,2} & =\text{span}_{i,j}\{\phi_{j,2}^{(i)}\}.
\end{align*}
For $u_{H}\in V_{H,1}+V_{H,2}$, we have $u_{H}=\sum_{i,j}u_{j,1}^{(i)}\phi_{j,1}^{(i)}+\sum_{i,j}u_{j,2}^{(i)}\phi_{j,2}^{(i)}$and
\[
a(u,v)=-(\sum_{i,j}u_{j,1}^{(i)}\mu_{j,1}^{(i)},v)-(\sum_{i,j}u_{j,2}^{(i)}\mu_{j,2}^{(i)},v)\;\forall v\in V_{H,1}+V_{H,2}.
\]
Thus, we can consider the weak formulation of $u_{H,tt}=\nabla\cdot(\kappa\nabla u_{H})+f$
as
\[
(\sum_{i,j}u_{j,1}^{(i)}\phi_{j,1}^{(i)}+\sum_{i,j}u_{j,2}^{(i)}\phi_{j,2}^{(i)},w)=(\sum_{i,j}u_{j,1}^{(i)}\mu_{j,1}^{(i)}+\sum_{i,j}u_{j,2}^{(i)}\mu_{j,2}^{(i)}+f,w)\;\forall w\in W_{H}
\]
for some testing space $W_{H}$. We can consider the $W_{H}$ to be
$(V_{aux,1}+V_{aux,2})$ and we have 
\begin{align*}
(\sum_{i,j}(u_{j,1}^{(i)})_{tt}\phi_{j,1}^{(i)}+\sum_{i,j}(u_{j,2}^{(i)})_{tt}\phi_{j,2}^{(i)},\psi_{k}^{(l)}) & =(\sum_{i,j}u_{j,1}^{(i)}\mu_{j,1}^{(i)}+\sum_{i,j}u_{j,2}^{(i)}\mu_{j,2}^{(i)}+f,\psi_{k}^{(l)})\;\forall l,k\\
(\sum_{i,j}(u_{j,1}^{(i)})_{tt}\phi_{j,1}^{(i)}+\sum_{i,j}(u_{j,2}^{(i)})_{tt}\phi_{j,2}^{(i)},\xi_{k}^{(l)}) & =(\sum_{i,j}u_{j,1}^{(i)}\mu_{j,1}^{(i)}+\sum_{i,j}u_{j,2}^{(i)}\mu_{j,2}^{(i)}+f,\xi_{k}^{(l)})\;\forall l,k.
\end{align*}
Since 
\begin{align*}
(\phi_{j,2}^{(i)},\psi_{k}^{(l)}) & =0\\
(\phi_{j,2}^{(i)},\xi_{k}^{(l)}) & =\delta_{il}\delta_{jk},
\end{align*}
and 
\begin{align*}
(\phi_{j,1}^{(i)},\psi_{k}^{(l)}) & =\delta_{il}\delta_{jk}\\
(\phi_{j,1}^{(i)},\xi_{k}^{(l)}) & =0,
\end{align*}
we have 
\begin{align*}
(u_{k,1}^{(l)})_{tt} & =(\sum_{i,j}u_{j,1}^{(i)}\mu_{j,1}^{(i)}+\sum_{i,j}u_{j,2}^{(i)}\mu_{j,2}^{(i)}+f,\psi_{k}^{(l)})\;\forall l,k\\
(u_{k,2}^{(l)})_{tt} & =(\sum_{i,j}u_{j,1}^{(i)}\mu_{j,1}^{(i)}+\sum_{i,j}u_{j,2}^{(i)}\mu_{j,2}^{(i)}+f,\xi_{k}^{(l)})\;\forall l,k.
\end{align*}
Since 
\begin{align*}
(\sum_{i,j}u_{j,1}^{(i)}\mu_{j,1}^{(i)}+\sum_{i,j}u_{j,2}^{(i)}\mu_{j,2}^{(i)}+f,v) & =-a(u_{H},v),
\end{align*}
and 
\begin{align*}
(\sum_{i,j}u_{j,1}^{(i)}\mu_{j,1}^{(i)}+\sum_{i,j}u_{j,2}^{(i)}\mu_{j,2}^{(i)},\psi_{k}^{(l)}) & =(\sum_{i,j}u_{j,1}^{(i)}\mu_{j,1}^{(i)}+\sum_{i,j}u_{j,2}^{(i)}\mu_{j,2}^{(i)},\phi_{k,1}^{(l)})=-a(u_{H},\phi_{k,1}^{(l)}),\\
(\sum_{i,j}u_{j,1}^{(i)}\mu_{j,1}^{(i)}+\sum_{i,j}u_{j,2}^{(i)}\mu_{j,2}^{(i)},\xi_{k}^{(l)}) & =(\sum_{i,j}u_{j,1}^{(i)}\mu_{j,1}^{(i)}+\sum_{i,j}u_{j,2}^{(i)}\mu_{j,2}^{(i)},\phi_{k,2}^{(l)})=-a(u_{H},\phi_{k,2}^{(l)}),
\end{align*}
We have 
\begin{align*}
(u_{k,1}^{(l)})_{tt}+\sum_{i,j}u_{j,1}^{(i)}a(\phi_{j,1}^{(i)},\phi_{k,1}^{(l)})+\sum_{i,j}u_{j,2}^{(i)}a(\phi_{j,2}^{(i)},\phi_{k,1}^{(l)}) & =(f,\psi_{k}^{(l)})\;\forall l,k\\
(u_{k,2}^{(l)})_{tt}+\sum_{i,j}u_{j,1}^{(i)}a(\phi_{j,1}^{(i)},\phi_{k,2}^{(l)})+\sum_{i,j}u_{j,2}^{(i)}a(\phi_{j,2}^{(i)},\phi_{k,2}^{(l)}) & =(f,\xi_{k}^{(l)})\;\forall l,k.
\end{align*}
Using discretization, we can obtain a diagonal mass matrix. We remark
that this can be considered 
as a mass lumping trick for the system (\ref{eq:split12}).
Above, we derived this for global basis functions; 
however, the calculations can be 
localized.
Instead of using the $L_2$ inner product for the mass matrix, we can use $(\pi(u),\pi(v))$ to obtain an approximated mass matrix where $\pi$ is the $L_2$ projection operator from $V_H$ to $V_{aux,1}+V_{aux,2}$. We have used this localized projection idea in our numerical simulations.

Following the idea of the CEM-GMsFEM, insteading of using solving a global problem to obtain the global basis functions, we can solve the problem in an oversampling domain to obtain the CEM basis functions. Since the CEM basis functions are exponentially converging to the global basis functions, we can use the inner product  $(\pi(u),\pi(v))$ for mass lumping of this CEM space.

We will show a numerical example for example case 2 with $f_0=1/2$.

\begin{figure}[H]
\centering
\includegraphics[scale=0.4]{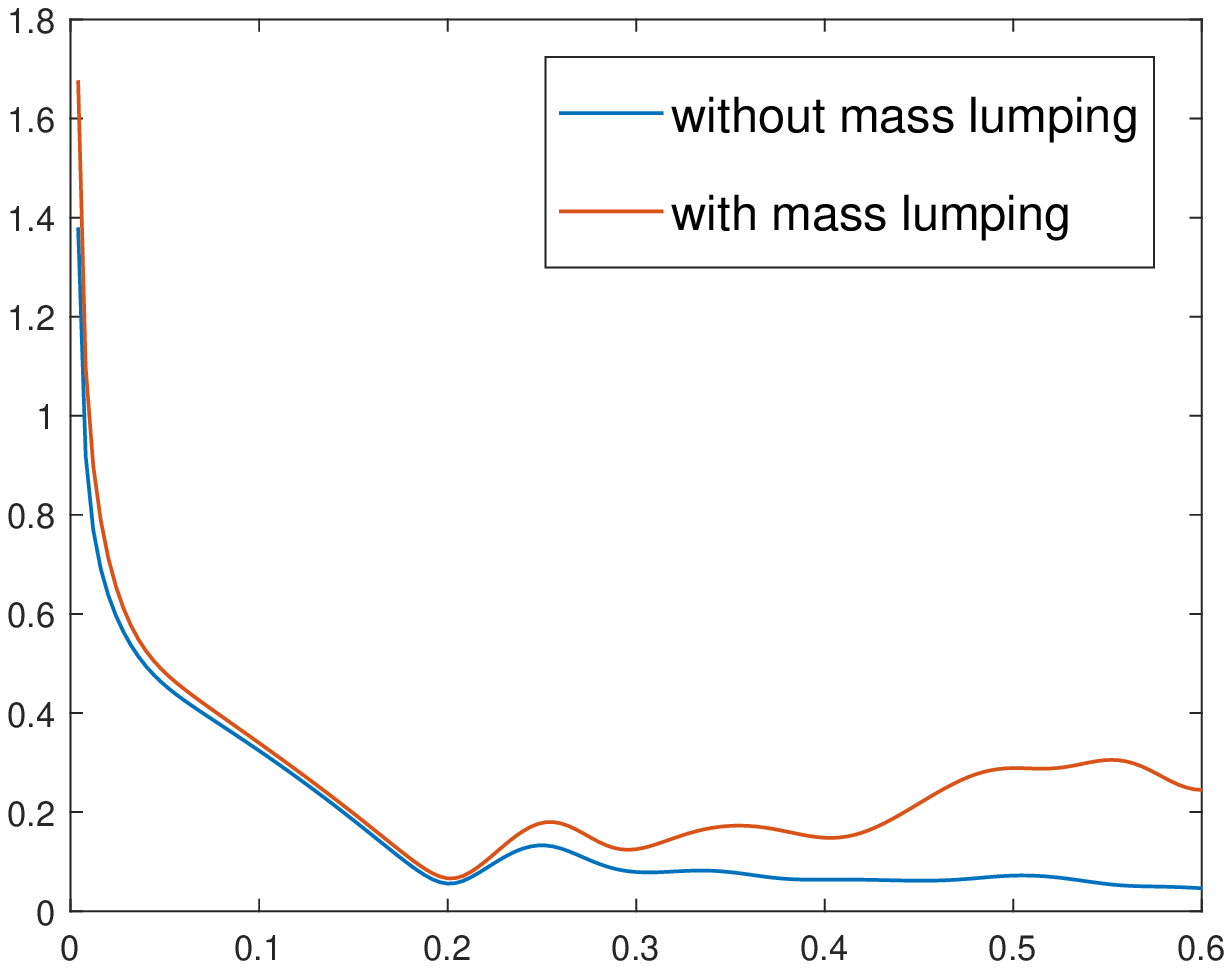} \includegraphics[scale=0.4]{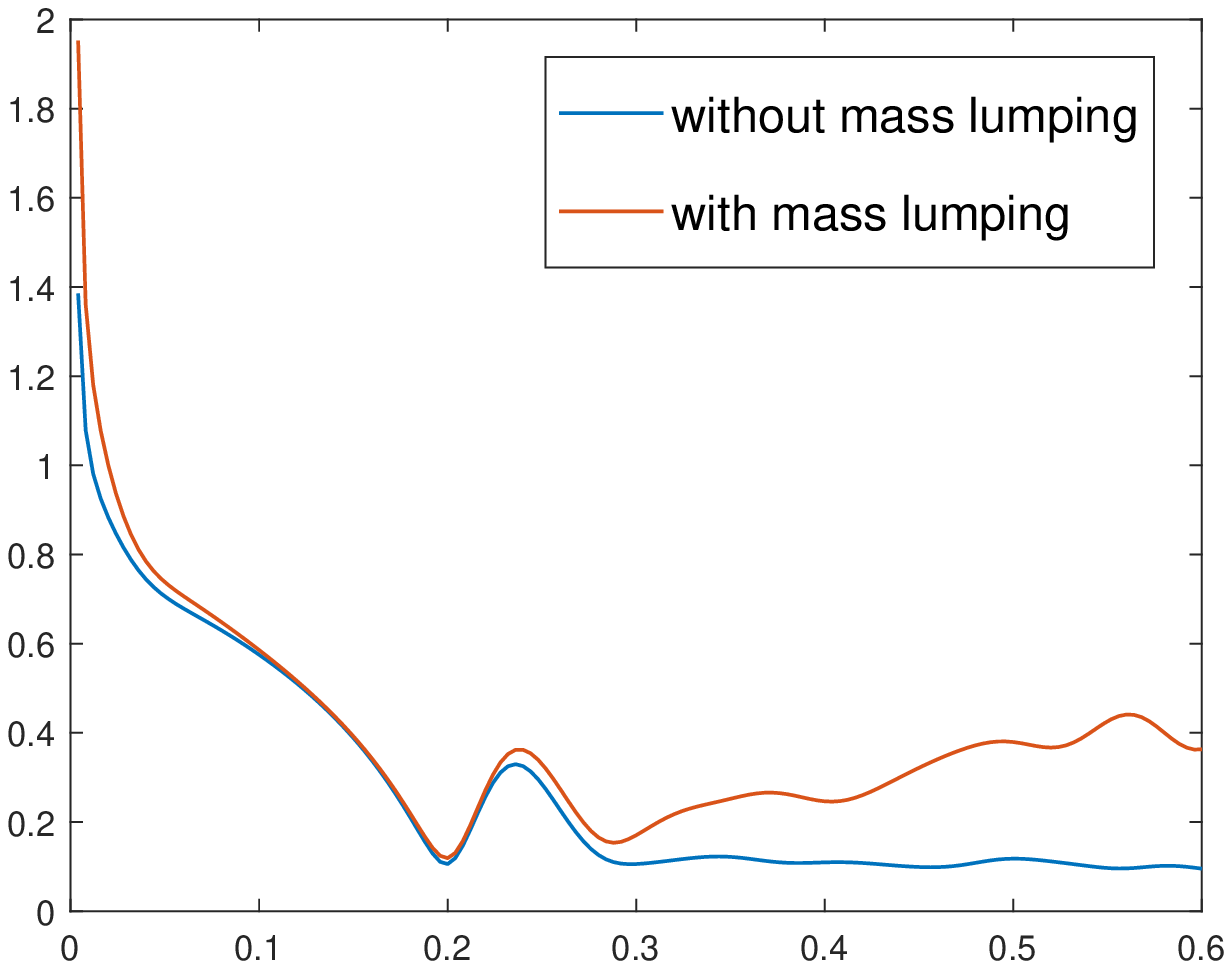}
\caption{Mass lumping result for third type of space and $\tau= 0.004$. (CEM Dof: $127$, $V_{2,H}$ Dof: $500$).
Left: $L_{2}$ error , Right: Energy error}
\end{figure}

\begin{figure}[H]
\centering

\includegraphics[scale=0.4]{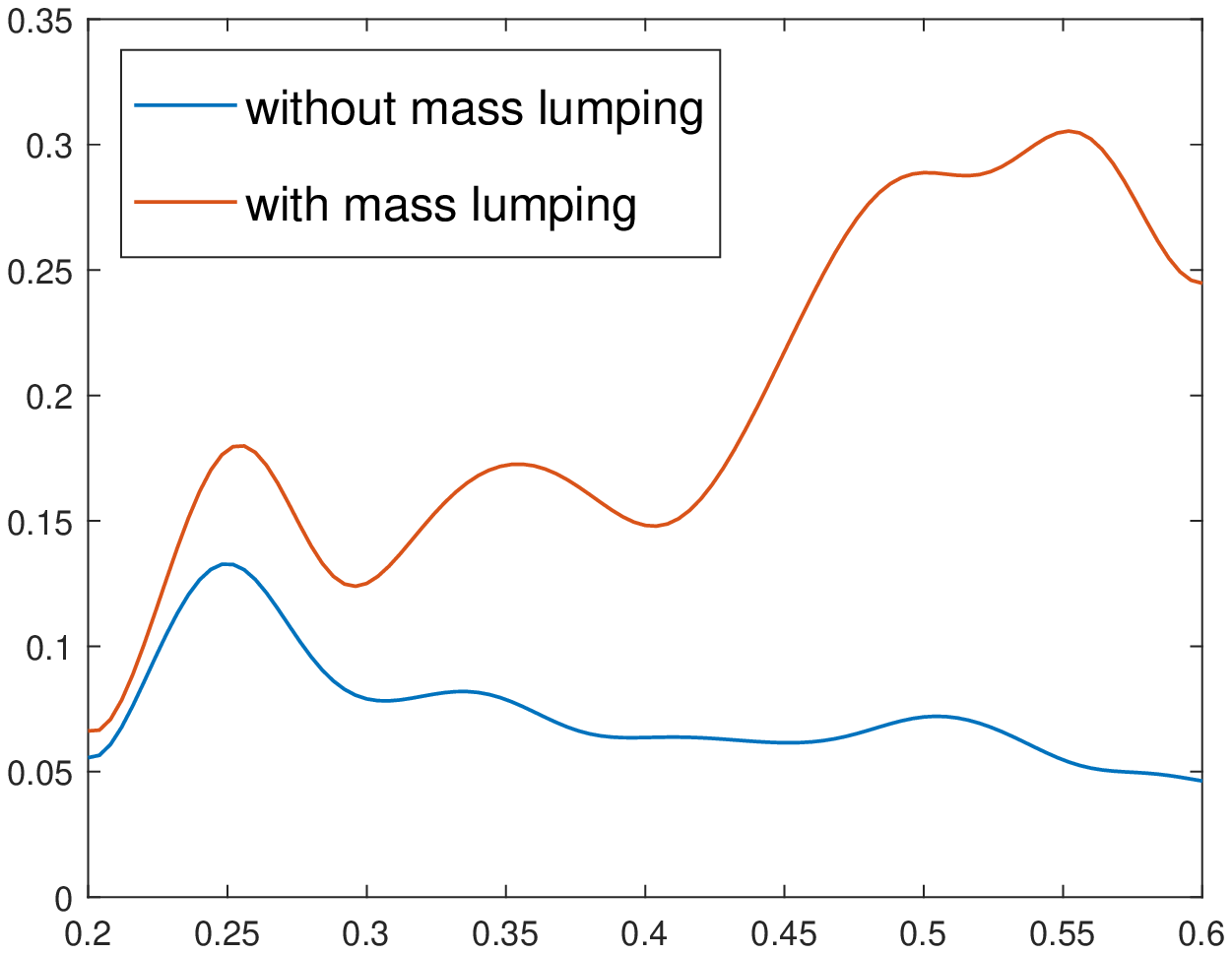} \includegraphics[scale=0.4]{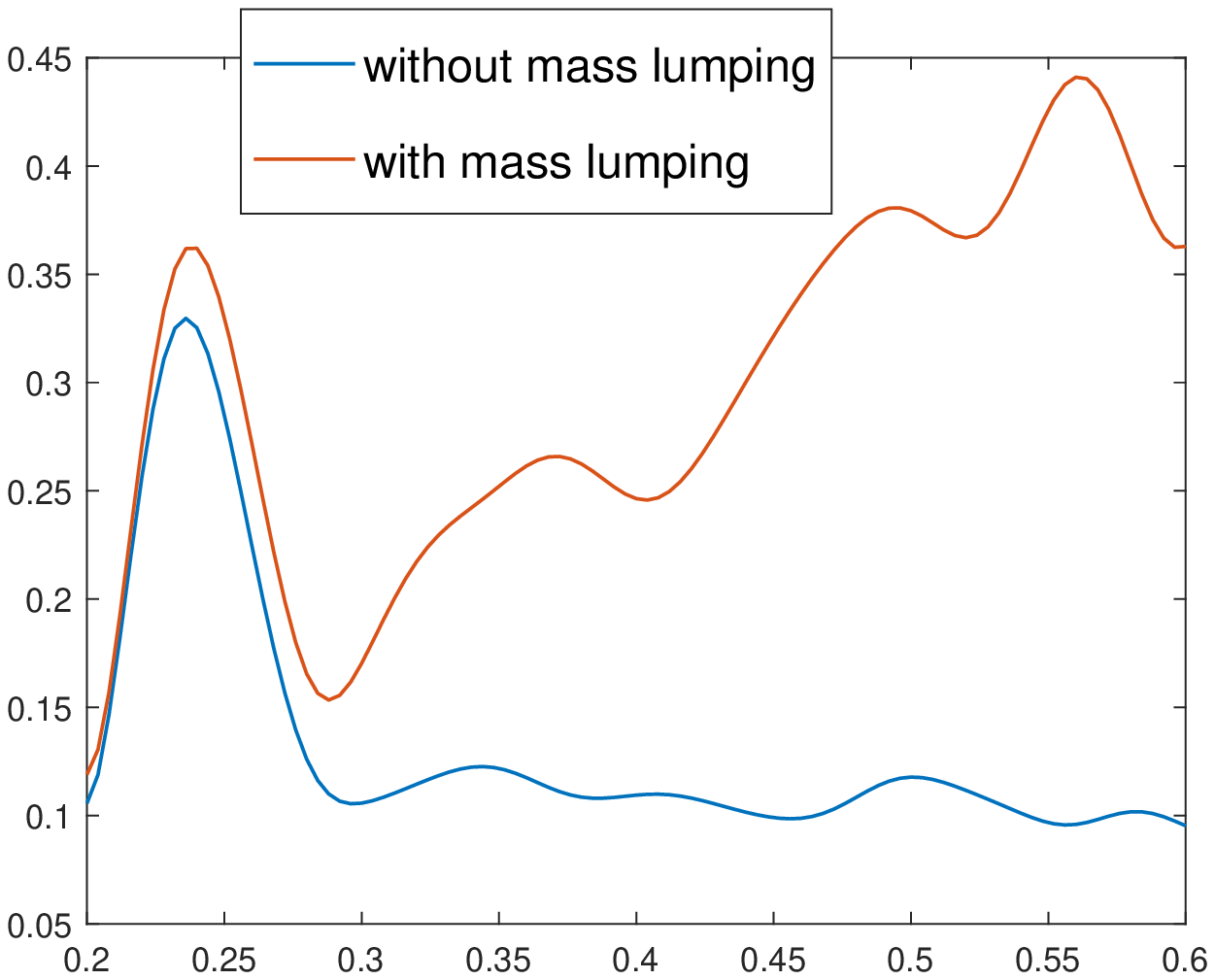}

\caption{Mass lumping result for third type of space and $\tau= 0.004$. Time from $.2$ to $0.6$. (CEM Dof: $127$, $V_{2,H}$ Dof: $500$).
Left: $L_{2}$ error , Right: Energy error}
\end{figure}

\section{Conclusions}

In this paper, we design contrast-independent partially explicit time discretization methods for wave equations. The proposed methods differ from our previous works, where we first introduced contrast-independent partial explicit methods for parabolic equations \cite{chung_partial_expliict21}. The proposed approach uses temporal splitting based on spatial multiscale splitting. We first introduce two spatial spaces, first account for spatial features related to fast time scales and the second for spatial features related to ``slow'' time scales. Using these spaces, we propose time splitting, where the first equation solves for fast components implicitly and the second equation solves for slow components explicitly. Our proposed method is still implicit via mass matrix; however, it is explicit in terms of stiffness matrix for the slow component (which is contrast independent).  Via mass lumping, one can remove the coupling, which is briefly discussed in the paper. We show a stability of the proposed splitting under suitable conditions for the second space, where the fast components are absent. We present numerical results, which show that the proposed methods provide very similar results as fully implicit methods using explicit methods with the time stepping that is independent of the contrast.

\appendix\section{Proof for $\omega=0$}
\label{sec:appendix}

For the case $\omega=0$, we only consider $V_{H,1}\perp V_{H,2}$. 
The scheme for $\omega=0$ has the form
\begin{align*}
(u_{H,1}^{n+1}-2u_{H,1}^{n}+u_{H,1}^{n-1},w)+\cfrac{\tau^{2}}{2}a(u_{H,1}^{n+1}+u_{H,1}^{n-1}+2u_{H,2}^{n},w) & =0\;\forall w\in V_{H,1}\\
(u_{H,2}^{n+1}-2u_{H,2}^{n}+u_{H,2}^{n-1},w)+\cfrac{\tau^{2}}{2}a(u_{H,1}^{n+1}+u_{H,1}^{n-1}+2u_{H,2}^{n},w) & =0\;\forall w\in V_{H,2}.
\end{align*}
We define an inner product $(\cdot,\cdot)_{m_{\tau}}$ such that 
\[
(u,v)_{m_{\tau}}=(u,v)+\cfrac{\tau^{2}}{2}a(u,v).
\]
We then define two operators $b_{\tau}:V_{H,1}\mapsto V_{H,1}$, $c_{\tau}:V_{H,1}+V_{H,2}\mapsto V_{H,1}$
such that 
\begin{align*}
(b_{\tau}(v_{1}),v)_{m_{\tau}} & =(v_{1},v)\;\forall v\in V_{H,1},\\
(c_{\tau}(u),v)_{m_{\tau}} & :=\frac{\tau^{2}}{2}a(u,v)\;\forall v\in V_{H,1},\\
a(d_{\tau}(u),v) & :=a(u,v)\;\forall v\in V_{H,1}.
\end{align*}
We remark that 
\[
(c_{\tau}(v_{2}),v)_{m_{\tau}}=(v_{2},v)_{m_{\tau}}\;\forall v\in V_{H,1},v_{2}\in V_{H,2},
\]
and 
\begin{align*}
\|b_{\tau}(v_{1})\|_{m_{\tau}}^{2} & =(v_{1},b_{\tau}(v_{1}))\leq\|v_{1}\|\|b_{\tau}(v_{1})\|\\
 & \leq\|v_{1}\|\|b_{\tau}(v_{1})\|.
\end{align*}
Equality holds if and only if $v_{1}$ is a constant function.

We define an inner product $(\cdot,\cdot)_{s_{\tau}}$ such that 
\[
(v_{1},v)_{s_{\tau}}=(v_{1},v)-(b_{\tau}(v_{1}),b_{\tau}(v))_{m_{\tau}}\;\text{for }v_{1},v\in V_{H,1}.
\]
We then define three (semi) norms $\|\cdot\|_{m_{\tau}},\|\cdot\|_{s_{\tau}},\|\cdot\|_{n_{\tau}}$
such that 
\[
\|v\|_{m_{\tau}}^{2}=(v,v)_{m_{\tau}},
\]
\[
\|v_{1}\|_{s_{\tau}}^{2}=(v_{1},v_{1})-\|b_{\tau}(v_{1})\|_{m_{\tau}}^{2},
\]
\[
\|v_{2}\|_{n_{\tau}}^{2}=\cfrac{\tau^{2}}{2}\|v_{2}\|_{a}^{2}-\|c_{\tau}(v_{2})\|_{m_{\tau}}^{2}-\|d_{\tau}(v_{2})\|_{s_{\tau}}^{2}.
\]

It is clear that $\|\cdot\|_{m_{\tau}},\|\cdot\|_{s_{\tau}}$ are
(semi) norms. To show that $\|\cdot\|_{n_{\tau}}$ is a norm in $V_{H,2}$,
we only need to check 
\[
\|v_{2}\|_{n_{\tau}}^{2}>0\;\forall v_{2}\in V_{H,2}.
\]

Since $a(d_{\tau}(v_{2}),v_{2}-d_{\tau}(v_{2}))=0$, we have 
\[
\cfrac{\tau^{2}}{2}\|v_{2}\|_{a}^{2}=\cfrac{\tau^{2}}{2}\|v_{2}-d_{\tau}(v_{2})\|_{a}^{2}+\cfrac{\tau^{2}}{2}\|d_{\tau}(v_{2})\|_{a}^{2}.
\]
We then have
\begin{align*}
\cfrac{\tau^{2}}{2}\|d_{\tau}(v_{2})\|_{a}^{2}-\|d_{\tau}(v_{2})\|_{s_{\tau}}^{2} & =\cfrac{\tau^{2}}{2}\|d_{\tau}(v_{2})\|_{a}^{2}-(d_{\tau}(v_{2}),d_{\tau}(v_{2}))+(b_{\tau}(d_{\tau}(v_{2})),b_{\tau}(d_{\tau}(v_{2})))_{m_{\tau}} = \\
 & \cfrac{\tau^{2}}{2}\|d_{\tau}(v_{2})\|_{a}^{2}-(d_{\tau}(v_{2}),d_{\tau}(v_{2}))+(b_{\tau}(d_{\tau}(v_{2})),d_{\tau}(v_{2})).
\end{align*}
Since 
\begin{align*}
(b_{\tau}(d_{\tau}(v_{2})),d_{\tau}(v_{2}))-(d_{\tau}(v_{2}),d_{\tau}(v_{2})) & =-\cfrac{\tau^{2}}{2}a(b_{\tau}(d_{\tau}(v_{2})),d_{\tau}(v_{2}))+(b_{\tau}(d_{\tau}(v_{2})),d_{\tau}(v_{2}))_{m_{\tau}}-(d_{\tau}(v_{2}),d_{\tau}(v_{2})) = \\
 & -\cfrac{\tau^{2}}{2}a(b_{\tau}(d_{\tau}(v_{2})),d_{\tau}(v_{2})),
\end{align*}
we have 
\[
\cfrac{\tau^{2}}{2}\|d_{\tau}(v_{2})\|_{a}^{2}-\|d_{\tau}(v_{2})\|_{s_{\tau}}^{2}=\cfrac{\tau^{2}}{2}a(d_{\tau}(v_{2})-b_{\tau}(d_{\tau}(v_{2})),d_{\tau}(v_{2})).
\]

Since 
\begin{align*}
(c_{\tau}(v_{2}),c_{\tau}(v_{2}))_{m_{\tau}} & =\cfrac{\tau^{2}}{2}a(c_{\tau}(v_{2}),v_{2})=\cfrac{\tau^{2}}{2}a(c_{\tau}(v_{2}),d_{\tau}(v_{2}))=\\
 & (c_{\tau}(v_{2}),d_{\tau}(v_{2}))_{m_{\tau}}-(c_{\tau}(v_{2}),d_{\tau}(v_{2})) = \\
 & (c_{\tau}(v_{2}),d_{\tau}(v_{2})-b_{\tau}d_{\tau}(v_{2}))_{m_{\tau}} = \\
 & \cfrac{\tau^{2}}{2}a(v_{2},d_{\tau}(v_{2})-b_{\tau}d_{\tau}(v_{2}))=\cfrac{\tau^{2}}{2}a(d_{\tau}(v_{2})-b_{\tau}(d_{\tau}(v_{2})),d_{\tau}(v_{2})),
\end{align*}
we have 
\[
\cfrac{\tau^{2}}{2}\|d_{\tau}(v_{2})\|_{a}^{2}-\|d_{\tau}(v_{2})\|_{s_{\tau}}^{2}-(c_{\tau}(v_{2}),c_{\tau}(v_{2}))_{m_{\tau}}=0.
\]
Therefore, we have 
\[
\|v_{2}\|_{n_{\tau}}^{2}=\cfrac{\tau^{2}}{2}\|v_{2}-d_{\tau}(v_{2})\|_{a}^{2}>0.
\]

\begin{lemma}
For $u_{1}\in V_{H,1}$ and $u_{2}\in V_{H,1}$, we have 
\begin{align*}
 & \|u_{1}\|_{2}^{2}+\cfrac{\tau^{2}}{2}\|u_{2}\|_{a}^{2}-\|b_{\tau}(u_{1})+c_{\tau}(u_{2})\|_{m_{\tau}}^{2} = \\
 & \|u_{1}+d_{\tau}(u_{2})\|_{s_{\tau}}^{2}+\|u_{2}\|_{n_{\tau}}^{2}.
\end{align*}
\end{lemma}

\begin{proof}
We have
\begin{align*}
 & \|u_{1}\|_{2}^{2}+\cfrac{\tau^{2}}{2}\|u_{2}\|_{a}^{2}-\|b_{\tau}(u_{1})-c_{\tau}(u_{2})\|_{m_{\tau}}^{2}=\\
 & \|u_{1}\|_{2}^{2}-\|b_{\tau}(u_{1})\|_{m_{\tau}}^{2}+2(b_{\tau}(u_{1}),c_{\tau}(u_{2}))_{m_{\tau}}+\cfrac{\tau^{2}}{2}\|u_{2}\|_{a}^{2}-\|c_{\tau}(u_{2})\|_{m_{\tau}}^{2}.
\end{align*}
By definition of $(\cdot,\cdot)_{s_{\tau}}$, we have 
\[
\|u_{1}\|_{s_{\tau}}^{2}=\|u_{1}\|_{2}^{2}-\|b_{\tau}(u_{1})\|_{m_{\tau}}^{2}.
\]
\begin{align*}
(u_{1},d_{\tau}(u_{2}))_{s_{\tau}} & =(u_{1},d_{\tau}(u_{2}))-(b_{\tau}u_{1},b_{\tau}d_{\tau}(u_{2}))_{m_{\tau}}=\\
 & (u_{1},d_{\tau}(u_{2}))-(b_{\tau}u_{1},d_{\tau}(u_{2}))=\\
 & (u_{1},d_{\tau}(u_{2}))-(b_{\tau}u_{1},d_{\tau}(u_{2}))_{m_{\tau}}+\cfrac{\tau^{2}}{2}a(b_{\tau}u_{1},d_{\tau}(u_{2})).
\end{align*}
Since $(u_{1},d_{\tau}(u_{2}))=(b_{\tau}u_{1},d_{\tau}(u_{2}))_{m_{\tau}}$,
we have 
\begin{align*}
(u_{1},d_{\tau}(u_{2}))_{s_{\tau}} & =\cfrac{\tau^{2}}{2}a(b_{\tau}u_{1},d_{\tau}(u_{2}))=\cfrac{\tau^{2}}{2}a(b_{\tau}u_{1},u_{2}) = \\
 & (b_{\tau}(u_{1}),c_{\tau}(u_{2}))_{m_{\tau}}.
\end{align*}
Thus, we have 
\begin{align*}
 & \|u_{1}\|_{2}^{2}+\cfrac{\tau^{2}}{2}\|u_{2}\|_{a}^{2}-\|b_{\tau}(u_{1})-c_{\tau}(u_{2})\|_{m_{\tau}}^{2} = \\
 & \|u_{1}\|_{s_{\tau}}^{2}+2(u_{1},d_{\tau}(u_{2}))_{s_{\tau}}+\cfrac{\tau^{2}}{2}\|u_{2}\|_{a}^{2}-\|c_{\tau}(u_{2})\|_{m_{\tau}}^{2} = \\
 & \|u_{1}+d_{\tau}(u_{2})\|_{s_{\tau}}^{2}+\cfrac{\tau^{2}}{2}\|u_{2}\|_{a}^{2}-\|c_{\tau}(u_{2})\|_{m_{\tau}}^{2}-\|d_{\tau}(u_{2})\|_{s_{\tau}}^{2} = \\
 & \|u_{1}+d_{\tau}(u_{2})\|_{s_{\tau}}^{2}+\|u_{2}\|_{n_{\tau}}^{2}.
\end{align*}
\end{proof}
We then define the energy $E^{n+\frac{1}{2}}$ by
\begin{align*}
E^{n+\frac{1}{2}}:= & \|b(u_{H,1}^{n+1})-c(u_{H,2}^{n+1})-b(u_{H,1}^{n})+c(u_{H,2}^{n})\|_{m_{\tau}}^{2}+\|u_{H,2}^{n+1}-u_{H,2}^{n}\|_{2}^{2}-\cfrac{\tau^{2}}{2}\|u_{H,2}^{n+1}-u_{H,2}^{n}\|_{a}^{2} + \\
 & \|u_{H,1}^{n+1}+d_{\tau}(u_{H,2}^{n+1})\|_{s_{\tau}}^{2}+\|u_{H,2}^{n+1}\|_{n_{\tau}}^{2}+\|u_{H,1}^{n}+d_{\tau}(u_{H,2}^{n})\|_{s_{\tau}}^{2}+\|u_{H,2}^{n}\|_{n_{\tau}}^{2}.
\end{align*}

\begin{theorem}
If $V_{H,1}\perp V_{H,2}$, $\omega=0$, we have 
\[
E^{n+\frac{1}{2}}=E^{n-\frac{1}{2}}
\]
and the scheme is stable if 
\[
\sup_{v\in V_{H,2}}\cfrac{\|v\|_{a}^{2}}{\|v\|_{}^{2}}\leq\cfrac{\tau^{2}}{2}.
\]
 
\end{theorem}
\begin{proof}
We consider the test functions $b_{\tau}(u_{H,1}^{n+1}-u_{H,1}^{n-1})-c_{\tau}(u_{H,1}^{n+1}-u_{H,1}^{n-1})$
and $u_{H,2}^{n+1}-u_{H,2}^{n-1}$. We have
\begin{align*}
(u_{H,1}^{n+1}-2u_{H,1}^{n}+u_{H,1}^{n-1},b_{\tau}(u_{H,1}^{n+1}-u_{H,1}^{n-1})-c_{\tau}(u_{H,2}^{n+1}-u_{H,2}^{n-1})) + & \\
\cfrac{\tau^{2}}{2}a(u_{H,1}^{n+1}+u_{H,1}^{n-1}+2u_{H,2}^{n},b_{\tau}(u_{H,1}^{n+1}-u_{H,1}^{n-1})-c_{\tau}(u_{H,2}^{n+1}-u_{H,2}^{n-1})) & =0\\
(u_{H,2}^{n+1}-2u_{H,2}^{n}+u_{H,2}^{n-1},u_{H,2}^{n+1}-u_{H,2}^{n-1})+\cfrac{\tau^{2}}{2}a(u_{H,1}^{n+1}+u_{H,1}^{n-1}+2u_{H,2}^{n},u_{H,2}^{n+1}-u_{H,2}^{n-1}) & =0.
\end{align*}
We consider 
\begin{align*}
B_{1} & =(u_{H,1}^{n+1}+u_{H,1}^{n-1},b_{\tau}(u_{H,1}^{n+1}-u_{H,1}^{n-1}))+\cfrac{\tau^{2}}{2}a(u_{H,1}^{n+1}+u_{H,1}^{n-1},b_{\tau}(u_{H,1}^{n+1}-u_{H,1}^{n-1})),
\end{align*}
\[
B_{2}=\tau^{2}a(u_{H,2}^{n},b_{\tau}(u_{H,1}^{n+1}-u_{H,1}^{n-1}))=2(c_{\tau}(u_{H,2}^{n}),b_{\tau}(u_{H,1}^{n+1}-u_{H,1}^{n-1}))_{m_{\tau}},
\]
\[
B_{3}=-\tau^{2}a(u_{H,2}^{n},c_{\tau}(u_{H,2}^{n+1}-u_{H,2}^{n-1}))=-2(c_{\tau}(u_{H,2}^{n}),c_{\tau}(u_{H,2}^{n+1}-u_{H,2}^{n-1}))_{m_{\tau}},
\]
\[
B_{4}=-2(u_{H,1}^{n},b_{\tau}(u_{H,1}^{n+1}-u_{H,1}^{n-1}))=-2(b_{\tau}(u_{H,1}^{n}),b_{\tau}(u_{H,1}^{n+1}-u_{H,1}^{n-1}))_{m_{\tau}},
\]
\[
B_{5}=2(u_{H,1}^{n},c_{\tau}(u_{H,2}^{n+1}-u_{H,2}^{n-1}))=2(b_{\tau}(u_{H,1}^{n}),c_{\tau}(u_{H,2}^{n+1}-u_{H,2}^{n-1}))_{m_{\tau}},
\]
\[
B_{6}=-(u_{H,1}^{n+1}+u_{H,1}^{n-1},c_{\tau}(u_{H,2}^{n+1}-u_{H,2}^{n-1})),
\]
\[
B_{7}=-\cfrac{\tau^{2}}{2}a(u_{H,1}^{n+1}+u_{H,1}^{n-1},c_{\tau}(u_{H,2}^{n+1}-u_{H,2}^{n-1})),
\]
\[
C_{1}=(u_{H,2}^{n+1}-2u_{H,2}^{n}+u_{H,2}^{n-1},u_{H,2}^{n+1}-u_{H,2}^{n-1})+\tau^{2}a(u_{H,2}^{n},u_{H,2}^{n+1}-u_{H,2}^{n-1}),
\]
\[
C_{2}=\cfrac{\tau^{2}}{2}a(u_{H,1}^{n+1}+u_{H,1}^{n-1},u_{H,2}^{n+1}-u_{H,2}^{n-1}),
\]
such that 
\[
\sum_{i=1}^{7}B_{i}=\sum_{i=1}^{2}C_{i}=0.
\]
By the definition of $B_{2},B_{3},B_{4},B_{5}$, we have 
\[
B_{2}+B_{3}+B_{4}+B_{5}=-2(b_{\tau}(u_{H,1}^{n+1})-c_{\tau}(u_{H,2}^{n+1}),b_{\tau}(u_{H,1}^{n})-c_{\tau}(u_{H,2}^{n}))_{m_{\tau}}.
\]

Since $(u,v)_{m_{\tau}}=(u,v)+\cfrac{\tau^{2}}{2}a(u,v)$, we have
\begin{align*}
B_{1}= & (u_{H,1}^{n+1}+u_{H,1}^{n-1},b_{\tau}(u_{H,1}^{n+1}-u_{H,1}^{n-1}))+\cfrac{\tau^{2}}{2}a(u_{H,1}^{n+1}+u_{H,1}^{n-1},b_{\tau}(u_{H,1}^{n+1}-u_{H,1}^{n-1})) = \\
& (u_{H,1}^{n+1}+u_{H,1}^{n-1},b_{\tau}(u_{H,1}^{n+1}-u_{H,1}^{n-1}))_{m_{\tau}}=(u_{H,1}^{n+1}+u_{H,1}^{n-1},u_{H,1}^{n+1}-u_{H,1}^{n-1}) = \\
 & \|u_{H,1}^{n+1}\|_{2}^{2}+\|u_{H,1}^{n}\|_{2}^{2}-\|u_{H,1}^{n}\|_{2}^{2}+\|u_{H,1}^{n-1}\|_{2}^{2}.
\end{align*}
\begin{align*}
C_{1}= & (u_{H,2}^{n+1}-2u_{H,2}^{n}+u_{H,2}^{n-1},u_{H,2}^{n+1}-u_{H,2}^{n-1})+\tau^{2}a(u_{H,2}^{n},u_{H,2}^{n+1}-u_{H,2}^{n-1}) = \\
 & \|u_{H,2}^{n+1}-u_{H,2}^{n}\|_{2}^{2}+\cfrac{\tau^{2}}{2}\|u_{H,2}^{n+1}\|_{a}^{2}+\cfrac{\tau^{2}}{2}\|u_{H,2}^{n}\|_{a}^{2}-\cfrac{\tau^{2}}{2}\|u_{H,2}^{n+1}-u_{H,2}^{n}\|_{a}^{2}- \\
 & \Big(\|u_{H,2}^{n}-u_{H,2}^{n-1}\|_{2}^{2}+\cfrac{\tau^{2}}{2}\|u_{H,2}^{n}\|_{a}^{2}+\cfrac{\tau^{2}}{2}\|u_{H,2}^{n-1}\|_{a}^{2}-\cfrac{\tau^{2}}{2}\|u_{H,2}^{n}-u_{H,2}^{n-1}\|_{a}^{2}\Big).
\end{align*}
\begin{align*}
C_{2} & =\cfrac{\tau^{2}}{2}a(u_{H,1}^{n+1}+u_{H,1}^{n-1},u_{H,2}^{n+1}-u_{H,2}^{n-1})  =(u_{H,1}^{n+1}+u_{H,1}^{n-1},c_{\tau}(u_{H,2}^{n+1}-u_{H,2}^{n-1}))_{m_{\tau}} = \\
 & (u_{H,1}^{n+1}+u_{H,1}^{n-1},c_{\tau}(u_{H,2}^{n+1}-u_{H,2}^{n-1}))+\cfrac{\tau^{2}}{2}a(u_{H,1}^{n+1}+u_{H,1}^{n-1},c_{\tau}(u_{H,2}^{n+1}-u_{H,2}^{n-1})) = \\
 & -B_{6}-B_{7}.
\end{align*}

Therefore, we have 
\begin{align*}
 & \|u_{H,1}^{n+1}\|_{2}^{2}+\|u_{H,1}^{n}\|_{2}^{2}+\|u_{H,2}^{n+1}-u_{H,2}^{n}\|_{2}^{2}+\cfrac{\tau^{2}}{2}\|u_{H,2}^{n+1}\|_{a}^{2}+\cfrac{\tau^{2}}{2}\|u_{H,2}^{n}\|_{a}^{2}-\cfrac{\tau^{2}}{2}\|u_{H,2}^{n+1}-u_{H,2}^{n}\|_{a}^{2} - \\
  & 2(b_{\tau}(u_{H,1}^{n+1})-c_{\tau}(u_{H,2}^{n+1}),b_{\tau}(u_{H,1}^{n})-c_{\tau}(u_{H,2}^{n}))_{m_{\tau}}=\\
 & \|u_{H,1}^{n}\|_{2}^{2}+\|u_{H,1}^{n-1}\|_{2}^{2}+\|u_{H,2}^{n}-u_{H,2}^{n-1}\|_{2}^{2}+\cfrac{\tau^{2}}{2}\|u_{H,2}^{n}\|_{a}^{2}+\cfrac{\tau^{2}}{2}\|u_{H,2}^{n-1}\|_{a}^{2}-\cfrac{\tau^{2}}{2}\|u_{H,2}^{n}-u_{H,2}^{n-1}\|_{a}^{2} - \\
 & 2(b_{\tau}(u_{H,1}^{n})-c_{\tau}(u_{H,2}^{n}),b_{\tau}(u_{H,1}^{n-1})-c_{\tau}(u_{H,2}^{n-1}))_{m_{\tau}}.
\end{align*}
We can observe that
\begin{align*}
 & -2(b_{\tau}(u_{H,1}^{n+1})-c_{\tau}(u_{H,2}^{n+1}),b_{\tau}(u_{H,1}^{n})-c_{\tau}(u_{H,2}^{n}))_{m_{\tau}}=\\
 & \|b_{\tau}(u_{H,1}^{n+1})-c_{\tau}(u_{H,2}^{n+1})-b_{\tau}(u_{H,1}^{n})+c_{\tau}(u_{H,2}^{n})\|_{m_{\tau}}^{2}-\|b_{\tau}(u_{H,1}^{n+1})-c_{\tau}(u_{H,2}^{n+1})\|_{m_{\tau}}^{2}-\|b_{\tau}(u_{H,1}^{n})-c_{\tau}(u_{H,2}^{n})\|_{m_{\tau}}^{2}
\end{align*}
Thus, we have 
\[
E^{n+\frac{1}{2}}=E^{n-\frac{1}{2}}.
\]
\end{proof}
Next, we will do some formal calculations to show that our proposed
energy is close to the continuous energy when $\tau$ is small.
We remark that for $\tau\rightarrow0$, we have 
\begin{align*}
b_{\tau}(u_{1}) & =u_{1}+O(\tau^{2})\;\text{in }L^{2}\\
c_{\tau}(u_{2}) & =O(\tau^{2})\;\text{in }L^{2},
\end{align*}
\begin{align*}
\|u_{1}\|_{s_{\tau}}^{2} & =(u_{1},u_{1})-\|b_{\tau}(u_{1})\|_{m_{\tau}}^{2} \approx \\
 & \cfrac{\tau^{2}}{2}\|u_{1}\|_{a}^{2}+O(\tau^{4}),
\end{align*}
\[
\|u\|_{m_{\tau}}^{2}\approx\|u\|_{}^{2}+O(\tau^{2}),
\]
\begin{align*}
\|u_{2}\|_{n_{\tau}}^{2} & =\cfrac{\tau^{2}}{2}\|u_{2}\|_{a}^{2}-\|c_{\tau}(u_{2})\|_{m_{\tau}}^{2}-\|d_{\tau}(u_{2})\|_{s_{\tau}}^{2} = \\
 & \cfrac{\tau^{2}}{2}\Big(\|u_{2}\|_{a}^{2}-\|d_{\tau}(u_{2})\|_{a}^{2}\Big)+O(\tau^{4}),
\end{align*}
\begin{align*}
E^{n+\frac{1}{2}}:= & \|b(u_{H,1}^{n+1})-c(u_{H,2}^{n+1})-b(u_{H,1}^{n})+c(u_{H,2}^{n})\|_{m_{\tau}}^{2}+\|u_{H,2}^{n+1}-u_{H,2}^{n}\|_{2}^{2}-\cfrac{\tau^{2}}{2}\|u_{H,2}^{n+1}-u_{H,2}^{n}\|_{a}^{2}+\\
 & \|u_{H,1}^{n+1}+d_{\tau}(u_{H,2}^{n+1})\|_{s_{\tau}}^{2}+\|u_{H,2}^{n+1}\|_{n_{\tau}}^{2}+\|u_{H,1}^{n}+d_{\tau}(u_{H,2}^{n})\|_{s_{\tau}}^{2}+\|u_{H,2}^{n}\|_{n_{\tau}}^{2}.
\end{align*}

If $\|u_{\alpha}^{n+1}-u_{\alpha}^{n}\|_{a}=\|u_{\alpha}^{n+1}-u_{\alpha}^{n}\|_{}=O(\tau^{2})$,
we have 
\begin{align*}
E^{n+\frac{1}{2}}= & \|u_{H,1}^{n+1}-u_{H,1}^{n}\|_{}^{2}+\|u_{H,2}^{n+1}-u_{H,2}^{n}\|_{}^{2}+\cfrac{\tau^{2}}{2}\|u_{H,1}^{n+1}+d_{\tau}(u_{H,2}^{n+1})\|_{a}^{2}+\cfrac{\tau^{2}}{2}\Big(\|u_{H,2}^{n+1}\|_{a}^{2}-\|d_{\tau}(u_{H,2}^{n+1})\|_{a}^{2}\Big) + \\
 & \cfrac{\tau^{2}}{2}\|u_{H,1}^{n}+d_{\tau}(u_{H,2}^{n})\|_{a}^{2}+\cfrac{\tau^{2}}{2}\Big(\|u_{H,2}^{n}\|_{a}^{2}-\|d_{\tau}(u_{H,2}^{n})\|_{a}^{2}\Big)+O(\tau^{4}).
\end{align*}

Since
\begin{align*}
a(u_{H,1}^{n+1}+u_{H,2}^{n+1},u_{H,1}^{n+1}+u_{H,2}^{n+1}) & =a(u_{H,1}^{n+1},u_{H,1}^{n+1})+2a(d_{\tau}(u_{H,2}^{n+1}),u_{H,1}^{n+1})+\|u_{H,2}^{n+1}\|_{a}^{2} = \\
 & \|u_{H,1}^{n+1}+d_{\tau}(u_{H,2}^{n+1})\|_{a}^{2}+\Big(\|u_{H,2}^{n+1}\|_{a}^{2}-\|d_{\tau}(u_{H,2}^{n+1})\|_{a}^{2}\Big),
\end{align*}
we have 
\begin{align*}
E^{n+\frac{1}{2}}= & \|u_{H,1}^{n+1}-u_{H,1}^{n}\|_{}^{2}+\|u_{H,2}^{n+1}-u_{H,2}^{n}\|_{}^{2}+\cfrac{\tau^{2}}{2}\|u_{H,1}^{n+1}+u_{H,2}^{n+1}\|_{a}^{2}+\cfrac{\tau^{2}}{2}\Big(\|u_{H,1}^{n}+u_{H,2}^{n}\|_{a}^{2}\Big)+O(\tau^{4})
\end{align*}
and 
\[
\tau^{-2}E^{n+\frac{1}{2}}\approx\|\frac{u_{H}^{n+1}-u_{H}^{n}}{\tau}\|_{}^{2}+\cfrac{\tau^{2}}{2}\Big(\|u_{H}^{n+1}\|_{a}^{2}+\|u_{H}^{n}\|_{a}^{2}\Big).
\]

The Appendix shows that our proposed approach can also be used 
in splitting method
with $\omega=0$.

\bibliographystyle{abbrv}
\bibliography{references,references4,references1,references2,references3,decSol}

\begin{thebibliography}{10}

\bibitem{abdulle2012explicit}
A.~Abdulle.
\newblock Explicit methods for stiff stochastic differential equations.
\newblock In {\em Numerical Analysis of Multiscale Computations}, pages 1--22.
  Springer, 2012.

\bibitem{ab05}
G.~Allaire and R.~Brizzi.
\newblock A multiscale finite element method for numerical homogenization.
\newblock {\em SIAM J. Multiscale Modeling and Simulation}, 4(3):790--812,
  2005.

\bibitem{ariel2009multiscale}
G.~Ariel, B.~Engquist, and R.~Tsai.
\newblock A multiscale method for highly oscillatory ordinary differential
  equations with resonance.
\newblock {\em Mathematics of Computation}, 78(266):929--956, 2009.

\bibitem{ascher1997implicit}
U.~M. Ascher, S.~J. Ruuth, and R.~J. Spiteri.
\newblock Implicit-explicit {R}unge-{K}utta methods for time-dependent partial
  differential equations.
\newblock {\em Applied Numerical Mathematics}, 25(2-3):151--167, 1997.

\bibitem{cheung2020explicit}
S.~W. Cheung, E.~T. Chung, Y.~Efendiev, and W.~T. Leung.
\newblock Explicit and energy-conserving constraint energy minimizing
  generalized multiscale discontinuous galerkin method for wave propagation in
  heterogeneous media.
\newblock {\em arXiv preprint arXiv:2009.00991}, 2020.

\bibitem{chung2016adaptiveJCP}
E.~T. Chung, Y.~Efendiev, and T.~Hou.
\newblock Adaptive multiscale model reduction with generalized multiscale
  finite element methods.
\newblock {\em Journal of Computational Physics}, 320:69--95, 2016.

\bibitem{MixedGMsFEM}
E.~T. Chung, Y.~Efendiev, and C.~Lee.
\newblock Mixed generalized multiscale finite element methods and applications.
\newblock {\em SIAM Multiscale Model. Simul.}, 13:338--366, 2014.

\bibitem{WaveGMsFEM}
E.~T. Chung, Y.~Efendiev, and W.~T. Leung.
\newblock Generalized multiscale finite element methods for wave propagation in
  heterogeneous media.
\newblock {\em Multiscale Modeling \& Simulation}, 12(4):1691--1721, 2014.

\bibitem{chung2015residual}
E.~T. Chung, Y.~Efendiev, and W.~T. Leung.
\newblock Residual-driven online generalized multiscale finite element methods.
\newblock {\em Journal of Computational Physics}, 302:176--190, 2015.

\bibitem{chung2018constraint}
E.~T. Chung, Y.~Efendiev, and W.~T. Leung.
\newblock Constraint energy minimizing generalized multiscale finite element
  method.
\newblock {\em Computer Methods in Applied Mechanics and Engineering},
  339:298--319, 2018.

\bibitem{chung2018constraintmixed}
E.~T. Chung, Y.~Efendiev, and W.~T. Leung.
\newblock Constraint energy minimizing generalized multiscale finite element
  method in the mixed formulation.
\newblock {\em Computational Geosciences}, 22(3):677--693, 2018.

\bibitem{chung2018fast}
E.~T. Chung, Y.~Efendiev, and W.~T. Leung.
\newblock Fast online generalized multiscale finite element method using
  constraint energy minimization.
\newblock {\em Journal of Computational Physics}, 355:450--463, 2018.

\bibitem{NLMC}
E.~T. Chung, Y.~Efendiev, W.~T. Leung, M.~Vasilyeva, and Y.~Wang.
\newblock Non-local multi-continua upscaling for flows in heterogeneous
  fractured media.
\newblock {\em Journal of Computational Physics}, 372:22--34, 2018.

\bibitem{chung2015staggered}
E.~T. Chung, C.~Y. Lam, and J.~Qian.
\newblock A staggered discontinuous {G}alerkin method for the simulation of
  seismic waves with surface topography.
\newblock {\em Geophysics}, 80(4):T119--T135, 2015.

\bibitem{cohen2001higher}
G.~Cohen, P.~Joly, J.~E. Roberts, and N.~Tordjman.
\newblock Higher order triangular finite elements with mass lumping for the
  wave equation.
\newblock {\em SIAM Journal on Numerical Analysis}, 38(6):2047--2078, 2001.

\bibitem{dur91}
L.~Durlofsky.
\newblock Numerical calculation of equivalent grid block permeability tensors
  for heterogeneous porous media.
\newblock {\em Water Resour. Res.}, 27:699--708, 1991.

\bibitem{ee03}
W.~E and B.~Engquist.
\newblock Heterogeneous multiscale methods.
\newblock {\em Comm. Math. Sci.}, 1(1):87--132, 2003.

\bibitem{GMsFEM13}
Y.~Efendiev, J.~Galvis, and T.~Hou.
\newblock Generalized multiscale finite element methods ({GM}s{FEM}).
\newblock {\em Journal of Computational Physics}, 251:116--135, 2013.

\bibitem{eh09}
Y.~Efendiev and T.~Hou.
\newblock {\em {Multiscale Finite Element Methods: Theory and Applications}},
  volume~4 of {\em Surveys and Tutorials in the Applied Mathematical Sciences}.
\newblock Springer, New York, 2009.

\bibitem{efendiev_split2020}
Y.~Efendiev, S.~Pun, and P.~N. Vabishchevich.
\newblock Temporal splitting algorithms for non-stationary multiscale problems.
\newblock {\em arXiv preprint}, 2020.

\bibitem{efendiev2020splitting}
Y.~Efendiev and P.~N. Vabishchevich.
\newblock Splitting methods for solution decomposition in nonstationary
  problems.
\newblock {\em arXiv preprint arXiv:2008.08111}, 2020.

\bibitem{engquist2005heterogeneous}
B.~Engquist and Y.-H. Tsai.
\newblock Heterogeneous multiscale methods for stiff ordinary differential
  equations.
\newblock {\em Mathematics of computation}, 74(252):1707--1742, 2005.

\bibitem{chung_partial_expliict21}
W.~T.~L. Eric~Chung, Yalchin~Efendiev and P.~N. Vabishchevich.
\newblock Contrast-independent partially explicit time discretizations for
  multiscale flow problems.
\newblock arXiv:2101.04863.

\bibitem{henning2012localized}
P.~Henning, A.~M{\aa}lqvist, and D.~Peterseim.
\newblock A localized orthogonal decomposition method for semi-linear elliptic
  problems.
\newblock {\em ESAIM: Mathematical Modelling and Numerical Analysis},
  48(5):1331--1349, 2014.

\bibitem{hw97}
T.~Hou and X.~Wu.
\newblock A multiscale finite element method for elliptic problems in composite
  materials and porous media.
\newblock {\em J. Comput. Phys.}, 134:169--189, 1997.

\bibitem{hou2018adaptive}
T.~Y. Hou, D.~Huang, K.~C. Lam, and P.~Zhang.
\newblock An adaptive fast solver for a general class of positive definite
  matrices via energy decomposition.
\newblock {\em Multiscale Modeling \& Simulation}, 16(2):615--678, 2018.

\bibitem{hou2017exploring}
T.~Y. Hou, Q.~Li, and P.~Zhang.
\newblock Exploring the locally low dimensional structure in solving random
  elliptic pdes.
\newblock {\em Multiscale Modeling \& Simulation}, 15(2):661--695, 2017.

\bibitem{hou2019model}
T.~Y. Hou, D.~Ma, and Z.~Zhang.
\newblock A model reduction method for multiscale elliptic pdes with random
  coefficients using an optimization approach.
\newblock {\em Multiscale Modeling \& Simulation}, 17(2):826--853, 2019.

\bibitem{jennylt03}
P.~Jenny, S.~Lee, and H.~Tchelepi.
\newblock Multi-scale finite volume method for elliptic problems in subsurface
  flow simulation.
\newblock {\em J. Comput. Phys.}, 187:47--67, 2003.

\bibitem{li2008effectiveness}
T.~Li, A.~Abdulle, et~al.
\newblock Effectiveness of implicit methods for stiff stochastic differential
  equations.
\newblock In {\em Commun. Comput. Phys}. Citeseer, 2008.

\bibitem{marchuk1990splitting}
G.~I. Marchuk.
\newblock Splitting and alternating direction methods.
\newblock {\em Handbook of numerical analysis}, 1:197--462, 1990.

\bibitem{oz06_1}
H.~Owhadi and L.~Zhang.
\newblock Metric-based upscaling.
\newblock {\em Comm. Pure. Appl. Math.}, 60:675--723, 2007.

\bibitem{rk07}
A.~Roberts and I.~Kevrekidis.
\newblock General tooth boundary conditions for equation free modeling.
\newblock {\em SIAM J. Sci. Comput.}, 29(4):1495--1510, 2007.

\bibitem{skr06}
G.~Samaey, I.~Kevrekidis, and D.~Roose.
\newblock Patch dynamics with buffers for homogenization problems.
\newblock {\em J. Comput. Phys.}, 213(1):264--287, 2006.

\bibitem{SamarskiiTheory}
A.~A. Samarskii.
\newblock {\em The Theory of Difference Schemes}.
\newblock Marcel Dekker, New York, 2001.

\bibitem{SamarskiiMatusVabischevich2002}
A.~A. Samarskii, P.~P. Matus, and P.~N. Vabishchevich.
\newblock {\em Difference Schemes with Operator Factors}.
\newblock Kluwer Academic Pub, 2002.

\bibitem{vab21}
P.~N. Vabishchevich.
\newblock Splitting methods for dynamic second order equations.
\newblock submitted.

\bibitem{VabishchevichAdditive}
P.~N. Vabishchevich.
\newblock {\em Additive Operator-Difference Schemes: Splitting Schemes}.
\newblock Walter de Gruyter GmbH, Berlin, Boston, 2013.

\bibitem{virieux1984sh}
J.~Virieux.
\newblock Sh-wave propagation in heterogeneous media: Velocity-stress
  finite-difference method.
\newblock {\em Geophysics}, 49(11):1933--1942, 1984.

\bibitem{weh02}
X.~Wu, Y.~Efendiev, and T.~Hou.
\newblock Analysis of upscaling absolute permeability.
\newblock {\em Discrete and Continuous Dynamical Systems, Series B.},
  2:158--204, 2002.

\end{thebibliography}

\end{document}